\documentclass[12pt]{article}

\usepackage[T1]{fontenc}
\usepackage{lmodern}
\usepackage[utf8]{inputenc}
\usepackage{amsmath,amssymb,amsthm,mathtools}
\usepackage{enumitem,array}
\usepackage{microtype}
\usepackage{needspace}
\usepackage{xurl}
\usepackage[hidelinks]{hyperref}
\makeatletter
\newif\if@firstauthorentry
\@firstauthorentrytrue
\gdef\@authorentries{}
\renewcommand{\author}[1]{%
  \if@firstauthorentry
    \global\@firstauthorentryfalse
  \else
    \g@addto@macro\@authorentries{\par\vspace{1.25ex}}%
  \fi
  \g@addto@macro\@authorentries{{\large #1\par}}%
}
\newcommand{\address}[1]{%
  \g@addto@macro\@authorentries{{\small #1\par}}%
}
\newcommand{\email}[1]{%
  \g@addto@macro\@authorentries{{\small\ttfamily
    \begingroup\def\@{@}#1\endgroup\par}}%
}
\gdef\@author{%
  \begin{minipage}{0.92\textwidth}
    \centering\@authorentries
  \end{minipage}%
}
\makeatother

\newtheorem{theorem}{Theorem}[section]
\newtheorem{proposition}[theorem]{Proposition}
\newtheorem{lemma}[theorem]{Lemma}
\newtheorem{corollary}[theorem]{Corollary}
\newtheorem{remark}[theorem]{Remark}
\newtheorem*{introtheorema}{Theorem A}
\newtheorem*{introtheoremb}{Theorem B}
\newtheorem*{introtheoremc}{Theorem C}
\newtheorem*{introreconstruction}{Reconstruction principle}
\theoremstyle{definition}
\newtheorem{definition}[theorem]{Definition}
\newtheorem{example}[theorem]{Example}

\numberwithin{equation}{section}

\newcommand{\Mid}{\operatorname{Mid}}
\newcommand{\gyr}{\operatorname{gyr}}
\newcommand{\id}{\operatorname{id}}
\newcommand{\sa}{\mathrm{sa}}

\newcommand{\Om}{\Omega}

\newcommand{\dT}{d_T}
\newcommand{\dH}{d_H^{\mathrm p}}

\newcommand{\poplus}{\mathbin{\oplus^{\mathrm p}}}
\newcommand{\potimes}{\mathbin{\otimes^{\mathrm p}}}
\newcommand{\prho}{\widehat{\rho}}

\title{Local-to-Global Isometries on Bounded Symmetric Domains and JB-Cones via Generalized Gyrovector Spaces}
\author{Osamu Hatori}
\address{Niigata University,
Niigata 950-2181, Japan}
\email{oppekepenguin\@gmail.com}

\hypersetup{
  pdfauthor={Osamu Hatori},
  pdftitle={Local-to-Global Isometries on Bounded Symmetric Domains and JB-Cones via Generalized Gyrovector Spaces}
}
\date{}

\begin{document}
% Make room for the abstract and keywords above the first-page footer.
\AddToHookNext{env/center/before}{\vspace{-3\baselineskip}}
\maketitle
\vspace{-1.2\baselineskip}

\begin{abstract}
\setlength{\baselineskip}{12pt}
Every bijective isometry between nonempty connected open subsets of
simply connected generalized gyrovector spaces (GGVs) with real
value-line structure extends uniquely to a global bijective isometry.
Gyromidpoints and point reflections give continuation along paths.
The normalized extension preserves gyroaddition and scalar multiplication.
In general, extension is equivalent to monodromy invariance of the germs
of the map and its inverse.

On bounded symmetric domains of arbitrary dimension, a bijection between
such open subsets preserving the ambient Kobayashi distances extends
uniquely to $z\mapsto g_b(Tz)$ in JB$^*$-triple ball coordinates.
Here $g_b$ is a transvection and $T$ a surjective real-linear triple
isomorphism. Thus the restricted distance determines the underlying real
triple. We realize the Kobayashi distance as a GGV gyrometric on the
Friedman--Ungar gyrogroup and recover the triple product from its
transvections. The finite-dimensional global formula follows from
Lemmens and Walsh.

The positive invertible cone of a nonzero unital JB-algebra $A$ and its
projectivization ($\dim A>1$ in the latter case) carry standard-norm
GGVs with Thompson's and Hilbert's metrics at $t=1$, respectively.
For every $t>0$, bijective isometries between nonempty connected open
subsets of the corresponding cones extend uniquely to global bijective
isometries. The local Thompson metric determines $A$ up to Jordan
isomorphism. Infinitesimal gyrations yield both global classifications.
In the projective case a common sign and a square-preserving lift
replace a Dye-type extension theorem.

\end{abstract}

\medskip
\noindent\textbf{2020 Mathematics Subject Classification.}
Primary 20N05, 46L70; Secondary 46B04, 51M10.

\smallskip
\noindent\textbf{Keywords.}\enspace\ignorespaces
Generalized gyrovector space, bounded symmetric domain, Kobayashi distance,
JB$^*$-triple, JB-algebra, Thompson metric, Hilbert metric, monodromy, local isometry.

\clearpage
\tableofcontents
\clearpage

\Needspace{8\baselineskip}
\section{Introduction}

How much of a complex JB$^*$-triple can be recovered from the
Kobayashi distance on a nonempty connected open subset of its unit
ball? Let $D$ be the open unit ball of a nonzero complex JB$^*$-triple
$E$, let $k_D$ be its Kobayashi distance, and let $U\subseteq D$ be
such an open subset. The metric space $(U,k_D|_{U\times U})$ determines
the underlying real JB$^*$-triple $E_{\mathbb R}$ up to real-linear
triple isomorphism (Theorem~B). Here $E_{\mathbb R}$ has the same norm
and triple product as $E$, with scalars restricted to $\mathbb R$.
The open set may be arbitrarily small and need not contain $0$.
For the positive invertible cone of a nonzero unital JB-algebra, the
restriction of the Thompson metric to such an open set determines the
algebra up to Jordan isomorphism (Corollary~\ref{cor:local-Jordan-both}).

Both reconstruction results begin with the same extension problem.
Given a bijective isometry between open subsets for the restricted
ambient distances, does it determine an isometry of the ambient
spaces? We first prove an extension theorem for the generalized
gyrovector spaces (GGVs) introduced by Abe and Hatori \cite{AH2015}.
These spaces have gyroaddition, scalar multiplication and a compatible
gyrometric. We assume \emph{real value-line structure}: addition and
scalar multiplication on the norm-value line are the ordinary real
operations. The gyrometric is then a real-valued metric, and its
gyrosegments have constant speed.

\begin{introtheorema}[Theorem~\ref{thm:general-local-global}]
For $i=1,2$, let $G_i$ be a simply connected GGV with real value-line
structure, identity $e_i$, and gyrometric $\rho_i$. Let $U_i\subset G_i$
be a nonempty open set and assume that $U_1$ is connected. Every bijective isometry
\[
 T:(U_1,\rho_1)\longrightarrow(U_2,\rho_2)
\]
extends uniquely to a bijective isometry
\[
 \widetilde T:(G_1,\rho_1)\longrightarrow(G_2,\rho_2).
\]
Moreover, there are a unique $b\in G_2$ and a unique bijective gyrometric
isometry $S:G_1\to G_2$ preserving the identity, gyroaddition, and scalar
multiplication such that
\[
 \widetilde T(x)=b\oplus_2S(x),\qquad x\in G_1.
\]
Here $b=\widetilde{T}(e_1)$.
\end{introtheorema}

Preservation of gyroaddition and scalar multiplication by $S$ is a
conclusion of the theorem. A GGV with real value-line structure
is \emph{standard-norm} if $x\mapsto r\otimes x$ is continuous for
each fixed $r\in\mathbb R$. The contraction $x\mapsto (1-s)\otimes x$
is jointly continuous in $x,s$ by
Proposition~\ref{prop:joint-scalar-continuity}, giving contractibility
(Corollary~\ref{cor:contractible}) and hence applicability of Theorem~A.

To continue the map, we use gyromidpoints to show that it intertwines
point reflections. Conjugation by two reflections moves the center of
an isometric representative while preserving its radius
(Lemma~\ref{lem:reflection-step}). This permits continuation along a
path without any completeness or differentiability assumption.
Theorems~\ref{thm:monodromy-action} and~\ref{thm:two-sided-monodromy}
give the monodromy criterion for the map and its inverse. In the
simply connected case continuation is independent of the path.
Subdivision of gyrosegments gives global distance preservation, and the
GGV Mazur--Ulam theorem \cite[Corollary~4.5]{AHrevisited} gives
preservation of the operations after normalization.

By Kaup's Riemann mapping theorem \cite{Kaup1983}, a bounded symmetric
domain in a complex Banach space can be realized as the open unit ball
of a complex JB$^*$-triple. In this realization, the gyroaddition of
Friedman and Ungar \cite{FriedmanUngar1994} is $a\oplus z=g_a(z)$,
where $g_a$ is the canonical transvection taking $0$ to $a$.
The scalar multiplication in Section~\ref{sec:bsd-kobayashi} is
defined in the spectral coordinate $L(a)=\operatorname{arctanh}a$.

\Needspace{7\baselineskip}
\begin{introtheoremb}[Corollary~\ref{cor:bsd-extension}]
For $i=1,2$, let $D_i$ be the open unit ball of a nonzero complex
JB$^*$-triple $E_i$ and $U_i\subset D_i$ a nonempty connected open
subset. Every bijection $f:U_1\to U_2$ preserving the restrictions of
the ambient Kobayashi distances extends uniquely to a bijective
Kobayashi isometry $\widetilde f:D_1\to D_2$. In fact,
\[
  \widetilde f(z)=g_b(Tz)\qquad(z\in D_1)
\]
for a unique $b\in D_2$ and a unique surjective real-linear triple
isomorphism $T:E_1\to E_2$. In particular, the restricted Kobayashi
distance on $U_i$ determines $(E_i)_{\mathbb R}$ up to real-linear
triple isomorphism.
If at least one of $D_1,D_2$ is irreducible, then $\widetilde f$ is
biholomorphic or antibiholomorphic.
\end{introtheoremb}

The local question in Theorem~B is whether an isometry on an arbitrarily
small connected open set determines an isometry of the entire domain.
In finite dimensions, the classification of global bijective isometries,
including the reducible case, follows from Lemmens and Walsh
\cite[Theorems~1.1--1.3 and~1.5]{LemmensWalsh2026}.
Their proof uses the horofunction compactification and the Gromov
product on its boundary. The closure of $U_1$ may be compact in $D_1$,
so $f$ need supply no boundary data from which to obtain an extension.
Cartan's theorem for locally symmetric Riemannian spaces starts with a
curvature-preserving tangent isometry, a simply connected source and a
complete target \cite[Theorem~B.5]{Ballmann2006}. Klotz's integration
theorem starts with a morphism of tangent Lie triple systems and a
simply connected source \cite[Theorem~5.20]{Klotz2011}.
Either approach would first require a suitable tangent map. The
infinitesimal Kobayashi norm can already be the maximum norm, as at
the origin of the bidisk. In finite dimensions, homogeneity makes the
Binet--Legendre metric associated with the infinitesimal Kobayashi metric
smooth, so the theorem of Matveev and Troyanov yields regularity of the
local isometry \cite[Theorem~A]{MatveevTroyanov}. Our GGV argument
constructs the extension in arbitrary dimension without this preliminary
regularity step.

Proposition~\ref{prop:bsd-ggv} realizes the ambient Kobayashi distance
as the gyrometric of a standard-norm GGV. The balls $D_i$ are
contractible, so Theorem~A applies directly. Gyromidpoints and point
reflections give the unique extension even when neither $U_i$ contains
the origin. This Mankiewicz-type extension uses only the metric and
the GGV operations in arbitrary dimension. It requires no prior
regularity of $f$, Riemannian metric or metric completeness.

The GGV operations also identify the extension. Its normalization $S$
intertwines the transvections: $Sg_a=g_{S(a)}S$. If $F_i$ and $L_i$
are the spectral $\tanh$ and $\operatorname{arctanh}$ coordinates,
scalar homogeneity and the tangent norm make $T=L_2SF_1$ a surjective
real-linear norm isometry. The cube-preservation theorem of Isidro,
Kaup and Rodr\'iguez \cite[Theorem~4.8]{IKR1995} and odd functional
calculus give $TF_1=F_2T$, hence $S=T|_{D_1}$. We may now
differentiate the canonical one-parameter transvections to obtain
$T\{z,u,z\}=\{Tz,Tu,Tz\}$. Polarization in the two outer variables
recovers the full triple product. The example on the ball of $c_0$ at
the end of Section~\ref{sec:bsd-kobayashi} illustrates the role of
irreducibility in the holomorphic or antiholomorphic conclusion.

We next ask what the local metric determines on the positive invertible
cone $A_+^\circ$ of a nonzero unital JB-algebra $A$ and on its space of
positive rays $\mathbb P A_+^\circ$. Power coordinates give a family
of standard-norm GGV structures on each cone, with $\dim A>1$ in the
projective case. At $t=1$ the gyrometrics are Thompson's metric on
$A_+^\circ$ and Hilbert's metric on $\mathbb P A_+^\circ$. Theorem~A
again supplies the extension. To identify it we recover the Jordan
product using the infinitesimal gyrations of the cone. The symbol
$U_b$ denotes the quadratic representation recalled in
Section~\ref{sec:JB-preliminaries}.

\Needspace{8\baselineskip}
\begin{introtheoremc}[Corollaries~\ref{cor:Thompson-local} and~\ref{cor:Hilbert-local}]
Let $A$ and $B$ be nonzero unital JB-algebras with units $e_A$ and $e_B$.
\begin{enumerate}[label=\textup{(\roman*)}]
\item Every bijective Thompson isometry $T$ between nonempty connected open
subsets of $A_+^\circ$ and $B_+^\circ$ extends uniquely to a bijective
Thompson isometry $\widetilde T$ from $A_+^\circ$ onto $B_+^\circ$, given by
\[
 \widetilde T(a)=U_{c^{1/2}}
 \bigl(pJ(a)+p^\perp J(a^{-1})\bigr)
 \qquad(a\in A_+^\circ).
\]
Here $c\in B_+^\circ$, $J:A\to B$ is a Jordan isomorphism,
$p\in B$ is a central projection, and $p^\perp=e_B-p$.
The data $c,J,p$ are unique, and $c=\widetilde T(e_A)$.

\item If both algebras have dimension at least two, every bijective Hilbert
isometry $T$ between nonempty connected open subsets of
$\mathbb P A_+^\circ$ and $\mathbb P B_+^\circ$ extends uniquely to a
bijective Hilbert isometry
$\widetilde T:\mathbb P A_+^\circ\to\mathbb P B_+^\circ$ of the form
\[
 \widetilde T(\bar a)=\overline{U_bJ(a^\varepsilon)}
 \qquad(a\in A_+^\circ),
\]
where $b\in B_+^\circ$, $J:A\to B$ is a Jordan isomorphism, and
$\varepsilon\in\{-1,1\}$. The ray $\bar b$ and, for each admissible sign,
the Jordan isomorphism $J$ are unique. The sign ambiguity is described
in Theorem~\ref{thm:projective-global-classification}.
\end{enumerate}
\end{introtheoremc}

The global Thompson classification in Theorem~C is due to Lemmens,
Roelands and Wortel \cite[Theorem~3.2]{LRW}, and the global Hilbert
classification for unital JB-algebras is due to Roelands and Wortel
\cite[Theorem~5.1]{RW}. We recover both formulas from the operations
preserved by the normalized isometry. Theorems~\ref{thm:global-classification}
and~\ref{thm:projective-global-classification} give their forms for every
$t>0$ by power conjugacy (Propositions~\ref{prop:power-conjugacy}
and~\ref{prop:projective-power-conjugacy}).

The algebraic step can be stated independently of the isometries that
produce it. Write $[A]=A/\mathbb Re_A$ with variation norm
$\|[x]\|_v=\operatorname{diam}\sigma(x)$, and put
$D_{x,y}=[L_x,L_y]$, where $L_xy=x\circ y$.
\Needspace{9\baselineskip}
\begin{introreconstruction}[Proposition~\ref{prop:projective-Jordan-reconstruction}]
Let $A,B$ be unital JB-algebras of dimension at least two. A surjective
real linear variation-norm isometry $S:[A]\to[B]$ satisfying
\[
 S([D_{x,y}z])=[D_{u,v}w],\qquad
 [u]=S[x],\ [v]=S[y],\ [w]=S[z],
\]
for all $x,y,z\in A$ and any representatives $u,v,w\in B$ of the
indicated classes has the form $S[x]=\varepsilon[Jx]$ for a common sign
$\varepsilon\in\{-1,1\}$ and a Jordan isomorphism $J:A\to B$.
\end{introreconstruction}
The classification conclusion is known even without the additional
Lie triple identity \cite[Theorem~4.4]{RW}. Here that identity
is supplied by GGV operation preservation and is used to construct
the lift without invoking that classification.

For both the positive and the projective cones, normalization at the
identity and the GGV Mazur--Ulam theorem
\cite[Corollary~4.5]{AHrevisited} give preservation of
gyroaddition and scalar multiplication. On the positive cone,
logarithmic linearization gives a linear norm isometry
$S=\log\circ g\circ\exp:A\to B$. Preservation of infinitesimal
gyrations makes $S(e_A)$ a central symmetry. Multiplication by this
symmetry gives a unital linear norm isometry; the order-interval and
bidual argument for square preservation reproves the unital result of
Wright and Youngson \cite[Theorem~4]{WrightYoungson1978}.

On the projective cone, logarithmic linearization gives a linear
variation-norm isometry of the scalar quotients. Roelands and Wortel
classify all such linear isometries \cite[Theorem~4.4]{RW}.
Our reconstruction uses the additional preservation of the quotient Lie
triple operation, obtained from infinitesimal gyrations. This recovers
Peirce $1/2$ components and operator commutativity of projections.
The projection-chain conclusion of Lemma~\ref{lem:Peirce-projection-chain}
was established in \cite[Lemma~4.10]{LRW}; its proof here uses Peirce symmetries.
These chains make the orientations of all three-projection partitions
consistent, yielding a sign $\varepsilon$ common to all projections in
the source bidual. Proposition~\ref{prop:projective-Jordan-reconstruction}
then constructs a centered lift whose square preservation follows from
spectral approximation. Lemma~\ref{lem:square-preserving-lift} derives
additivity from quotient linearity and scalar-translation covariance.
For Euclidean Jordan algebras, \cite[Remark~4.20]{LRW} already gives a
spectral lift whose linearity follows using trace-zero subspaces.
Our centered lift works on the full biduals without a trace and
establishes square preservation before additivity.
Restriction gives the Jordan isomorphism on the original algebras,
without a Dye-type extension theorem or a separate treatment of the
type $I_2$ summand. Thus the reconstruction exploits extra structure
available for GGV isometries, rather than strengthening the
classification of all linear variation-norm isometries.

Part~(i) of Theorem~C extends the $t=1$ case of Hatori's theorem for
unital $C^*$-algebras \cite[Theorem~8]{HatoriExtension} to unital
JB-algebras. Corollary~\ref{cor:JB-local-classification} gives the
corresponding local classification and extension for every $t>0$.
In particular, the local metric geometry of either cone determines the
Jordan isomorphism class under the hypotheses of
Corollary~\ref{cor:local-Jordan-both}.

The midpoint argument in Theorem~A has its antecedent in Mankiewicz's
proof of affine extension for isometries between connected open subsets
of normed spaces \cite[Lemma~3 and Theorem~2]{Mankiewicz}. In the setting of sprays,
Kim and Lawson prove smoothness for continuous maps preserving local
midpoints \cite{KimLawson2015}.

\begin{center}
\small
\textbf{Examples of standard-norm generalized gyrovector spaces}\par
\smallskip
\begingroup
\setlength{\tabcolsep}{3pt}
\renewcommand{\arraystretch}{1.17}
\begin{tabular}{@{}>{\raggedright\arraybackslash}p{.39\linewidth}
  >{\raggedright\arraybackslash}p{.30\linewidth}
  >{\raggedright\arraybackslash}p{.25\linewidth}@{}}
\hline
\textbf{Space} & \textbf{Gyrometric} & \textbf{Construction} \\ \hline
Nonzero real normed space & Norm distance & Ex.~\ref{ex:normed-space} \\
Pointed nontrivial bounded symmetric domain &
Kobayashi distance & Prop.~\ref{prop:bsd-ggv}, using \cite{FriedmanUngar1994} \\
Positive invertible cone of a nonzero unital JB-algebra &
Thompson distance at $t=1$ & Thm.~\ref{thm:JB-GGV} \\
Projective JB-cone ($\dim A>1$) and its normalized sections &
Hilbert distance at $t=1$ & Thm.~\ref{thm:projective-JB-GGV};
Cor.~\ref{cor:density-local} \\
Operator exponential sets and positive definite matrices &
Logarithmic distances from unitarily invariant norms &
Exs.~\ref{ex:operator-exponential-GGV}, \ref{ex:affine-riemannian-GGV} \\
Finite products; countable $\ell^p$-products ($1\le p<\infty$) &
$\ell^p$-norm of the factor distances &
Prop.~\ref{prop:lp-products} \\
Bounded maps $\Lambda\to G$ ($\Lambda\ne\varnothing$, $G$ standard-norm) &
Supremum distance. For infinite $\Lambda$, each scalar map must be
uniformly continuous on bounded sets. &
Prop.~\ref{prop:bounded-maps-continuity} \\ \hline
\end{tabular}
\endgroup
\end{center}
In the product row the factors are standard-norm, and finite products
also allow $p=\infty$.

Section~\ref{sec:continuation} proves Theorem~A, including the
two-sided monodromy criterion, and recovers Mankiewicz's theorem for
normed spaces. Section~\ref{sec:bsd-kobayashi} proves Theorem~B by
realizing the Kobayashi distance as a gyrometric and recovering the
triple product from transvections. The proof of Theorem~C occupies
Sections~\ref{sec:cone-linearization}--\ref{sec:Thompson-geometry}.
Section~\ref{sec:cone-linearization} constructs the two cone geometries
and derives their infinitesimal gyrations.
Section~\ref{sec:reconstruction} reconstructs the Jordan structure.
Section~\ref{sec:Thompson-geometry} combines this reconstruction with
Theorem~A to classify global and local cone isometries and treats the
trace-one section of full-rank density matrices
(Corollary~\ref{cor:density-local}).

Appendix~\ref{app:GGV-constructions} treats scalar continuity,
standardization, products and bounded mapping spaces.
Appendix~\ref{app:cone-details} contains the functional-calculus and
GGV-axiom verifications for the cones, together with operator and
matrix examples. Appendix~\ref{app:cone-lifts} gives cone-level lifts,
raywise surjectivity criteria and the parameters for density-matrix
isometries.

\Needspace{8\baselineskip}
\section{Metric continuation and local-to-global extension}\label{sec:continuation}

We prove Theorem~A by a metric continuation argument. After the metric
and continuity preliminaries, local midpoint preservation yields
covariance of point reflections and an identity principle. The reflection
step retains the radius of each ball representative, so continuation
exists along every path without a metric completeness assumption.
Small-cell consistency gives homotopy invariance. Applying continuation
to both a map and its inverse yields mutually inverse global local
isometries, and subdivision of gyrosegments gives global distance
preservation. The standard-norm continuity condition supplies simple
connectedness through a contraction.

\Needspace{8\baselineskip}
\subsection{Generalized gyrovector spaces and real value-line structure}

We recall the definition from \cite{AHrevisited}.  Let
$(G,\oplus)$ be a gyrocommutative gyrogroup with identity $e$, and let
$\otimes:\mathbb R\times G\to G$.  Let $\phi$ be an injection from $G$ into a
real normed space $(V,\|\cdot\|)$.
We write $\ominus a$ for the gyroinverse of $a$ and
$\gyr[u,v]$ for the gyration determined by $u,v$, an automorphism of $G$.

\Needspace{4\baselineskip}
\begin{definition}\label{def:GGV}
The quadruple $(G,\oplus,\otimes,\phi)$ is a \emph{generalized gyrovector
space}, or GGV, if the following conditions hold for
$u,v,a,b\in G$ and $r,r_1,r_2\in\mathbb R$, subject to the indicated
restrictions:
\begin{align*}
\text{(GGV0)}\quad&
 \|\phi(\gyr[u,v]a)\|=\|\phi(a)\|;\\
\text{(GGV1)}\quad& 1\otimes a=a;\\
\text{(GGV2)}\quad&
 (r_1+r_2)\otimes a=(r_1\otimes a)\oplus(r_2\otimes a);\\
\text{(GGV3)}\quad&
 (r_1r_2)\otimes a=r_1\otimes(r_2\otimes a);\\
\text{(GGV4)}\quad&
 \frac{\phi(|r|\otimes a)}{\|\phi(r\otimes a)\|}
 =\frac{\phi(a)}{\|\phi(a)\|}
 \quad(a\ne e,\ r\ne0);\\
\text{(GGV5)}\quad&
 \gyr[u,v](r\otimes a)=r\otimes\gyr[u,v]a;\\
\text{(GGV6)}\quad&
 \gyr[r_1\otimes v,r_2\otimes v]=\id_G.
\end{align*}
Moreover, the norm-value condition \textup{(GGVV)} is that
\[
 \|\phi(G)\|:=\{\pm\|\phi(a)\|:a\in G\}
\]
is a one-dimensional real linear space with addition $\oplus'$ and scalar
multiplication $\otimes'$, and
\begin{align*}
\text{(GGV7)}\quad&
 \|\phi(r\otimes a)\|=|r|\otimes'\|\phi(a)\|;\\
\text{(GGV8)}\quad&
 \|\phi(a\oplus b)\|
 \le \|\phi(a)\|\oplus'\|\phi(b)\|.
\end{align*}
\end{definition}

The gyrometric is
\[
   \rho(a,b)=\|\phi(a\ominus b)\|,
   \qquad a\ominus b:=a\oplus(\ominus b).
\]
The automorphic inverse property gives
$\ominus(a\ominus b)=(\ominus a)\oplus b$. Since
$\ominus z=(-1)\otimes z$ by \cite[Proposition~2.1(v)]{AHrevisited},
(GGV7) gives $\|\phi(\ominus z)\|=\|\phi(z)\|$. Thus we also have
\[
 \rho(a,b)=\|\phi((\ominus a)\oplus b)\|.
\]
For $a\in G$, write $L_a(x)=a\oplus x$ for left gyrotranslation by $a$.
The gyrometric is symmetric and invariant under every $L_a$
\cite[Proposition~3.5]{AHrevisited}. The gyrotriangle inequality
\cite[Proposition~3.6, equation~(3.1)]{AHrevisited} is
\[
 \rho(a,b)\le \rho(a,c)\oplus'\rho(c,b).
\]

The gyrogroup coaddition is
\[
  a\boxplus b=a\oplus\gyr[a,\ominus b]b.
\]

\Needspace{4\baselineskip}
\begin{definition}[Gyromidpoints]\label{def:gyromidpoint}
For $a,b\in G$, the point
\[
 P(a,b)=\frac12\otimes(a\boxplus b)
       =a\oplus\frac12\otimes(\ominus a\oplus b)
\]
is called the \emph{gyromidpoint} of $a$ and $b$.
The second equality is \cite[equation~(4.1)]{AHrevisited}.
\end{definition}

By \cite[equation~(4.2)]{AHrevisited}, the gyromidpoint satisfies
\begin{equation}\label{eq:midpoint-distance-general}
 \rho(a,P(a,b))=\rho(b,P(a,b))
 =\frac12\otimes'\rho(a,b).
\end{equation}

We shall also use the point reflections of
\cite[Proposition~6.1]{AHrevisited}.  For $c\in G$ define
\begin{equation}\label{eq:reflection}
  \sigma_c(x)=(2\otimes c)\ominus x.
\end{equation}
By parts (p1)--(p3) of that proposition, $\sigma_c$ is a bijective
involutive gyrometric-preserving map whose unique fixed point is $c$.
Part (p4) explicitly gives $\sigma_c(x)=y$ and $\sigma_c(y)=x$ whenever
$c=P(x,y)$. Moreover, part (p5) gives
\begin{equation}\label{eq:reflection-distance-general}
  \rho(\sigma_c(x),x)=2\otimes'\rho(c,x).
\end{equation}

\Needspace{4\baselineskip}
\begin{definition}[Real value-line structure and standard-norm GGV]\label{def:standard-norm}
A GGV has a \emph{real value-line structure} if the vector-space
operations on $\|\phi(G)\|$ are the ordinary real operations, that is,
\[
 \alpha\oplus'\beta=\alpha+\beta,
 \qquad r\otimes'\alpha=r\alpha.
\]
In this case $\|\phi(G)\|=\mathbb R$, since it is a one-dimensional
real vector subspace of $\mathbb R$. In particular, $G\ne\{e\}$.
We call a GGV with real value-line structure a \emph{standard-norm GGV} if,
for every $r\in\mathbb R$, the map
\[
 S_r:G\longrightarrow G,\qquad S_r(a)=r\otimes a,
\]
is continuous in the gyrometric topology.
\end{definition}

\Needspace{4\baselineskip}
\begin{remark}\label{rem:coordinate-gyrometric}
The real value-line condition concerns only the operations in (GGVV), while
standard-norm also requires the fixed-scalar continuity stated in the
definition. Neither condition asserts that $\phi$ is linear or a homeomorphism.
The gyrometric is defined by
\[
 \rho(x,y)=\|\phi(x\ominus y)\|.
\]
In general, this need not equal $\|\phi(x)-\phi(y)\|$, although equality
holds for a normed space with its usual operations and $\phi$ the identity map.

For example, let $G=\mathbb R$ with its usual addition and scalar
multiplication, let $V=\mathbb R^2$ with the Euclidean norm, and define
\[
 \phi(t)=
 \begin{cases}
 (t,0),&t\geq0,\\
 (0,-t),&t<0.
 \end{cases}
\]
This map is injective, $\|\phi(t)\|=|t|$, and
$\phi(rt)=r\phi(t)$ for $r>0$. All gyrations are the identity.
Thus the GGV axioms hold with the ordinary operations on the value line,
and $\rho(s,t)=|s-t|$. Each fixed-scalar map is continuous, so this is a
standard-norm GGV. However,
\[
 \rho(1,-1)=2,\qquad \|\phi(1)-\phi(-1)\|=\sqrt{2}.
\]

For a general GGV, Proposition~2.1(i) of \cite{AHrevisited} states that
$\|\phi(e)\|$ is the additive identity of the value line
$(\|\phi(G)\|,\oplus',\otimes')$. This identity need not be the real number
$0$. Under the real value-line condition, the operations are the ordinary
real operations, so $\|\phi(e)\|=0$ and hence $\phi(e)=0_V$.
Consequently, at the identity we do have
\[
 \rho(x,e)=\|\phi(x)\|=\|\phi(x)-\phi(e)\|.
\]
\end{remark}

\Needspace{4\baselineskip}
\begin{lemma}[Real value-line rays]\label{lem:real-value-line-rays}
Let $G$ be a GGV with real value-line structure.  Then $\rho$ is a genuine
metric and, for every $a\in G$ and $r,s\in\mathbb R$,
\begin{equation}\label{eq:same-ray}
  \rho(r\otimes a,s\otimes a)
  =|r-s|\rho(e,a).
\end{equation}
\end{lemma}

\begin{proof}
Remark~\ref{rem:coordinate-gyrometric} gives $\phi(e)=0_V$.
Thus $\rho(a,b)=0$ if and only if $a\ominus b=e$, by injectivity of $\phi$.
If $a\ominus b=e$, then $a\oplus(\ominus b)=a\oplus(\ominus a)$.
Left cancellation gives $\ominus b=\ominus a$, hence $a=b$.
Conversely, $a\ominus a=e$, so $\rho(a,a)=0$. The gyrometric identities above
give symmetry and, since the value-line operations are ordinary, the usual
triangle inequality.  Thus $\rho$ is a metric.

By \cite[Proposition~2.1(v)]{AHrevisited},
$\ominus(s\otimes a)=(-s)\otimes a$.  Hence, using the definition of the
gyrometric together with (GGV2) and (GGV7),
\[
\begin{aligned}
 \rho(r\otimes a,s\otimes a)
 &=\|\phi((r\otimes a)\oplus((-s)\otimes a))\|\\
 &=\|\phi((r-s)\otimes a)\|\\
 &=|r-s|\,\|\phi(a)\|\\
 &=|r-s|\rho(e,a).
\end{aligned}
\]
Here $\|\phi(a)\|=\rho(a,e)=\rho(e,a)$ by the definition of $\rho$ and symmetry.
\end{proof}

For a GGV $G$ with real value-line structure, $a\in G$, and $R>0$, we write
\[
 B_\rho(a,R)=\{x\in G:\rho(a,x)<R\}
\]
for the open ball with centre $a$ and radius $R$.

\Needspace{4\baselineskip}
\begin{proposition}\label{prop:metric-geodesic}
Let $G$ be a GGV with real value-line structure.  For $a,b\in G$, the curve
\begin{equation}\label{eq:gyrosegment}
 \gamma_{a,b}(t)
 =a\oplus\bigl(t\otimes(\ominus a\oplus b)\bigr),
 \qquad 0\le t\le1,
\end{equation}
has endpoints and gyromidpoint given by
\[
 \gamma_{a,b}(0)=a,\qquad \gamma_{a,b}(1)=b,\qquad
 \gamma_{a,b}\bigl(\tfrac12\bigr)=P(a,b),
\]
and satisfies
\begin{equation}\label{eq:gyrosegment-distance}
 \rho(\gamma_{a,b}(s),\gamma_{a,b}(t))
 =|s-t|\rho(a,b).
\end{equation}
In particular, $G$ is a geodesic, locally path connected metric space.
\end{proposition}

\begin{proof}
Put $c=\ominus a\oplus b$. The identities $0\otimes c=e$ and
$1\otimes c=c$, together with the left inverse property, give
$\gamma_{a,b}(0)=a\oplus e=a$ and
$\gamma_{a,b}(1)=a\oplus(\ominus a\oplus b)=b$.
Definition~\ref{def:gyromidpoint} also gives
$\gamma_{a,b}(1/2)=P(a,b)$.
Left invariance and Lemma~\ref{lem:real-value-line-rays} give
\[
 \rho(\gamma_{a,b}(s),\gamma_{a,b}(t))
 =\rho(s\otimes c,t\otimes c)=|s-t|\rho(e,c).
\]
Since $\ominus a\oplus a=e$, left invariance under $L_{\ominus a}$
gives $\rho(e,c)=\rho(a,b)$, proving \eqref{eq:gyrosegment-distance}.
This formula shows that $\gamma_{a,b}$ is Lipschitz continuous and,
with the endpoint identities, is a constant-speed geodesic from $a$ to $b$.
If $b\in B_\rho(a,R)$, then
$\rho(a,\gamma_{a,b}(t))=t\rho(a,b)<R$, so every metric ball is path
connected. Since open balls form a neighborhood basis, $G$ is locally
path connected.
\end{proof}

\Needspace{4\baselineskip}
\begin{proposition}[Joint scalar continuity]
\label{prop:joint-scalar-continuity}
Every standard-norm GGV has jointly continuous scalar multiplication.
\end{proposition}

\begin{proof}
Let $r_n\to r$ in $\mathbb R$ and $a_n\to a$ in $G$.  By the triangle
inequality and \eqref{eq:same-ray},
\[
 \rho(r_n\otimes a_n,r\otimes a)
 \le |r_n-r|\rho(e,a_n)+\rho(S_r(a_n),S_r(a)).
\]
The first term tends to zero because $\rho(e,a_n)\to\rho(e,a)$, and the second
by the defining continuity of $S_r$.
\end{proof}
\Needspace{4\baselineskip}
\begin{corollary}\label{cor:contractible}
Every standard-norm GGV is contractible and hence simply connected.
\end{corollary}

\begin{proof}
The map
\[
 H:G\times[0,1]\longrightarrow G,
 \qquad H(x,s)=(1-s)\otimes x,
\]
is jointly continuous by Proposition~\ref{prop:joint-scalar-continuity}.
Its endpoints are $H(x,0)=x$ and $H(x,1)=e$, using $0\otimes x=e$
from \cite[Proposition~2.1(iii)]{AHrevisited}.
\end{proof}

\Needspace{4\baselineskip}
\begin{example}[Normed spaces]\label{ex:normed-space}
Let $X$ be a nonzero real normed space.  With ordinary addition and scalar
multiplication, coordinate space $V=X$, and $\phi=\id_X$, the quadruple
$(X,+,\cdot,\id_X)$ is a standard-norm GGV.  Its gyrometric is the
ordinary norm metric.
\end{example}

\begin{proof}
With all gyrations equal to the identity, the vector-space and norm
axioms give (GGV0)--(GGV8). Since $X\ne\{0\}$, the signed norm-value set
$\{\pm\|x\|:x\in X\}$ is $\mathbb R$, with its ordinary operations.
The gyrometric is $\rho(x,y)=\|x-y\|$, and
$\rho(rx,ry)=|r|\rho(x,y)$ shows that $x\mapsto rx$ is Lipschitz,
hence continuous for each fixed $r$. The requirement (GGVV) excludes
$X=\{0\}$, which is treated separately in Corollary~\ref{cor:mankiewicz-open}.
\end{proof}

\Needspace{8\baselineskip}
\subsection{Gyromidpoints and an identity principle}

\Needspace{4\baselineskip}
\begin{definition}[Metric midpoints]\label{def:metric-midpoints}
Let $(X,\rho)$ be a metric space and let $x,y\in X$.  Put
\[
 \Mid(x,y)=\left\{z\in X:\rho(x,z)=\rho(z,y)=\frac12\rho(x,y)\right\}.
\]
A point in $\Mid(x,y)$ is called a \emph{metric midpoint} of $x$ and $y$,
and $\Mid(x,y)$ is called their \emph{metric midpoint set}.  Metric midpoints
need not be unique.
\end{definition}

\Needspace{4\baselineskip}
\begin{lemma}[Rigidity of the metric midpoint set]\label{lem:lens-rigidity}
Let $G$ be a GGV with real value-line structure and let $x,y\in G$.
The gyromidpoint $p=P(x,y)$ belongs to $\Mid(x,y)$, and every bijective
isometry $S:\Mid(x,y)\to \Mid(x,y)$ fixes $p$.
\end{lemma}

The fixed-point principle in this lemma is the one used in
\cite[Lemma~1]{Mankiewicz}.  We give its proof by adapting Nica's doubling
argument for the Mazur--Ulam theorem \cite{Nica2012} to the gyroreflection
on the bounded metric midpoint set $\Mid(x,y)$.

\begin{proof}
By \eqref{eq:midpoint-distance-general}, $p\in \Mid(x,y)$.  The reflection
$\sigma_p$ interchanges $x$ and $y$ and therefore preserves $\Mid(x,y)$.
Let $\mathcal W$ be the family of bijective self-isometries of $\Mid(x,y)$ and set
\[
 \lambda=\sup\{\rho(S(p),p):S\in\mathcal W\}.
\]
The metric midpoint set $\Mid(x,y)$ is bounded because, for $z,w\in \Mid(x,y)$,
\[
 \rho(z,w)\le\rho(z,x)+\rho(x,w)=\rho(x,y),
\]
so $\lambda<\infty$.

For $S\in\mathcal W$ put
\[
 S^\sharp=\sigma_p\circ S^{-1}\circ\sigma_p\circ S.
\]
Since $\sigma_p|_{\Mid(x,y)}\in\mathcal W$, this composition belongs to
$\mathcal W$. The map $\sigma_p$ is an isometry fixing $p$, and $S^{-1}$
is an isometry of $\Mid(x,y)$, which contains $\sigma_pS(p)$.
Thus \eqref{eq:reflection-distance-general}, with the ordinary real
value-line operations, gives
\[
\begin{aligned}
 \rho(S^\sharp(p),p)
 &=\rho(S^{-1}\sigma_pS(p),p)\\
 &=\rho(\sigma_pS(p),S(p))\\
 &=2\rho(S(p),p).
\end{aligned}
\]
Hence $2\rho(S(p),p)\le\lambda$ for every $S\in\mathcal W$.  Taking the
supremum gives $2\lambda\le\lambda$, and therefore $\lambda=0$.
\end{proof}

\Needspace{4\baselineskip}
\begin{lemma}\label{lem:reflection-midpoint}
For every $c,x$ in a GGV with real value-line structure,
\begin{equation}\label{eq:reflection-midpoint}
   P(x,\sigma_c(x))=c.
\end{equation}
\end{lemma}

\begin{proof}
Put $y=\sigma_c(x)$ and $p=P(x,y)$. Since $\sigma_c$ is an
involutive isometry interchanging $x$ and $y$, it restricts to a
bijective self-isometry of $\Mid(x,y)$. Lemma~\ref{lem:lens-rigidity}
therefore gives $\sigma_c(p)=p$. The unique fixed point of
$\sigma_c$ is $c$, so $p=c$.
\end{proof}

The next lemma extends the metric-midpoint-set and reflection argument in
\cite[Lemma~3]{Mankiewicz} from normed spaces to GGVs with real value-line
structure.
\Needspace{4\baselineskip}
\begin{lemma}[Local preservation of gyromidpoints]\label{lem:local-midpoint}
Let $G_1,G_2$ be GGVs with real value-line structure.
Write $P_i$ and $\Mid_i$ for the gyromidpoint and metric midpoint set
in $(G_i,\rho_i)$, respectively. Suppose
\[
 F:B_{\rho_1}(a,R)\longrightarrow B_{\rho_2}(b,R)
\]
is a bijective isometry with $F(a)=b$.  If
$x,y\in B_{\rho_1}(a,R/4)$, then
\begin{equation}\label{eq:local-midpoint}
 F(P_1(x,y))=P_2(F(x),F(y)).
\end{equation}
\end{lemma}

\begin{proof}
Put
\[
 p=P_1(x,y),\qquad p'=P_2(F(x),F(y)),
\]
and let $E_1=\Mid_1(x,y)$ and $E_2=\Mid_2(F(x),F(y))$. If $z\in E_1$, then
\[
 \rho_1(a,z)\le\rho_1(a,x)+\rho_1(x,z)<\frac R4+\frac R4=\frac R2,
\]
so $E_1\subset B_{\rho_1}(a,R)$. Since $F(a)=b$, we have
$\rho_2(b,F(x))=\rho_1(a,x)<R/4$, and similarly for $y$.
The same estimate therefore gives $E_2\subset B_{\rho_2}(b,R)$.
Distance preservation gives $F(E_1)\subset E_2$. Conversely, every
$z'\in E_2$ has a preimage $z$ in the source ball, and distance
preservation gives $z\in E_1$. Thus $F(E_1)=E_2$.

The map
\[
 S=F^{-1}\circ\sigma_{p'}\circ F\circ\sigma_p
\]
is therefore a bijective self-isometry of $E_1$.  By
Lemma~\ref{lem:lens-rigidity}, $S(p)=p$.  As $\sigma_p(p)=p$,
\[
  \sigma_{p'}(F(p))=F(p).
\]
The reflection $\sigma_{p'}$ has the unique fixed point $p'$, so $F(p)=p'$.
The same estimates work for $x,y\in B_{\rho_1}(a,R/2)$, since each
midpoint is then less than $R/2$ from its endpoint. We use $R/4$ for
the reflection estimates below.
\end{proof}

\Needspace{4\baselineskip}
\begin{lemma}\label{lem:ball-restriction}
Let $(X_1,\rho_1)$ and $(X_2,\rho_2)$ be metric spaces, let
$a\in X_1$, $b\in X_2$, and let $R>0$. Suppose
\[
 F:B_{\rho_1}(a,R)\longrightarrow B_{\rho_2}(b,R)
\]
is a bijective isometry with $F(a)=b$.  If $x\in B_{\rho_1}(a,R)$ and
\[
 0<r<R-\rho_1(a,x),
\]
then
\[
 F\bigl(B_{\rho_1}(x,r)\bigr)=B_{\rho_2}(F(x),r).
\]
\end{lemma}

\begin{proof}
For $z\in B_{\rho_1}(x,r)$, the triangle inequality gives
$\rho_1(a,z)<\rho_1(a,x)+r<R$. Thus this ball lies in the domain
of $F$, and distance preservation gives the inclusion from left to right.
Conversely, if $w\in B_{\rho_2}(F(x),r)$, then
\[
 \rho_2(b,w)
 \le \rho_2(b,F(x))+\rho_2(F(x),w)
 <\rho_1(a,x)+r<R.
\]
Thus $w=F(z)$ for some $z\in B_{\rho_1}(a,R)$.  Since
\[
 \rho_1(x,z)=\rho_2(F(x),w)<r,
\]
we have $z\in B_{\rho_1}(x,r)$.
\end{proof}

A bijective isometry
$F:B_{\rho_1}(a,R)\to B_{\rho_2}(b,R)$ is called a
\emph{center-preserving isometry between balls} if $F(a)=b$.
All balls here are open gyrometric balls.

\Needspace{4\baselineskip}
\begin{definition}\label{def:local-isometry}
Let $G_1,G_2$ be GGVs with real value-line structure and gyrometrics
$\rho_1,\rho_2$, and let $O\subset G_1$ be open.  A map $F:O\to G_2$ is called a
\emph{local gyrometric isometry} if, for every $x\in O$, there exists $R>0$
such that
\[
 B_{\rho_1}(x,R)\subset O
\]
and
\[
 F:B_{\rho_1}(x,R)\longrightarrow B_{\rho_2}(F(x),R)
\]
is a bijective isometry.
\end{definition}

A local gyrometric isometry is continuous, since it is $1$-Lipschitz
on each witnessing ball. It is open because, by Lemma~\ref{lem:ball-restriction},
it maps every smaller concentric ball onto the corresponding open ball about
its image center. By the same lemma, a center-preserving isometry between
balls is a local gyrometric isometry on its source ball. Restricting a
local gyrometric isometry to an open subset of its domain again gives a
local gyrometric isometry, since the lemma allows the witnessing balls
to be shrunk inside that subset.

Let $G_1,G_2,G_3$ be GGVs with real value-line structure, and let
$F_1:O_1\to G_2$ and $F_2:O_2\to G_3$ be local gyrometric isometries,
where $O_i\subset G_i$ is open. Their composition is defined on
$F_1^{-1}(O_2)$, which is open by continuity of $F_1$. At a point $x$ of
this domain, choose radii $R_1,R_2$ for the ball isometries of $F_1$ at
$x$ and of $F_2$ at $F_1(x)$. With $r=\tfrac12\min(R_1,R_2)$,
Lemma~\ref{lem:ball-restriction} ensures that each map takes the ball
of radius $r$ about its source point isometrically onto the ball of
radius $r$ about its image. In particular,
\[
 F_1(B_{\rho_1}(x,r))=B_{\rho_2}(F_1(x),r)
 \subset B_{\rho_2}(F_1(x),R_2)\subset O_2,
\]
so $B_{\rho_1}(x,r)\subset F_1^{-1}(O_2)$. Thus $F_2\circ F_1$
is a local gyrometric isometry.

\Needspace{4\baselineskip}
\begin{corollary}[Local covariance of point reflections]
\label{cor:local-reflection}
Let $F:O\to G_2$ be a local gyrometric isometry between GGVs with real
value-line structure. For $c\in O$, choose $R>0$ so that $F$ restricts
to a center-preserving isometry from $B_{\rho_1}(c,R)\subset O$
onto $B_{\rho_2}(F(c),R)$. Then
\begin{equation}\label{eq:reflection-covariance}
 F(\sigma_c(x))=\sigma_{F(c)}(F(x))
\end{equation}
for every $x\in B_{\rho_1}(c,R/4)$.
\end{corollary}

\begin{proof}
Let $x\in B_{\rho_1}(c,R/4)$ and put
$y=\sigma_c(x)$. Since $\sigma_c$ is an isometry fixing $c$,
$y\in B_{\rho_1}(c,R/4)$. By Lemma~\ref{lem:reflection-midpoint},
$P_1(x,y)=c$. Lemma
\ref{lem:local-midpoint} gives
\[
 P_2(F(x),F(y))=F(c).
\]
The reflection about this gyromidpoint interchanges $F(x)$ and $F(y)$, hence
$F(y)=\sigma_{F(c)}(F(x))$.
\end{proof}

\Needspace{4\baselineskip}
\begin{lemma}[Identity principle]\label{lem:identity-principle}
Let $G_1,G_2$ be GGVs with real value-line structure, let $O\subset G_1$ be connected and
open, and let $F_1,F_2:O\to G_2$ be local gyrometric isometries.  If they agree
on a nonempty open subset of $O$, then $F_1=F_2$ on $O$.
\end{lemma}

\begin{proof}
Let
\[
 C=\{x\in O: F_1=F_2\text{ on a neighborhood of }x\}.
\]
Then $C$ is nonempty and open.  We show that it is closed in $O$.

Let $x\in\overline C\cap O$.  Since local isometries are continuous,
$F_1(x)=F_2(x)$. Choose $4R$ smaller than the minimum of the two witnessing radii
at $x$. By Lemma~\ref{lem:ball-restriction}, for $i=1,2$,
\[
 F_i:B_{\rho_1}(x,4R)\longrightarrow
 B_{\rho_2}(F_i(x),4R)
\]
is a center-preserving isometry and $B_{\rho_1}(x,4R)\subset O$.
Choose $y\in C$ with $\rho_1(x,y)<R/2$ and put $m=P_1(x,y)$.  Applying
Lemma~\ref{lem:local-midpoint} to each $F_i$ gives
\[
 F_i(m)=P_2(F_i(x),F_i(y)).
\]
Since $F_1(x)=F_2(x)$ and $F_1(y)=F_2(y)$, we have
\[
 F_1(m)=F_2(m)=:m'.
\]

Because $\rho_1(x,m)=\rho_1(y,m)<R/4$, Lemma~\ref{lem:ball-restriction}
applied to these balls of radius $4R$ shows that, for $i=1,2$,
\[
 F_i:B_{\rho_1}(m,3R)\longrightarrow B_{\rho_2}(m',3R)
\]
is again a center-preserving isometry.
Corollary~\ref{cor:local-reflection}, applied with radius $3R$, gives
\[
 F_i(\sigma_m(w))=\sigma_{m'}(F_i(w)),\qquad i=1,2,
\]
for every $w\in B_{\rho_1}(m,3R/4)$.  Since
$\rho_1(m,y)<R/4<3R/4$ and $y\in C$, we may choose a nonempty open
neighborhood
\[
 W\subset B_{\rho_1}(m,3R/4)
\]
of $y$ on which $F_1=F_2$. Since $\sigma_m$ is an isometry fixing $m$,
\[
 \sigma_m(W)\subset B_{\rho_1}(m,3R/4)
 \subset B_{\rho_1}(x,R)\subset O.
\]
The covariance formula therefore gives $F_1=F_2$ on $\sigma_m(W)$.
This set is a neighborhood of $x$ because $\sigma_m$ is a homeomorphism and
$\sigma_m(y)=x$.  Thus $x\in C$.

Hence $C$ is both open and closed in the connected set $O$, so $C=O$.
\end{proof}

\Needspace{8\baselineskip}
\subsection{Continuation of local isometries}

\Needspace{4\baselineskip}
\begin{definition}[Germs and continuation]\label{def:germs-continuation}
Let $G_1,G_2$ be GGVs with real value-line structure and let $a\in G_1$.
Two local gyrometric isometries $F:O_F\to G_2$ and $H:O_H\to G_2$,
defined on open neighborhoods of $a$ in $G_1$, are equivalent at $a$ if
there exists an open neighborhood $W$ of $a$ such that
\[
 W\subset O_F\cap O_H,\qquad F|_W=H|_W.
\]
The equivalence class of $F$ is called its \emph{germ at $a$} and is
denoted by $[F]_a$.  Its value at $a$ is $F(a)$ and is independent of the
representative.  Germs are equal when their representatives agree on a
neighborhood of $a$; equality of values at $a$ is not sufficient.
All germs and representatives below have target $G_2$.

Let $\gamma:[0,1]\to G_1$ be a continuous path.  A \emph{continuation} of a local
isometry germ $\mathfrak f_0$ at $\gamma(0)$ along $\gamma$ is a family
$(\mathfrak f_t)_{t\in[0,1]}$ such that $\mathfrak f_t$ is a local
isometry germ at $\gamma(t)$, the germ at $t=0$ is $\mathfrak f_0$, and for
every $t_0\in[0,1]$ there exist a relative open interval $I\subset[0,1]$
containing $t_0$ and a local gyrometric isometry $F$ defined on an open
neighborhood of $\gamma(I)$ such that
\[
 \mathfrak f_t=[F]_{\gamma(t)}\qquad(t\in I).
\]
\end{definition}

For paths $\alpha,\beta$ with $\alpha(1)=\beta(0)$, the concatenation
$\alpha*\beta$ traverses $\alpha$ first and then $\beta$.
We write $\alpha^{-1}(t)=\alpha(1-t)$ for the reversed path and
$\operatorname{const}_a$ for the constant path at $a$.

Continuations concatenate when the terminal germ of the first equals
the initial germ of the second. Indeed, representatives near the joining
point agree on a neighborhood of that point, and continuity of the paths
provides a common representative near the joining parameter.
Reversing $(\mathfrak f_t)$ gives $(\mathfrak f_{1-t})$ along the reversed
path. A constant family of germs gives a continuation along a constant path.

\Needspace{4\baselineskip}
\begin{lemma}[Uniqueness of a representative of fixed radius]
\label{lem:fixed-radius-representative}
Let $G_1,G_2$ be GGVs with real value-line structure. If two
center-preserving isometries from $B_{\rho_1}(x,R)$ onto balls of radius
$R$ in $G_2$ represent the same germ at $x$, then they agree on
$B_{\rho_1}(x,R)$.
\end{lemma}

\begin{proof}
Write the maps as $F$ and $H$. Equality of their germs gives
$F(x)=H(x)$, so their target balls are both $B_{\rho_2}(F(x),R)$.
By Lemma~\ref{lem:ball-restriction}, both maps are local gyrometric
isometries on the source ball. They agree near $x$ because their germs coincide,
and the ball is path connected by Proposition~\ref{prop:metric-geodesic},
so Lemma~\ref{lem:identity-principle} gives agreement throughout the ball.
\end{proof}

\Needspace{4\baselineskip}
\begin{lemma}[Reflection step]\label{lem:reflection-step}
Let $G_1,G_2$ be GGVs with real value-line structure, and let
$F_u:B_{\rho_1}(u,R)\to B_{\rho_2}(F_u(u),R)$ be a center-preserving
isometry. If $\rho_1(u,v)<R/4$, put $m=P_1(u,v)$ and define
\begin{equation}\label{eq:continued-map}
 F_v(z)=\sigma_{F_u(m)}\bigl(F_u(\sigma_m(z))\bigr),
 \qquad z\in B_{\rho_1}(v,R).
\end{equation}
Then $F_v$ is a center-preserving isometry onto
$B_{\rho_2}(F_u(v),R)$, and
\begin{equation}\label{eq:consecutive-agreement}
 F_u=F_v\qquad\text{on }B_{\rho_1}(v,R/2).
\end{equation}
\end{lemma}

\begin{proof}
Since $u,v\in B_{\rho_1}(u,R/4)$, Lemma~\ref{lem:local-midpoint}
gives
\[
 F_u(m)=P_2(F_u(u),F_u(v)).
\]
The isometric involution $\sigma_m$ interchanges $u$ and $v$, so it
maps $B_{\rho_1}(v,R)$ onto $B_{\rho_1}(u,R)$. Likewise,
$\sigma_{F_u(m)}$ maps $B_{\rho_2}(F_u(u),R)$ onto
$B_{\rho_2}(F_u(v),R)$. Hence \eqref{eq:continued-map} is a bijective
isometry between the stated balls, and
$F_v(v)=\sigma_{F_u(m)}(F_u(u))=F_u(v)$.

Both endpoints are within $R/8$ of $m$. By
Lemma~\ref{lem:ball-restriction}, $F_u$ restricts to a center-preserving
isometry on $B_{\rho_1}(m,r)$ for $r=7R/8$. Also
$B_{\rho_1}(m,r/4)\subset B_{\rho_1}(v,R/2)$, since
$R/8+r/4=11R/32<R/2$. For $z\in B_{\rho_1}(m,r/4)$,
$\sigma_m(z)$ stays in that ball, since $\sigma_m$ is an isometry
fixing $m$. Applying Corollary~\ref{cor:local-reflection} to
$F_u|_{B_{\rho_1}(m,r)}$ with center $m$ and evaluation point
$\sigma_m(z)\in B_{\rho_1}(m,r/4)$ gives
\[
 F_v(z)=\sigma_{F_u(m)}\bigl(F_u(\sigma_m(z))\bigr)=F_u(z).
\]
The ball $B_{\rho_1}(v,R/2)$ lies in both domains by $\rho_1(u,v)<R/4$,
and it is path connected by Proposition~\ref{prop:metric-geodesic}.
The restrictions are local gyrometric isometries by
Lemma~\ref{lem:ball-restriction}. Lemma~\ref{lem:identity-principle}
therefore extends the agreement to that whole ball.
\end{proof}

The radius $R$ is unchanged in the reflection step. Successive steps
therefore use a single radius throughout a path and, later, throughout
a homotopy grid; no limiting construction in a complete metric space
is needed.

\Needspace{4\baselineskip}
\begin{lemma}[Continuation along paths]\label{lem:path-continuation}
Let $G_1,G_2$ be GGVs with real value-line structure.  Every local gyrometric-isometry germ
from $G_1$ to $G_2$ can be continued along every continuous path in $G_1$
starting at the base point of the germ.
\end{lemma}

\begin{proof}
By Definition~\ref{def:local-isometry}, represent the initial germ
by a center-preserving isometry between balls
\[
 F_0:B_{\rho_1}(a_0,R)\longrightarrow B_{\rho_2}(b_0,R).
\]
Let $\gamma:[0,1]\to G_1$ be a path with $\gamma(0)=a_0$.  By uniform
continuity, choose a partition
\[
 0=t_0<t_1<\cdots<t_N=1
\]
so fine that, with $a_j=\gamma(t_j)$,
\begin{equation}\label{eq:path-small-segment}
 \operatorname{diam}_{\rho_1}\gamma([t_j,t_{j+1}])<R/4
 \qquad(j=0,\ldots,N-1).
\end{equation}
In particular, $\rho_1(a_j,a_{j+1})<R/4$.

Starting with $F_0$, apply Lemma~\ref{lem:reflection-step} successively
with $u=a_j$ and $v=a_{j+1}$. It gives a representative $F_{j+1}$ on
$B_{\rho_1}(a_{j+1},R)$ and agreement \eqref{eq:consecutive-agreement}
on $B_{\rho_1}(a_{j+1},R/2)$.

By \eqref{eq:path-small-segment}, the whole subpath
$\gamma([t_j,t_{j+1}])$ lies in $B_{\rho_1}(a_j,R/4)$ and hence in the
domain of $F_j$.  For $t\in[t_j,t_{j+1}]$, let $\mathfrak f_t$ be the germ
of $F_j$ at $\gamma(t)$. This is well defined at each interior
partition point $t_{j+1}$, since \eqref{eq:consecutive-agreement} gives
$F_j=F_{j+1}$ on $B_{\rho_1}(a_{j+1},R/2)$.
At such a point take $I=(t_j,t_{j+2})$. By \eqref{eq:path-small-segment},
$\gamma(I)\subset B_{\rho_1}(a_{j+1},R/2)$, and by
\eqref{eq:consecutive-agreement} the restriction of $F_{j+1}$ to this ball
is a common representative on $I$. Away from the interior partition points,
the individual $F_j$ provide the representatives required by
Definition~\ref{def:germs-continuation}, including $F_0$ at $0$ and
$F_{N-1}$ at $1$.
\end{proof}

\Needspace{4\baselineskip}
\begin{lemma}[Uniqueness along a fixed path]\label{lem:fixed-path-uniqueness}
Let $G_1,G_2$ be GGVs with real value-line structure, let
$\gamma:[0,1]\to G_1$ be a continuous path, and fix a local
isometry germ $\mathfrak f_0$ at $\gamma(0)$.  Then any two continuations of
$\mathfrak f_0$ along $\gamma$ coincide at every parameter value.  In
particular, the terminal germ at $\gamma(1)$ is uniquely determined by
$\gamma$ and $\mathfrak f_0$.
\end{lemma}

\begin{proof}
Let $(\mathfrak f_t)$ and $(\mathfrak g_t)$ be two continuations of
the same initial germ, and put
\[
 E=\{t\in[0,1]:\mathfrak f_t=\mathfrak g_t\}.
\]
Fix $t_0\in[0,1]$ and take local representatives $F,G$ of the two
families near $t_0$. Choose a ball $B$ about $\gamma(t_0)$ in their
common domain. By continuity of $\gamma$, shrink the intersection of the
two parameter intervals to a relative open interval $I$ about $t_0$ with
\[
 \gamma(I)\subset B\subset\operatorname{dom}F\cap\operatorname{dom}G.
\]
If $I\cap E\ne\varnothing$, equality of the germs at a point of
$\gamma(I\cap E)$ gives agreement of $F,G$ on a nonempty open
subset of $B$. The ball is path connected by
Proposition~\ref{prop:metric-geodesic}, so Lemma~\ref{lem:identity-principle}
gives $F=G$ on $B$, so $I\subset E$. Thus membership in $E$
is locally constant. Since $0\in E$ and $[0,1]$ is connected,
$E=[0,1]$.
\end{proof}

A continuous nondecreasing surjection $\theta:[0,1]\to[0,1]$
satisfies $\theta(0)=0$ and $\theta(1)=1$. The family $(\mathfrak f_{\theta(t)})_t$ is a continuation along
$\gamma\circ\theta$: continuity of $\theta$ lets us use the same local
representatives on sufficiently small parameter intervals. By
Lemma~\ref{lem:fixed-path-uniqueness}, reparametrization preserves
the terminal germ.

\Needspace{4\baselineskip}
\begin{definition}[One-step continuation]\label{def:one-step-continuation}
In Lemma~\ref{lem:reflection-step}, $F_v$ is the
\emph{one-step continuation} of $F_u$ from $u$ to $v$.
For fixed $R$, Lemma~\ref{lem:fixed-radius-representative} makes $F_u$,
and hence $F_v$, unique for the given initial germ whenever a
radius-$R$ representative exists.
\end{definition}

If $\rho_1(u,v)<R/4$, then along any path from $u$ to $v$ in
$B_{\rho_1}(u,R)$, the germs of $F_u$ give a continuation from $[F_u]_u$ to $[F_u]_v=[F_v]_v$.

\Needspace{4\baselineskip}
\begin{lemma}[Small-cell consistency]\label{lem:small-cell}
Let $F_a:B_{\rho_1}(a,R)\to B_{\rho_2}(a',R)$ be a center-preserving isometry
between GGVs with real value-line structure.  Let $b,c,d\in G_1$ satisfy
\[
 \rho_1(u,v)<\frac{R}{8}
 \qquad (u,v\in\{a,b,c,d\}).
\]
Let $F_b$ and $F_c$ be the one-step continuations of $F_a$ from $a$ to $b$
and from $a$ to $c$, respectively, as in Definition~\ref{def:one-step-continuation}.  Let $F_d^{(b)}$ be the one-step
continuation of $F_b$ from $b$ to $d$, and let $F_d^{(c)}$ be the one-step
continuation of $F_c$ from $c$ to $d$. Then
\[
 F_d^{(b)}=F_d^{(c)}\qquad\text{on }B_{\rho_1}(d,R).
\]
\end{lemma}

\begin{proof}
The bound $R/8$ is chosen to match the homotopy grid below; the
four edge lengths being less than $R/4$ would suffice here.
The four steps are defined, and
\[
 W=B_{\rho_1}(d,R/2)\cap B_{\rho_1}(b,R/2)
                     \cap B_{\rho_1}(c,R/2)
\]
is an open neighborhood of $d$.
By \eqref{eq:consecutive-agreement} for the four steps, both terminal
maps agree with $F_a$ on $W$.
Lemma~\ref{lem:fixed-radius-representative} extends their agreement
to $B_{\rho_1}(d,R)$.
\end{proof}

We write $\pi_1(G_1,a_0)$ for the fundamental group of loops based at
$a_0$, with homotopies fixing the base point and multiplication
$[\alpha][\beta]=[\alpha*\beta]$.

\Needspace{4\baselineskip}
\begin{definition}[Continuation orbit]\label{def:monodromy-action}
Let $G_1,G_2$ be GGVs with real value-line structure, let $a_0\in G_1$,
and let $\mathfrak f_0$ be a germ at $a_0$ of a local gyrometric isometry
from $G_1$ to $G_2$.
Let $\mathcal O(\mathfrak f_0)$ be the set of germs at $a_0$ obtained by
continuing $\mathfrak f_0$ along loops based at $a_0$; we call this set
the \emph{continuation orbit} of $\mathfrak f_0$.  For a loop $\alpha$ based at $a_0$,
continuation along $\alpha$ maps $\mathcal O(\mathfrak f_0)$ to itself.
\end{definition}

\Needspace{4\baselineskip}
\begin{theorem}[Continuation and monodromy]\label{thm:monodromy-action}
Let $G_1,G_2$ be GGVs with real value-line structure.  A germ
$\mathfrak f_0$ at $a_0\in G_1$ of a local gyrometric isometry from $G_1$
to $G_2$ can be continued
along every path starting at $a_0$. Two such paths with the same endpoint
give the same terminal germ whenever they are homotopic relative to their
endpoints.  Consequently, continuation defines a right
action, called the \emph{monodromy action},
\[
 \mathcal O(\mathfrak f_0)\times\pi_1(G_1,a_0)
 \longrightarrow\mathcal O(\mathfrak f_0).
\]
Moreover, $\mathfrak f_0$ extends to a local gyrometric isometry
$F:G_1\to G_2$ defined on all of $G_1$ if and only if it is fixed by this action.  When it exists,
the extension is unique. The criterion for a bijective isometric extension
of a map between open sets is given in Theorem~\ref{thm:two-sided-monodromy}.
\end{theorem}

\begin{proof}
Lemmas~\ref{lem:path-continuation} and \ref{lem:fixed-path-uniqueness}
give existence and uniqueness of continuation along each fixed path.  We prove
homotopy invariance. Fix a representative of $\mathfrak f_0$ that is a
center-preserving isometry between balls of radius $R$, and let
$H:[0,1]^2\to G_1$ be a fixed-endpoint homotopy from
$\gamma_0$ to $\gamma_1$, with
\[
 H(s,0)=\gamma_0(s),\quad H(s,1)=\gamma_1(s),\quad
 H(0,t)=a_0,\quad H(1,t)=\gamma_0(1).
\]
Since $[0,1]^2$ is compact, $H$ is uniformly continuous as a map into
$(G_1,\rho_1)$. Thus there is a finite rectangular grid such that the image of every closed cell has diameter less
than $R/8$.

Along each prescribed finite grid path, one-step continuation preserves
the radius $R$ and the center-preserving property by
Lemma~\ref{lem:reflection-step}. Thus radius-$R$ representatives
are available by induction before path independence is proved. Each
edge image lies in the $R/8$-ball about its initial vertex image, so the
initial representative supplies continuation along that edge. The
one-step continuation gives the same terminal germ, by the overlap
property and Lemma~\ref{lem:fixed-path-uniqueness}.

In each cell the four vertex images are pairwise within $R/8$.
The two routes coincide up to the lower-left vertex of the cell, so
they reach it with the same radius-$R$ representative.
Lemma~\ref{lem:small-cell} gives the same terminal radius-$R$
representative when a rightward and then an upward edge are exchanged for
an upward and then a rightward edge across the cell. Applying the same
reflection steps along a common remaining edge path preserves this
equality. A finite sequence of such
exchanges takes the route along the bottom and then the right side of
the parameter square to the route along the left and then the top side.
Continuation takes place along the images of these routes under $H$.
On each vertical boundary edge $u=v$, so $m=u$ and
\eqref{eq:consecutive-agreement} gives agreement on $B_{\rho_1}(u,R/2)$.
Lemma~\ref{lem:fixed-radius-representative} then gives $F_v=F_u$ on
$B_{\rho_1}(u,R)$. The horizontal boundary routes are reparametrizations of
$\gamma_0$ and $\gamma_1$ under $H$, so these paths have the same terminal
germ.

For $\mathfrak g\in\mathcal O(\mathfrak f_0)$ and a based loop $\alpha$,
define $\mathfrak g\cdot[\alpha]$ to be the terminal germ of its
continuation along $\alpha$. The preceding existence, uniqueness,
and homotopy-invariance arguments apply to every initial germ: each
has a ball representative of some positive radius by
Definition~\ref{def:local-isometry}, and the argument imposes no further
condition on that radius. Thus this depends only on $\mathfrak g$
and $[\alpha]$.
If continuation of $\mathfrak f_0$ along a based loop $\eta$ ends at $\mathfrak g$,
concatenation along $\eta*\alpha$ shows that
$\mathfrak g\cdot[\alpha]\in\mathcal O(\mathfrak f_0)$.
The constant-path and concatenation properties following
Definition~\ref{def:germs-continuation}, together with fixed-path
uniqueness, give
\[
 \mathfrak g\cdot[\operatorname{const}_{a_0}]=\mathfrak g,\qquad
 (\mathfrak g\cdot[\alpha])\cdot[\beta]
 =\mathfrak g\cdot[\alpha*\beta].
\]
These are the laws of a right action for the stated convention.

Suppose now that every based loop fixes $\mathfrak f_0$.
Let $\gamma,\eta$ be paths from $a_0$ to $x$, with terminal germs
$\mathfrak g_\gamma,\mathfrak g_\eta$. Continuation along
$\gamma*\eta^{-1}$ returns to $\mathfrak f_0$. By concatenation and
fixed-path uniqueness, the continuation along its second part
$\eta^{-1}$ runs from $\mathfrak g_\gamma$ to $\mathfrak f_0$.
Reversing it and using Lemma~\ref{lem:fixed-path-uniqueness} along $\eta$
gives $\mathfrak g_\gamma=\mathfrak g_\eta$.

Since $G_1$ is path connected by Proposition~\ref{prop:metric-geodesic},
there is a well-defined terminal germ at every $x\in G_1$.
Define $F(x)$ to be its value at $x$, and represent this germ by
a center-preserving isometry
\[
 F_x:B_{\rho_1}(x,r)\longrightarrow B_{\rho_2}(F(x),r).
\]
For any $y\in B_{\rho_1}(x,r)$, append the canonical gyrosegment from
$x$ to $y$, which stays in that ball by
Proposition~\ref{prop:metric-geodesic}. The representative $F_x$ supplies
continuation along this segment, so path independence gives $F(y)=F_x(y)$.
Hence $F$ agrees with the center-preserving isometry $F_x$ on the whole
ball and is a local gyrometric isometry. At $a_0$, the constant path
has terminal germ $\mathfrak f_0$, so $[F]_{a_0}=\mathfrak f_0$.

Conversely, a local gyrometric isometry defined on all of $G_1$ and
extending $\mathfrak f_0$ supplies continuation around every based loop
returning to $\mathfrak f_0$. Fixed-path uniqueness therefore gives
the fixed-point condition. Lemma~\ref{lem:identity-principle} gives
uniqueness of the global extension.
\end{proof}

If $G_1$ is simply connected, then $\pi_1(G_1,a_0)$ is trivial,
so its action fixes every germ in the continuation orbit. Hence
Theorem~\ref{thm:monodromy-action} gives a unique continuation of each germ
to a local gyrometric isometry on all of $G_1$.

\Needspace{4\baselineskip}
\begin{theorem}[Two-sided monodromy criterion]
\label{thm:two-sided-monodromy}
For $i=1,2$, let $G_i$ be a GGV with real value-line structure and
gyrometric $\rho_i$, and let $U_i\subset G_i$ be a nonempty open set.
Assume that $U_1$ is connected. Suppose
that
\[
 T:U_1\longrightarrow U_2
\]
is a bijection satisfying
\[
 \rho_2(Tx,Ty)=\rho_1(x,y),\qquad x,y\in U_1.
\]
Fix $a_0\in U_1$.  If the germ of $T$ at $a_0$ and the germ of $T^{-1}$ at
$T(a_0)$ are fixed by their respective monodromy actions, then $T$ extends
uniquely to a bijective gyrometric isometry
\[
 \widetilde T:G_1\longrightarrow G_2.
\]
Conversely, if such an extension exists, then continuation of the germ of
$T$ along every loop based at $a_0$ returns to the same germ, and the
analogous statement holds for the germ of $T^{-1}$ at $T(a_0)$.
\end{theorem}

\begin{proof}
For each $x\in U_1$, choose $R>0$ such that
$B_{\rho_1}(x,R)\subset U_1$ and
$B_{\rho_2}(T(x),R)\subset U_2$.  Distance preservation and surjectivity give
\[
 T\bigl(B_{\rho_1}(x,R)\bigr)=B_{\rho_2}(T(x),R).
\]
Indeed, each point of the target ball has a preimage in $U_1$, and the
distance of this preimage from $x$ equals the corresponding target distance.
Thus $T$ and $T^{-1}$ are local gyrometric isometries in the sense of
Definition~\ref{def:local-isometry}.

By the two monodromy assumptions and Theorem~\ref{thm:monodromy-action},
the germs of $T$ and $T^{-1}$ extend to local gyrometric isometries
\[
 \widetilde T:G_1\longrightarrow G_2,
 \qquad \widetilde S:G_2\longrightarrow G_1.
\]
The maps $T$ and $\widetilde T$ agree on a neighborhood of $a_0$.
Since $U_1$ is connected, Lemma~\ref{lem:identity-principle} gives
$\widetilde T|_{U_1}=T$. Since $T$ is continuous and onto, $U_2=T(U_1)$
is connected as well. The same argument gives $\widetilde S|_{U_2}=T^{-1}$.

The compositions $\widetilde S\circ\widetilde T$ and
$\widetilde T\circ\widetilde S$ are local gyrometric isometries by
the composition argument following Definition~\ref{def:local-isometry}:
Lemma~\ref{lem:ball-restriction} restricts the two maps to concentric
balls of a common radius smaller than their two witnessing radii.
Since $\widetilde T(U_1)=T(U_1)=U_2$ and
$\widetilde S|_{U_2}=T^{-1}$, their first composition is the identity
on $U_1$; similarly, the second is the identity on $U_2$.
The spaces $G_i$ are connected by Proposition~\ref{prop:metric-geodesic},
so Lemma~\ref{lem:identity-principle} gives
\[
 \widetilde S\circ\widetilde T=\operatorname{id}_{G_1},
 \qquad
 \widetilde T\circ\widetilde S=\operatorname{id}_{G_2}.
\]
Thus $\widetilde T$ is bijective with inverse $\widetilde S$.

For $x,y\in G_1$, let $\gamma$ be their canonical gyrosegment from
Proposition~\ref{prop:metric-geodesic}. Applying the Lebesgue number lemma
to the cover of $[0,1]$ by preimages under $\gamma$ of the witnessing balls
for $\widetilde T$ in Definition~\ref{def:local-isometry} gives a partition
$0=t_0<\cdots<t_N=1$ with each $\gamma([t_{j-1},t_j])$ contained in one
such ball. By the triangle inequality, isometry on each ball, and the
constant-speed formula \eqref{eq:gyrosegment-distance},
\[
\begin{aligned}
 \rho_2(\widetilde T(x),\widetilde T(y))
 &\le \sum_{j=1}^{N}\rho_2\bigl(\widetilde T(\gamma(t_{j-1})),
                                  \widetilde T(\gamma(t_j))\bigr)\\
 &=\sum_{j=1}^{N}\rho_1(\gamma(t_{j-1}),\gamma(t_j))
 =\rho_1(x,y).
\end{aligned}
\]
The same argument for $\widetilde S$ in $G_2$, applied to
$\widetilde T(x),\widetilde T(y)$, gives the reverse inequality because
$\widetilde S\circ\widetilde T=\operatorname{id}_{G_1}$.
Thus $\widetilde T$ is a global bijective gyrometric isometry.
Every global bijective gyrometric isometry is a local gyrometric
isometry by the first paragraph with $U_i=G_i$. Any other such extension
agrees with $\widetilde T$ on $U_1$, so Lemma~\ref{lem:identity-principle}
gives uniqueness.
Conversely, a global bijective gyrometric isometric extension and
its inverse extend the two germs as local gyrometric isometries on the
whole spaces. The necessity direction of the extension criterion in
Theorem~\ref{thm:monodromy-action} shows that both germs are fixed by
their respective monodromy actions. This argument applies at every
$a_0\in U_1$, so the two-sided criterion is independent of the base point.
\end{proof}

\Needspace{4\baselineskip}
\begin{theorem}[Simply connected local-to-global extension theorem]
\label{thm:general-local-global}
For $i=1,2$, let $G_i$ be a simply connected GGV with real value-line
structure, identity $e_i$, and gyrometric $\rho_i$. Let $U_i\subset G_i$
be a nonempty open set and assume that $U_1$ is connected. If
\[
 T:U_1\longrightarrow U_2
\]
is a bijection satisfying
\[
 \rho_2(Tx,Ty)=\rho_1(x,y),\qquad x,y\in U_1,
\]
then $T$ extends uniquely to a bijective gyrometric isometry
\[
 \widetilde T:G_1\longrightarrow G_2.
\]
Moreover, there are a unique $b\in G_2$ and a unique bijective gyrometric
isometry $S:G_1\to G_2$ with $S(e_1)=e_2$ such that
\[
 \widetilde T(x)=b\oplus_2S(x),\qquad x\in G_1,
\]
and
\[
 S(x\oplus_1 y)=S(x)\oplus_2S(y),
 \qquad
 S(r\otimes_1x)=r\otimes_2S(x)
\]
for all $x,y\in G_1$ and $r\in\mathbb R$. They are determined by
\[
 b=\widetilde T(e_1),
 \qquad
 S(x)=(\ominus b)\oplus_2\widetilde T(x).
\]
\end{theorem}

\begin{proof}
The fundamental groups of $G_1$ and $G_2$ are trivial, so both
monodromy actions fix their initial germs.
Theorem~\ref{thm:two-sided-monodromy} gives the unique extension
$\widetilde T$ as a bijective gyrometric isometry.

Put $b=\widetilde T(e_1)$ and $S=L_{\ominus b}\circ\widetilde T$,
where $L_c(y)=c\oplus_2y$.
Left gyrotranslations are bijective isometries by the left inverse
property and \cite[Proposition~3.5]{AHrevisited}, so $S$ is a bijective
gyrometric isometry with $S(e_1)=e_2$.
The GGV Mazur--Ulam theorem \cite[Corollary~4.5]{AHrevisited} applies
to surjective gyrometric-preserving maps between GGVs, so its hypotheses
are met here. Its normalized map is $S$ by left cancellation. Equations
(4.3) and (4.4) of that corollary give preservation of gyroaddition and
of scalar multiplication for every $r\in\mathbb R$.
The left inverse property $L_b\circ L_{\ominus b}=\operatorname{id}_{G_2}$
gives $\widetilde T=L_b\circ S$.
Conversely, this representation with $S(e_1)=e_2$ forces
$b=\widetilde T(e_1)$ and $S=L_{\ominus b}\circ\widetilde T$, proving
uniqueness.
\end{proof}

The same conclusion holds for a local gyrometric isometry
$f:O\to G_2$ in the sense of Definition~\ref{def:local-isometry},
where $O\subset G_1$ is nonempty, connected and open and $G_1,G_2$
satisfy the hypotheses of Theorem~\ref{thm:general-local-global}.
Indeed, on a ball about any point of $O$, the restriction of $f$ maps
that ball isometrically onto the corresponding ball in $G_2$.
The source ball is path connected by its radial gyrosegments.
Theorem~\ref{thm:general-local-global} therefore extends this restriction
to a global bijective isometry. Lemma~\ref{lem:identity-principle}
gives agreement with $f$ throughout $O$ and uniqueness on $G_1$.

\Needspace{4\baselineskip}
\begin{corollary}[Standard-norm local-to-global extension theorem]
\label{cor:standard-local-global}
For $i=1,2$, let $G_i$ be a standard-norm GGV with identity $e_i$,
gyrometric $\rho_i$, and coordinate map $\phi_i$. Let $U_i\subset G_i$
be a nonempty open set and assume that $U_1$ is connected. Every bijective
gyrometric isometry $T:U_1\to U_2$ extends uniquely to a bijective gyrometric
isometry $\widetilde T:G_1\to G_2$.
Moreover, there are a unique $b\in G_2$ and a unique bijective
gyrometric isometry $S:G_1\to G_2$ fixing the identities and preserving
gyroaddition and scalar multiplication such that
\[
 \widetilde T(x)=b\oplus_2S(x),\qquad x\in G_1.
\]
Here $b=\widetilde T(e_1)$, and
$\|\phi_2(Sx)\|=\rho_2(Sx,e_2)=\rho_1(x,e_1)=\|\phi_1(x)\|$.
\end{corollary}

\begin{proof}
By Definition~\ref{def:standard-norm}, $G_1,G_2$ have real value-line
structure. Corollary~\ref{cor:contractible} makes them contractible and
hence simply connected. Theorem~\ref{thm:general-local-global} gives the unique
extension and its stated form.
\end{proof}

\Needspace{4\baselineskip}
\begin{corollary}[Mankiewicz's theorem for connected open subsets]
\label{cor:mankiewicz-open}
Let $X$ and $Y$ be real normed spaces, $U\subset X$ a nonempty connected
open set, and $V\subset Y$ an open set. Every bijective isometry
$T:U\to V$ extends uniquely to a surjective affine isometry
$\widetilde T:X\to Y$.  Equivalently, there exist a unique $b\in Y$ and
a unique surjective real linear isometry $L:X\to Y$ such that
\[
 \widetilde T(x)=b+Lx,\qquad x\in X.
\]
\end{corollary}

\begin{proof}
If either space is zero-dimensional, bijectivity makes both $U$ and $V$
singletons. A nonzero normed space has no isolated points, so openness
forces both spaces to be zero-dimensional. The assertion follows in this
case. Otherwise Example~\ref{ex:normed-space}
makes $X,Y$
standard-norm GGVs with their usual operations and norm metrics.
Corollary~\ref{cor:standard-local-global} gives the unique extension
$\widetilde T(x)=b+Lx$ and the unique data $b,L$, where $L$ is a
bijective isometry preserving the usual addition and scalar
multiplication, hence real linear.
\end{proof}

This is \cite[Theorem~2]{Mankiewicz}. The midpoint and reflection argument
is common to both proofs. Here continuation of germs replaces the
identification of affine extensions.
Connectedness of the source open set is essential, even for normed
spaces and the positive cone of $\mathbb R^2$; the complete counterexample
is given in Remark~\ref{rem:connectedness}.

\Needspace{8\baselineskip}
\section{Kobayashi isometries of bounded symmetric domains}
\label{sec:bsd-kobayashi}

We apply Theorem~A to the open unit balls of complex JB$^*$-triples.
For such a triple $E$, we write $E_{\mathbb R}$ for its underlying real
JB$^*$-triple, obtained by restricting scalars to $\mathbb R$ and
retaining the same norm and triple product. Let $D$ be the open unit
ball of $E$. Let $k_D$ denote the Kobayashi distance on $D$.
Throughout this section, we use the normalization
$k_{\mathbb D}(0,t)=\operatorname{arctanh}t$ for $0\le t<1$,
where $\mathbb D$ is the open unit disc in $\mathbb C$.
Theorem~B shows that, in arbitrary
dimension, the restriction $k_D|_{U\times U}$ to any nonempty connected
open subset $U\subseteq D$ determines the underlying real JB$^*$-triple
$E_{\mathbb R}$ up to real-linear triple isomorphism. More precisely,
if two such open subsets are isometric for their respective restricted
distances, then the corresponding ambient triples admit a real-linear
bijection preserving the triple product.

The proof uses the domain gyrogroup of Friedman and Ungar
\cite{FriedmanUngar1994} twice. First, we equip it with scalar
multiplication and a norm coordinate so that its gyrometric is the
Kobayashi distance. Theorem~A then gives the global extension.
Composing with a target transvection gives a bijective isometry
$S:D_1\to D_2$ with $S(0)=0$. Theorem~A also gives
$S(a\oplus_1 z)=S(a)\oplus_2 S(z)$. On $D_i$, the canonical
transvection $g_a^{(i)}$ taking $0$ to $a$ acts as left gyrotranslation
by $a$ in the Friedman--Ungar gyrogroup. Thus $g_a^{(i)}(z)=a\oplus_i z$,
and the preceding identity becomes
\[
  S\circ g^{(1)}_a = g^{(2)}_{S(a)}\circ S \qquad (a\in D_1).
\]
In Theorem~\ref{thm:bsd-classification}, the tangent norm first gives a
real-linear norm isometry $T:E_1\to E_2$ in spectral coordinates.
Cube preservation and odd functional calculus then identify $S$ with
$T|_{D_1}$. For
$u\in E_i$, the one-parameter family $g^{(i)}_{\tanh(tu)}$ has
generating vector field $X^{(i)}_u(z)=u-\{z,u,z\}_i$.
The translation identity and preservation of scalar multiplication
give $Tg^{(1)}_{\tanh(tu)}(z)=g^{(2)}_{\tanh(tTu)}(Tz)$.
Differentiating in $t$ at zero gives
$T X^{(1)}_u(z)=X^{(2)}_{Tu}(Tz)$. Cancelling the terms $Tu$ gives
$T\{z,u,z\}_1=\{Tz,Tu,Tz\}_2$. Real homogeneity extends this identity
from $D_1$ to $E_1$, and polarization in the two outer variables gives
preservation of the full triple product.

\subsection{The standard-norm GGV and the Kobayashi distance}

By Kaup's Riemann mapping theorem \cite{Kaup1983}, every bounded symmetric
domain in a complex Banach space has a biholomorphic realization as the
open unit ball $D$ of a complex JB$^*$-triple $E$. With a choice of origin,
the realization may be taken to send it to $0$. We assume $E\ne0$.

For $x\in E$, the closed subtriple generated by $x$ is the norm-closed
complex linear span of its odd triple powers, defined by $x^{[1]}=x$
and $x^{[2n+1]}=\{x,x,x^{[2n-1]}\}$ for $n\ge1$. For $x\ne0$, the
triple spectrum $\operatorname{Sp}(x)$ in the convention of
\cite[\S\,1.3]{Mellon2015} is a compact subset of $[-\|x\|,\|x\|]$
symmetric about $0$. If $0$ belongs to $\operatorname{Sp}(x)$, it is an
accumulation point of $\operatorname{Sp}(x)$. Its positive part is
$S_x=\{s\in\operatorname{Sp}(x):s>0\}$.
The spectral theorem identifies the subtriple generated by $x$ with
the complex-valued continuous odd functions on $\operatorname{Sp}(x)$,
sending $x$ to the coordinate function $s\mapsto s$. Restricting an odd
function $f$ to $S_x$ determines it everywhere, since $f(-s)=-f(s)$
and $f(0)=0$ whenever $0$ belongs to the spectrum. Conversely, every
function in $C_0(S_x)$ extends in this way. If $0$ is an accumulation
point, vanishing at infinity on $S_x$ means precisely that its values
tend to zero as $s\downarrow0$ in $S_x$.
Thus the subtriple is isometrically triple isomorphic to $C_0(S_x)$,
where $x^{[2n+1]}$ corresponds to $s\mapsto s^{2n+1}$.
To apply a continuous odd function
to $x$, approximate it uniformly by odd polynomials on a symmetric
interval containing $S_x$, replace each monomial $t^{2n+1}$ by
$x^{[2n+1]}$, and take the norm limit. At $x=0$ the resulting value
is $0$.

Using this odd functional calculus, define
\begin{equation}\label{eq:bsd-coordinate-maps}
  \begin{aligned}
    F&\colon E\to D, & F(X)&\coloneqq \tanh X,\\
    L&\colon D\to E, & L(a)&\coloneqq \operatorname{arctanh}a.
  \end{aligned}
\end{equation}
Since $\tanh:\mathbb R\to(-1,1)$ and
$\operatorname{arctanh}:(-1,1)\to\mathbb R$ are inverse functions,
the functional calculus gives $L=F^{-1}$.

To see that $F$ is continuous at $X_0\in E$, choose $M>\|X_0\|$.
For every $X$ in a sufficiently small norm neighborhood of $X_0$,
we have $\|X\|\le M$, so the triple spectrum of $X$ is contained
in the same interval $[-M,M]$. Similarly, to study $L$ at $a_0\in D$,
choose $r$ with $\|a_0\|<r<1$. For $a$ sufficiently close to $a_0$
we have $\|a\|\le r$, and all these spectra lie in
$[-r,r]\subset(-1,1)$. We can thus approximate $\tanh$ on $[-M,M]$
and $\operatorname{arctanh}$ on $[-r,r]$ by odd polynomials.
For example, if $p$ approximates $\operatorname{arctanh}$ on $[-r,r]$
with uniform error at most $\varepsilon$, then the functional
calculus gives $\|L(a)-p(a)\|\le\varepsilon$ for every $\|a\|\le r$.
Each odd polynomial defines a continuous map on $E$ because it is
a finite linear combination of odd triple powers. Thus $F$ and $L$
are locally uniform limits of continuous maps and are continuous.

The coordinate $L$ identifies the scalar multiplication on $D$ with
ordinary real scalar multiplication on $E_{\mathbb R}$. To multiply
$a\in D$ by $r\in\mathbb R$, move to $L(a)$, multiply by $r$ in
$E_{\mathbb R}$, and return to $D$ using $F$. Thus
\[
  r\otimes a=F(rL(a)),\qquad L(r\otimes a)=rL(a).
\]
For instance, let $e$ be a nonzero tripotent ($\{e,e,e\}=e$).
On the complex line $\mathbb C e$, write $x=\lambda e$ with
$\lambda=te^{i\theta}$, $t\ge0$. Since the triple product is conjugate
linear in its middle variable,
\[
  x^{[3]}=\{x,x,x\}
     =\lambda\overline{\lambda}\lambda e
     =e^{i\theta}t^3e .
\]
For $t>0$, the subtriple generated by $x$ is the line $\mathbb C e$.
The triple isomorphism $\alpha e^{i\theta}e\mapsto\alpha$ sends $x$ to $t$,
so $S_x=\{t\}$ and $\operatorname{Sp}(x)=\{-t,t\}$.
Iterating the triple-power calculation gives
$x^{[2n+1]}=e^{i\theta}t^{2n+1}e$ for $n\ge0$. Thus an odd polynomial
$p$ acts by $p(te^{i\theta}e)=e^{i\theta}p(t)e$ in the triple functional
calculus. Passing to the uniform polynomial limits gives
\[
  \begin{aligned}
  F(te^{i\theta}e)&=e^{i\theta}\tanh(t)e &&(t\ge0),\\
  L(te^{i\theta}e)&=e^{i\theta}\operatorname{arctanh}(t)e &&(0\le t<1).
  \end{aligned}
\]
Since $\|e\|=1$, for $0<t<1$ the map $L$ keeps the unit direction
$e^{i\theta}e$ and changes the radius from $t$ to $\operatorname{arctanh}(t)$.

On the line $\mathbb{C}e$, $L$ acts on a single positive spectral value.
More generally, let $a\in D\setminus\{0\}$. Under the isometric
identification of the closed subtriple generated by $a$ with $C_0(S_a)$,
the element $a$ corresponds to the coordinate function $s\mapsto s$,
and $L(a)$ corresponds to $s\mapsto\operatorname{arctanh}s$.
Since $\operatorname{arctanh}$ is continuous and increasing on $[0,1)$,
we obtain
\begin{equation}\label{eq:bsd-L-norm}
  \|L(a)\|
  =\sup_{s\in S_a}\operatorname{arctanh}s
  =\operatorname{arctanh}\!\left(\sup_{s\in S_a}s\right)
  =\operatorname{arctanh}\|a\|.
\end{equation}
The identity also holds for $a=0$, since $L(0)=0$.

Recall that $g_a^{(i)}$ denotes the canonical transvection on $D_i$
taking $0$ to $a$. Since we now work with a single domain $D$, we drop
the superscript and write $g_a$ for this transvection. Thus
$g_a(z)=a\oplus z$ for all $z\in D$.
Let $\operatorname{Aut}(D)$
denote the group of biholomorphic maps from $D$ onto itself.
Let $K=\{k\in\operatorname{Aut}(D):k(0)=0\}$ be the stabilizer of the origin.
For $a,b\in D$, define
\[
  k_{a,b}\coloneqq g_{a\oplus b}^{-1}\circ g_a\circ g_b.
\]
As a composition of biholomorphic automorphisms of $D$, $k_{a,b}$ belongs
to $\operatorname{Aut}(D)$. Since $g_b(0)=b$, $g_a(b)=a\oplus b$, and
$g_{a\oplus b}(0)=a\oplus b$, we have
\[
  k_{a,b}(0)=g_{a\oplus b}^{-1}(g_a(b))
            =g_{a\oplus b}^{-1}(a\oplus b)=0.
\]
Hence $k_{a,b}\in K$. Composing the defining identity on the left with
$g_{a\oplus b}$ gives the factorization
\[
  g_a\circ g_b=g_{a\oplus b}\circ k_{a,b}.
\]
Evaluating at $z\in D$ gives
\[
  a\oplus(b\oplus z)=(a\oplus b)\oplus k_{a,b}(z).
\]
The map $k_{a,b}$ is the gyration associated with $a,b$. It accounts
for the change of parentheses in the gyroaddition.

For $z\in D\setminus\{0\}$, the holomorphic map $f:\mathbb D\to D$
given by $f(\zeta)=\zeta z/\|z\|$ satisfies $f(0)=0$ and $f(\|z\|)=z$.
By contractivity of the Kobayashi distance under holomorphic maps,
\[
  k_D(0,z)\le k_{\mathbb D}(0,\|z\|)=\operatorname{arctanh}\|z\|.
\]
For this fixed $z$, the complex Hahn--Banach theorem gives a norm-one
functional $\ell\in E^*$ such that $\ell(z)=\|z\|$. Since $\ell$ is
continuous and complex linear, the map $w\mapsto\ell(w)$ is holomorphic.
Moreover, $|\ell(w)|\leq\|w\|<1$ for every $w\in D$, so its restriction
maps $D$ into $\mathbb D$. By contractivity,
\[
  \operatorname{arctanh}\|z\|
  =k_{\mathbb D}\bigl(\ell(0),\ell(z)\bigr)
  \leq k_D(0,z).
\]
Combining the two inequalities with~\eqref{eq:bsd-L-norm}, we obtain
\begin{equation}\label{eq:bsd-origin-distance}
  k_D(0,z)=\operatorname{arctanh}\|z\|=\|L(z)\| \qquad (z\in D).
\end{equation}
At $z=0$ all three terms vanish. For two points $a,b\in D$, apply the same
biholomorphic automorphism $g_a^{-1}=g_{-a}$ to both of them
\cite{Mellon2015}. It takes $a$ to $0$, so invariance of the distance gives
\[
\begin{aligned}
  k_D(a,b)
  &=k_D\bigl(g_{-a}(a),\,g_{-a}(b)\bigr)\\
  &=k_D\bigl(0,\,g_{-a}(b)\bigr)\\
  &=\bigl\|L(g_{-a}(b))\bigr\|.
\end{aligned}
\]

\begin{lemma}[The transvection product]\label{lem:bsd-kim-lawson}
On $D$, the operation $a\oplus b=g_a(b)$ coincides with the Bruck-loop
product constructed from point symmetries by Kim and Lawson.
Consequently, $(D,\oplus)$ is a gyrocommutative gyrogroup with identity $0$.
\end{lemma}

\begin{proof}
Write $s_z$ for the point symmetry of $D$ at $z$, so $s_0=-\id_D$.
In the construction of Kim and Lawson
\cite[Proposition~3.3 and formula~(4.10)]{KimLawson2011}, the product
with base point $0$ is $s_m(s_0(b))$, where $m$ is the unique midpoint
of symmetry of $0$ and $a$, characterized by $s_m(0)=a$.

\hypertarget{bsd-transvection-flow}{Mellon}
\cite[\S1.4, pp.~206--207]{Mellon2015} gives the complete holomorphic
vector field $X_\alpha(z)=\alpha-\{z,\alpha,z\}$ at the foot of p.~206
and identifies $g_a$ on p.~207 as the time-one map $\exp(X_\alpha)$
for $\alpha=\tanh^{-1}(a)=L(a)$.
Fix $a\in D$ and put $u=L(a)$, so $F(u)=a$ by~\eqref{eq:bsd-coordinate-maps}.
Taking $F(tu)$ in place of $a$ gives $\alpha=L(F(tu))=tu$.
Since $t$ is real and the triple product is conjugate linear in its
middle variable, $X_{tu}=tX_u$. Thus
\[
  h_t\coloneqq g_{F(tu)}=\exp(X_{tu})=\exp(tX_u)
\]
is the time-$t$ map of the flow of $X_u$. Hence $h_r\circ h_t=h_{r+t}$,
$h_t(0)=F(tu)$, and $h_1=g_a$.

Put $m=h_{1/2}(0)=F(u/2)$. Conjugating $s_0$ by $h_{1/2}$ gives
$s_m=h_{1/2}\circ s_0\circ h_{-1/2}$. Since $X_u(-z)=X_u(z)$,
conjugation by $s_0$ changes $X_u$ to $-X_u$, so
$s_0\circ h_t\circ s_0=h_{-t}$. It follows that
\[
  s_m\circ s_0
  =h_{1/2}\circ(s_0\circ h_{-1/2}\circ s_0)
  =h_{1/2}\circ h_{1/2}=h_1=g_a.
\]
At $0$ this gives $s_{F(u/2)}(0)=F(u)=a$. As $u$ varies over $E$,
$w\mapsto s_w(0)$ is therefore the bijection $w\mapsto F(2L(w))$
of $D$. Thus $m$ is the unique midpoint of symmetry of $0$ and $a$.
More generally, for every $x,y\in D$ there is a unique $m\in D$
such that $s_m(x)=y$, since $g_{-x}$ sends the pair $(x,y)$ to
$(0,g_{-x}(y))$ and conjugates point symmetries to point symmetries.

Evaluating $s_m\circ s_0=g_a$ at any $b\in D$ gives
\[
  s_m(s_0(b))=g_a(b)=a\oplus b.
\]
By \cite[Proposition~3.3]{KimLawson2011}, the point-symmetry product
is a Bruck loop with identity $0$. By the equivalence of Bruck loops
and gyrocommutative gyrogroups \cite[\S3]{KimLawson2011},
$(D,\oplus)$ is therefore a gyrocommutative gyrogroup.
\end{proof}

\begin{proposition}[The standard-norm structure on the Friedman--Ungar gyrogroup]\label{prop:bsd-ggv}
Let $F$ and $L$ be the maps in~\eqref{eq:bsd-coordinate-maps}.
Equip the domain gyrogroup with $L$ as its coordinate map into the
underlying real normed space of $E$. Then
\[
 a\oplus b=g_a(b),\qquad r\otimes a=F(rL(a))
\]
give $D$ a standard-norm GGV structure. Its gyrometric satisfies
\[
  \rho(a,b)=\|L(g_{-a}(b))\|
           =\operatorname{arctanh}\|g_{-a}(b)\|=k_D(a,b).
\]
\end{proposition}

\begin{proof}
By Lemma~\ref{lem:bsd-kim-lawson}, $(D,\oplus)$ is a gyrocommutative
gyrogroup for $a\oplus b=g_a(b)$. We verify the GGV axioms with
$r\otimes a=F(rL(a))$, so that $L(r\otimes a)=rL(a)$.

\textup{(GGV0)} Fix $u,v\in D$ and write $k=\gyr[u,v]=k_{u,v}$.
Since $k\in K$, it extends to a complex-linear isometric triple
isomorphism of $E$ \cite{Mellon2015}, also denoted by $k$.
It preserves odd triple powers and their norm limits, so
$L(ka)=kL(a)$ for $a\in D$ and $F(kY)=kF(Y)$ for $Y\in E$.
It follows that $\|L(ka)\|=\|kL(a)\|=\|L(a)\|$.

\textup{(GGV1)} For every $a\in D$, we have
$1\otimes a=F(1\cdot L(a))=F(L(a))=a$, since $F=L^{-1}$.

\textup{(GGV2)} For fixed $a\in D$, let $h_t=g_{F(tL(a))}$ as in
the proof of Lemma~\ref{lem:bsd-kim-lawson}. Then $h_r\circ h_t=h_{r+t}$
and $h_t=g_{t\otimes a}$, so $h_t(0)=t\otimes a$. Evaluating the
flow identity at $0$ gives
\[
  (r\otimes a)\oplus(t\otimes a)=(r+t)\otimes a.
\]

\textup{(GGV3)} For $a\in D$ and $r_1,r_2\in\mathbb R$, we obtain
\[
  r_1\otimes(r_2\otimes a)
  =F\bigl(r_1 L(r_2\otimes a)\bigr)
  =F\bigl((r_1r_2)L(a)\bigr)
  =(r_1r_2)\otimes a.
\]

\textup{(GGV4)} Suppose $a\ne0$ and $r\ne0$. Since $L$ is injective
and $L(0)=0$, we have $L(a)\ne0$. Thus
\[
  \frac{L(|r|\otimes a)}{\|L(r\otimes a)\|}
  =\frac{|r|L(a)}{\|rL(a)\|}
  =\frac{|r|L(a)}{|r|\,\|L(a)\|}
  =\frac{L(a)}{\|L(a)\|}.
\]

\textup{(GGV5)} Let $k=\gyr[u,v]$ and $r\in\mathbb R$. By the linearity of $k$
and the identities $kF(Y)=F(kY)$ and $kL(a)=L(ka)$
from \textup{(GGV0)},
\[
  k(r\otimes a)=kF(rL(a))=F(rkL(a))=F(rL(ka))=r\otimes(ka).
\]

\textup{(GGV6)} For $a\in D$ and $r,t\in\mathbb R$, take the flow
$h_s=g_{F(sL(a))}$ from \textup{(GGV2)}. Put $x=r\otimes a$
and $y=t\otimes a$. Recall that the gyration $\gyr[x,y]$ is the map
$k_{x,y}=g_{x\oplus y}^{-1}\circ g_x\circ g_y$.
Here $g_x=h_r$ and $g_y=h_t$. By \textup{(GGV2)},
$x\oplus y=(r+t)\otimes a$, so $g_{x\oplus y}=h_{r+t}$.
Using the flow identity $h_r\circ h_t=h_{r+t}$, we obtain
\[
  \gyr[x,y]
  =h_{r+t}^{-1}\circ(h_r\circ h_t)
  =h_{r+t}^{-1}\circ h_{r+t}
  =\id_D,
\]
as required.

\textup{(GGVV)} Since $E\ne\{0\}$, choose $u\in E$ with $\|u\|=1$.
Given $t\ge0$, take $a=F(tu)\in D$. Then $\|L(a)\|=\|tu\|=t$.
Thus $\|L(D)\|=\{\,\pm\|L(a)\|:a\in D\,\}=\mathbb R$.
On this set define addition and scalar multiplication by
\[
  \alpha\oplus'\beta\coloneqq\alpha+\beta,\qquad
  r\otimes'\alpha\coloneqq r\alpha
  \qquad(\alpha,\beta,r\in\mathbb R).
\]
With these operations $\|L(D)\|$ is the usual one-dimensional real
vector space, as required by \textup{(GGVV)}.

\textup{(GGV7)} For $r\in\mathbb R$ and $a\in D$, the definition of
$\otimes'$ gives
\[
  \|L(r\otimes a)\|=\|rL(a)\|=|r|\,\|L(a)\|=|r|\otimes'\|L(a)\|.
\]

\textup{(GGV8)} For $a,b\in D$, we have $g_a(0)=a$ and
$g_a(b)=a\oplus b$. By invariance of the Kobayashi distance under $g_a$,
$k_D(a,a\oplus b)=k_D(0,b)$. The triangle inequality now gives
\[
  \begin{aligned}
    \|L(a\oplus b)\| &= k_D(0,a\oplus b)\\
    &\le k_D(0,a)+k_D(a,a\oplus b)\\
    &= k_D(0,a)+k_D(0,b)\\
    &=\|L(a)\|+\|L(b)\|=\|L(a)\|\oplus'\|L(b)\|.
  \end{aligned}
\]

The gyroinverse of $a$ is $-a$, since $g_{-a}=g_a^{-1}$.
Left invariance of the gyrometric and its value at the origin now give
\[
 \rho(a,b)=\rho(0,(-a)\oplus b)
 =\|L((-a)\oplus b)\|=k_D(a,b),
\]
where the last equality follows from the transvection formula above.
The same transvection formula for the Carath\'eodory distance is recorded
by Mackey and Mellon \cite{MackeyMellon2004}, so the two invariant
distances coincide here.

To conclude that this GGV is standard-norm, we verify the two
conditions in Definition~\ref{def:standard-norm}.

\textit{(i) Real value-line structure.}
In \textup{(GGVV)} we identified $\|L(D)\|$ with $\mathbb R$ and
defined $\oplus'$ and $\otimes'$ there to be ordinary real addition
and scalar multiplication. These are the operations used in
\textup{(GGV7)} and \textup{(GGV8)}, so this GGV has the real
value-line structure of Definition~\ref{def:standard-norm}.

\textit{(ii) Continuity of scalar multiplication for each fixed scalar.}
Fix $r\in\mathbb R$. The maps $L$ and $F$ are norm-continuous by
the functional calculus above, and $Y\mapsto rY$ is norm-continuous
on $E$. Hence their composition $S_r(z)=F(rL(z))=r\otimes z$
is norm-continuous on $D$. To pass to the gyrometric topology,
fix $a\in D$ and let $(z_n)$ be a sequence in $D$. Since
$\rho(a,z_n)=\operatorname{arctanh}\|g_{-a}(z_n)\|$, we have
\[
  \rho(a,z_n)\longrightarrow 0
  \quad\Longleftrightarrow\quad \|g_{-a}(z_n)\|\longrightarrow 0
  \quad\Longleftrightarrow\quad \|z_n-a\|\longrightarrow 0.
\]
The last equivalence holds because $g_{-a}$ is a norm homeomorphism
with $g_{-a}(a)=0$ and inverse $g_a$. Thus the norm and gyrometric
topologies on $D$ agree, so $S_r$ is continuous in the gyrometric
topology for every fixed $r$. Together with the real value-line structure
in \textit{(i)}, this proves that $(D,\oplus,\otimes,L)$ is a
standard-norm GGV in the sense of Definition~\ref{def:standard-norm}.
\end{proof}

\subsection{Global isometries and recovery of the triple product}

\smallskip
\noindent\textit{The finite-dimensional case.}\enspace
For finite-dimensional domains the global formula in
Theorem~\ref{thm:bsd-classification} follows from Lemmens and Walsh
\cite[Theorems~1.1--1.3 and~1.5]{LemmensWalsh2026}, since the
Carath\'eodory and Kobayashi distances coincide. To see this, let
$f:D_1\to D_2$ be a bijective isometry, put $b=f(0)$ and normalize
by $S=g_{-b}\circ f$. Applying their Theorem~1.1 to $S$ and $S^{-1}$
gives equality of ranks. Applying their Theorem~1.2 in both directions
gives equality of the numbers of irreducible factors. Thus, after
reordering the factors, $S$ splits into bijective isometries between
corresponding factors, each fixing zero. Each pair of factors has equal
rank by the same rank inequality, so their Theorem~1.3 makes each
component holomorphic or antiholomorphic.

By their Theorem~1.5, each holomorphic component extends to a
complex-linear triple isomorphism. For an antiholomorphic component,
apply the same result with the conjugate target triple to obtain a
conjugate-linear triple isomorphism. Together these maps give a
surjective real-linear triple isomorphism $T$ with $f=g_b\circ T$ on
$D_1$. Their results also treat distance-preserving maps that need
not be onto. Another known case, in arbitrary Hilbert dimension, is
the complex Hilbert ball, where the isometries fixing the origin are
the unitary and antiunitary operators \cite[\S\,2.B]{Monod2020}.
The proof below treats bijections in arbitrary JB$^*$-triple balls
using the gyrovector operations.

\begin{theorem}[Global Kobayashi isometries]\label{thm:bsd-classification}
Let $D_i$ be the open unit ball of a nonzero complex JB$^*$-triple
$E_i$, $i=1,2$. A map $f:D_1\to D_2$ is a bijective Kobayashi isometry
if and only if
\[
 f(z)=g_b(Tz)\qquad(z\in D_1)
\]
for some $b\in D_2$ and a surjective real-linear triple isomorphism
$T:E_1\to E_2$. Here $T$ is a real-linear bijection satisfying
$T\{x,y,z\}_1=\{Tx,Ty,Tz\}_2$ for all $x,y,z\in E_1$.
Both $b=f(0)$ and $T$ are unique.
\end{theorem}

\begin{proof}
For each pair $(E_i,D_i)$, denote the maps in~\eqref{eq:bsd-coordinate-maps}
by $F_i:E_i\to D_i$ and $L_i:D_i\to E_i$. For $u\in E_i$ and
$t\in\mathbb R$, put $h_t^u=g_{F_i(tu)}$.
By the \hyperlink{bsd-transvection-flow}{flow identity in the proof of
Lemma~\ref*{lem:bsd-kim-lawson}}, applied in $E_i$, we have
$h_t^u=\exp(tX_u)$, where $X_u(z)=u-\{z,u,z\}_i$.
Fix $z\in D_i$ and set $\gamma_z(t)=h_t^u(z)$.
Since $\exp(tX_u)$ is the time-$t$ flow of $X_u$, the curve $\gamma_z$
satisfies the initial value problem
\[
  \gamma_z'(t)=X_u(\gamma_z(t)),\qquad \gamma_z(0)=z.
\]
Evaluating at $t=0$ gives
\[
  \left.\frac{d}{dt}\right|_{t=0}h_t^u(z)
  =\gamma_z'(0)=X_u(z)=u-\{z,u,z\}_i.
\]
Since $\tanh s=s+O(s^3)$ at zero, spectral functional calculus gives
$F_i(tv)=tv+O(|t|^3\|v\|^3)$ for fixed $v$ as $t\to0$.
For fixed $u,v\in E_i$, the chain rule now gives the derivative of
$h_{-t}^u(F_i(tv))$ at $t=0$ as $-X_u(0)+v=v-u$.
Put $w_t=h_{-t}^u(F_i(tv))$. Since $h_{-t}^u(F_i(tu))=0$, biholomorphic
invariance gives $k_{D_i}(F_i(tu),F_i(tv))=k_{D_i}(0,w_t)$.
By~\eqref{eq:bsd-origin-distance} with $D=D_i$ and $z=w_t$, this equals
$\operatorname{arctanh}\|w_t\|$. Now $w_t=t(v-u)+o(t)$, and
$\operatorname{arctanh}s=s+o(s)$ as $s\to0$. Hence
\begin{equation}\label{eq:bsd-tangent-norm}
 \lim_{t\downarrow0}\frac{k_{D_i}(F_i(tu),F_i(tv))}{t}=\|u-v\|.
\end{equation}

Suppose $f$ is a bijective isometry, put $b=f(0)$ and set
$S=g_{-b}\circ f$. By Theorem~A and Proposition~\ref{prop:bsd-ggv},
$S$ preserves gyroaddition and scalar multiplication. Hence the
bijection $T=L_2\circ S\circ F_1$ satisfies $T(tu)=tT(u)$.
Applying \eqref{eq:bsd-tangent-norm} on both sides of
\[
 k_{D_2}(F_2(tTu),F_2(tTv))=k_{D_1}(F_1(tu),F_1(tv))
\]
shows that $T$ is a norm isometry. Since $T(0)=0$, the Mazur--Ulam
theorem gives real linearity.

Regard $E_1$ and $E_2$ as real JB$^*$-triples. By Isidro, Kaup and
Rodr\'iguez \cite[Theorem~4.8, pp.~330--331]{IKR1995}, a real-linear
bijection between real JB$^*$-triples is an isometry if and only if it
preserves the triple cube $x^{[3]}=\{x,x,x\}$. Since $T$ is a surjective
real-linear isometry, the implication \textup{(i)}$\Rightarrow$\textup{(ii)}
of that theorem gives
\[
  T\{x,x,x\}_1=\{Tx,Tx,Tx\}_2 \qquad (x\in E_1).
\]
Real polarization of this identity shows that $T$ preserves the fully
symmetrized product
\[
  P_i(x,y,z)=\tfrac13\bigl(\{x,y,z\}_i+\{y,z,x\}_i+\{z,x,y\}_i\bigr).
\]
Indeed, $P_i$ is symmetric and real trilinear, and its diagonal is the
triple cube. For $u\in E_i$ and $n\ge1$, each of the three terms in
$P_i(u,u,u^{[2n-1]})$ corresponds to the real function
$s\mapsto s^{2n+1}$ in the functional calculus of $u$. Thus
$P_i(u,u,u^{[2n-1]})=u^{[2n+1]}$. Starting from $u^{[1]}=u$,
induction gives $T(u^{[2n+1]})=(Tu)^{[2n+1]}$ for every $u\in E_1$.
Approximating $\tanh$ uniformly on $[-\|u\|,\|u\|]$ by real odd
polynomials and using the continuity of $T$, we obtain
$T F_1(u)=F_2(Tu)$. The definition $T=L_2\circ S\circ F_1$ also gives
$F_2(Tu)=S F_1(u)$. Since $F_1(E_1)=D_1$, it follows that $S=T|_{D_1}$.

To recover the middle variable of the original triple product, we now
use gyroaddition. Since $h_t^u(z)=F_1(tu)\oplus z$ on $D_1$ and $S$
preserves gyroaddition, the identities just proved give
\[
  T h_t^u(z)=F_2(tTu)\oplus Tz=h_t^{Tu}(Tz).
\]
Differentiating at $t=0$ gives $T X_u(z)=X_{Tu}(Tz)$. Substituting the
formulas for the vector fields and cancelling $Tu$ yields
\[
  T\{z,u,z\}_1=\{Tz,Tu,Tz\}_2 \qquad (u\in E_1,\ z\in D_1).
\]
Real scaling extends the identity to every $z\in E_1$. For $x,y\in E_1$,
subtract the identity at $z=x-y$ from that at $z=x+y$ and divide by
four. Symmetry in the two outer variables gives
$T\{x,u,y\}_1=\{Tx,Tu,Ty\}_2$, as required.

Conversely, a surjective real-linear triple isomorphism $T$ preserves
triple cubes, so \cite[Theorem~4.8, \textup{(ii)}$\Rightarrow$\textup{(i)}]{IKR1995}
shows that $T$ is an isometry. The same odd polynomial approximation
gives $T F_1=F_2 T$, and the triple identity gives
$T X_u(z)=X_{Tu}(Tz)$. Uniqueness of
their flows gives $Tg_a(z)=g_{Ta}(Tz)$ for $a,z\in D_1$. The distance
formula in Proposition~\ref{prop:bsd-ggv} gives
\[
 k_{D_2}(Tx,Ty)=\operatorname{arctanh}\|Tg_{-x}(y)\|=k_{D_1}(x,y).
\]
Composition with $g_b$ preserves this distance.

To see uniqueness, evaluate $f(z)=g_b(Tz)$ at $z=0$. Since $T(0)=0$,
we get $f(0)=g_b(0)=b$. Thus $Tz=g_{-b}(f(z))$ for every $z\in D_1$.
For any $x\in E_1$, choose $r>\|x\|$. Then $x/r\in D_1$, and real
linearity gives
\[
  Tx=r\,T(x/r)=r\,g_{-b}\bigl(f(x/r)\bigr).
\]
Hence $f$ determines both $b$ and $T$ uniquely.
\end{proof}

The complex Hilbert space $\mathbb C^2$ illustrates what the
transvections recover beyond the triple cube. The real-linear norm isometry
$R(z_1,z_2)=(\overline{z_1},z_2)$ preserves the cube $\|z\|^2z$.
For the standard unit vectors, however,
$R\{ie_1,e_1,e_2\}=ie_2/2$ and
$\{R(ie_1),R(e_1),R(e_2)\}=-ie_2/2$.

A bounded symmetric domain $D$ is called \emph{irreducible} if it is not
biholomorphic to a product $U\times V$ of two bounded symmetric domains
each containing more than one point.

\begin{corollary}[Irreducible domains]\label{cor:bsd-irreducible}
Let $D_1,D_2$ be nontrivial bounded symmetric domains in complex Banach
spaces. If at least one of them is irreducible, every bijective
Kobayashi isometry $f:D_1\to D_2$ is biholomorphic or antibiholomorphic.
\end{corollary}

\begin{proof}
After exchanging the domains and using $f^{-1}$ if necessary, assume
$D_1$ is irreducible. In the ball realizations let $f=g_b\circ T$ be
the representation from Theorem~\ref{thm:bsd-classification}. Write $J_i$
for multiplication by $i$ on $E_i$ and put $J'=T^{-1}J_2T$ on $E_1$.
In the following, the triple product is that of $E_1$. Both $J_1$ and
$J'$ move into either outer variable without a sign change and into
the middle variable with a minus sign. For $J'$ these identities follow
from the triple identity for $T$ and the corresponding identities for
$J_2$. In particular,
\[
 J'J_1\{x,y,z\}=\{J_1x,y,J'z\}=J_1J'\{x,y,z\}.
\]
Given $a\in E_1$, apply the odd functional calculus to $a$ with the
function $s\mapsto\sqrt[3]{s}$ to obtain $u\in E_1$ such that
$\{u,u,u\}=a$. Taking $x=y=z=u$ in the preceding identity gives
$J'J_1a=J_1J'a$. Thus $J'$ commutes with $J_1$ on $E_1$.

The operator $A=-J_1J'$ is therefore a bounded complex-linear involution.
The two minus signs in the middle variable cancel, so
\[
 A\{x,y,z\}=\{Ax,y,z\}=\{x,Ay,z\}=\{x,y,Az\}.
\]
The eigenspaces $E_\pm=\ker(A\mp I)=\{u\in E_1:Au=\pm u\}$ are closed
complex subspaces of $E_1$ because $A\mp I$ are bounded complex-linear
operators. To verify the triple ideal property, fix $\varepsilon\in\{-1,1\}$
and $u\in E_1$ with $Au=\varepsilon u$. For arbitrary $v,w\in E_1$,
the preceding identity gives
\[
  \begin{aligned}
    A\{u,v,w\}&=\{Au,v,w\}=\varepsilon\{u,v,w\},\\
    A\{v,u,w\}&=\{v,Au,w\}=\overline{\varepsilon}\{v,u,w\}
                          =\varepsilon\{v,u,w\}.
  \end{aligned}
\]
The second line uses conjugate linearity in the middle variable and the
fact that $\varepsilon$ is real. Symmetry in the outer variables gives
the same conclusion when $u$ is in the third position. Thus a triple
product with any entry in $E_+$ belongs to $E_+$, and likewise for $E_-$.
This proves that both eigenspaces are triple ideals. A triple product
with one entry in $E_+$ and another in $E_-$ therefore belongs to
$E_+\cap E_-=\{0\}$.

The projections $(I\pm A)/2$ give a direct sum, and the sum map from
the JB$^*$-triple $E_+\oplus_\infty E_-$ onto $E_1$ is a triple
isomorphism. On the underlying real JB$^*$-triples, the sum map is
therefore a real-linear bijection preserving triple cubes. By
\cite[Theorem~4.8, \textup{(ii)}$\Rightarrow$\textup{(i)},
pp.~330--331]{IKR1995}, it is an isometry.

Write $B_+$ and $B_-$ for the open unit balls of $E_+$ and $E_-$,
respectively. Since the direct sum carries the maximum norm, the sum map
$\Sigma(u,v)=u+v$ satisfies
\[
  \|u+v\|=\max\{\|u\|,\|v\|\} \qquad (u\in E_+,\ v\in E_-).
\]
Thus $\Sigma$ maps $B_+\times B_-$ onto $D_1$. Both $\Sigma$ and its
inverse $x\mapsto\bigl((I+A)x/2,(I-A)x/2\bigr)$ are bounded complex-linear
maps, so this restriction is biholomorphic.

If both $E_+$ and $E_-$ were nonzero, their open unit balls would each
contain more than one point. As open unit balls of complex JB$^*$-triples,
they are bounded symmetric domains. The resulting product decomposition
of $D_1$ would contradict its irreducibility. Hence one eigenspace is
zero. If $E_-=\{0\}$, then $E_1=E_+=\ker(A-I)$, so $A=I$.
If $E_+=\{0\}$, then $E_1=E_-=\ker(A+I)$, so $A=-I$.
Since $J'=J_1A$, it follows that $J'=J_1$ or $J'=-J_1$.
Hence $T$ is complex-linear or conjugate-linear, respectively.
Composition with $g_b$ gives the asserted dichotomy.
\end{proof}

\subsection{Local isometries and the underlying real triple}

For open subsets $U_i\subset D_i$ the distance used here is
$k_{D_i}|_{U_i\times U_i}$. The intrinsic distance of $U_i$ can give
a different extension problem. In the unit disk $\mathbb D$, choose
$0<r<1/2$. The biholomorphism $z\mapsto2z$ from $r\mathbb D$ onto
$2r\mathbb D$ preserves their intrinsic Kobayashi distances, yet
$k_{\mathbb D}(0,2z)>k_{\mathbb D}(0,z)$ for $0<|z|<r$.
Hence it has no isometric extension to $\mathbb D$.

For the restricted ambient distances, the sets may be arbitrarily small and
may omit the chosen origins. Preservation of this distance alone
allows Theorem~A to continue the map by gyromidpoints and reflections.
Each ambient ball $D_i$ is contractible, so the simple connectivity
required by Theorem~A holds independently of the topology of $U_i$.

\begin{corollary}[Local Kobayashi isometries]\label{cor:bsd-extension}
Let $D_1,D_2$ be nontrivial bounded symmetric domains, realized as
the open unit balls of complex JB$^*$-triples $E_1,E_2$. Let
$U_i\subset D_i$ be nonempty connected open subsets equipped with the
restricted Kobayashi distances. Every bijective isometry $f:U_1\to U_2$
has the form $f(z)=g_b(Tz)$ on $U_1$, for a unique pair consisting of
$b\in D_2$ and a surjective real-linear triple isomorphism $T:E_1\to E_2$.
The same formula on $D_1$ gives its unique global bijective isometric
extension. If at least one domain is irreducible, this extension is
biholomorphic or antibiholomorphic. In particular,
$(E_1)_{\mathbb R}$ and $(E_2)_{\mathbb R}$ are isomorphic as real
JB$^*$-triples via $T$.
\end{corollary}

\begin{proof}
By Proposition~\ref{prop:bsd-ggv}, Theorem~A applies to the given
open subsets and gives the unique global bijective isometric extension
$\widetilde f:D_1\to D_2$. Apply
Theorem~\ref{thm:bsd-classification} to $\widetilde f$, so that
$b=\widetilde f(0)$ even when $0\notin U_1$. Any other pair representing
$f$ would give a global isometric extension and hence the same pair.
Corollary~\ref{cor:bsd-irreducible} gives the holomorphic or
antiholomorphic conclusion under irreducibility.
\end{proof}

\begin{corollary}[Recognition from an open subset]\label{cor:bsd-recognition}
For the balls $D_i$ of the preceding corollary, the following are equivalent.
\begin{enumerate}[label=\textup{(\roman*)}]
\item There exist nonempty connected open subsets $U_i\subset D_i$
and a bijective isometry $U_1\to U_2$ for the restricted ambient
Kobayashi distances.
\item $D_1$ and $D_2$ are bijectively Kobayashi isometric.
\item $(E_1)_{\mathbb R}$ and $(E_2)_{\mathbb R}$ are isomorphic as
real JB$^*$-triples.
\end{enumerate}
\end{corollary}
\begin{proof}
Corollary~\ref{cor:bsd-extension} gives \textup{(i)}$\Rightarrow$\textup{(ii)}.
Theorem~\ref{thm:bsd-classification} gives
\textup{(ii)}$\Leftrightarrow$\textup{(iii)}. Restrict a global
isometry to any nonempty connected open subset to obtain \textup{(i)}.
\end{proof}

The irreducibility hypothesis in the holomorphic or antiholomorphic
conclusion of Corollary~\ref{cor:bsd-extension} is essential.
On $c_0(\mathbb N,\mathbb C)$ with its coordinatewise triple product,
let $T$ conjugate the even coordinates and fix the odd ones.
Then $T$ is a real-linear triple automorphism, so its restriction to
the open unit ball is a Kobayashi isometry by
Theorem~\ref{thm:bsd-classification}. On the first coordinate disc it
is the identity, and on the second it is complex conjugation. Thus $T$
is neither holomorphic nor antiholomorphic on the ball.

\Needspace{8\baselineskip}
\section{JB-cones and logarithmic linearization}
\label{sec:cone-linearization}

\label{sec:positive-cones}\label{sec:projective-JB}
Section~\ref{sec:bsd-kobayashi} recovered the underlying real triple
from the local Kobayashi metric. We now turn to a unital JB-algebra and ask what the Thompson metric on an open subset of its positive
invertible cone determines. Corollary~\ref{cor:local-Jordan-both}
will recover the algebra up to Jordan isomorphism. Here we construct
the positive and projective cone geometries and identify their
topologies. The infinitesimal identities at the end of this section
are used in both cone reconstruction arguments. Algebraic
axiom checks that are independent of the classification are retained
in Appendix~\ref{app:cone-details}.
\Needspace{8\baselineskip}
\subsection{JB-algebraic preliminaries}
\label{sec:JB-preliminaries}

We use the standard facts about JB-algebras in \cite{HOS}.
Throughout the JB-algebraic results, all unital JB-algebras are assumed nonzero. Let $A$ be
such an algebra with unit $e$.  We write
 \[
 A_+:=\{a^2:a\in A\}
 \]
 for its positive cone, and
 \[
 \Om:=A_+^\circ=\operatorname{int}_A(A_+)
 \]
 for the interior of this cone in the norm topology. The order is
defined by $a\le b$ if $b-a\in A_+$. This is a closed proper convex
cone \cite[Lemmas~3.3.5 and~3.3.7, pp.~82--84]{HOS}. We recall why
\[
 A_+^\circ
 =\{a\in A:a\ge\delta e\text{ for some }\delta>0\}
 =A_+\cap\operatorname{Inv}(A),
\]
where $\operatorname{Inv}(A)$ denotes the set of Jordan-invertible
elements of $A$.

For each $a\in A$, the norm-closed unital JB-subalgebra generated
by $a$ and $e$ is commutative and associative. Continuous functional
calculus identifies it isometrically with $C(\sigma(a),\mathbb R)$,
where $\sigma(a)$ is the nonempty compact real spectrum of $a$
\cite[Theorems~3.2.2 and~3.2.4]{HOS}.
An element is positive if and only if its spectrum is contained
in $[0,\infty)$. By \cite[3.2.8--3.2.10, p.~80]{HOS}, Jordan
invertibility of $a$, the condition $0\notin\sigma(a)$, and bounded
invertibility of its quadratic representation $U_a$ are equivalent.
The inverse lies in the same singly generated subalgebra and corresponds
to $\lambda\mapsto1/\lambda$ there. We also have $\|e\|=1$ and
\[
 -\|h\|e\le h\le\|h\|e\qquad(h\in A).
\]
The latter inequalities follow from
$\sigma(h)\subset[-\|h\|,\|h\|]$ by functional calculus.
These are standard spectral and order properties of JB-algebras;
see \cite{HOS}.

By spectral translation, $a\ge\delta e$ is equivalent to
$\sigma(a)\subset[\delta,\infty)$. Existence of such a $\delta>0$
is therefore equivalent to positivity and invertibility.
If $a\ge\delta e$ and $\|h\|<\delta$, then
$a+h\ge(\delta-\|h\|)e\ge0$, so $a\in A_+^\circ$.
Conversely, if $a\in A_+^\circ$, then $a-\delta e\in A_+$
for sufficiently small $\delta>0$, giving $a\ge\delta e$.

If $z$ is central and
$x\in A$, we shall write $zx$ for the Jordan product $z\circ x$.  For
 $a,b\in A_+^\circ$, put
 \begin{equation}\label{eq:order-ratio}
 M(a/b):=\inf\{\lambda>0:a\leq \lambda b\}.
 \end{equation}
 The Thompson and Hilbert metrics are respectively
 \begin{equation}\label{eq:order-metrics}
 \begin{aligned}
 d_T(a,b)&:=\log\max\{M(a/b),M(b/a)\},\\
 \dH(\bar a,\bar b)&:=\log\bigl(M(a/b)M(b/a)\bigr),
 \end{aligned}
 \end{equation}
 where $\bar a$ denotes the ray generated by $a$. The formula for $\dH$
is independent of the positive representatives and defines a metric on
$\mathbb P\Omega$. See Corollary~\ref{cor:rho-Thompson} for $d_T$ and
Proposition~\ref{prop:prho-Hilbert} for $\dH$ in dimension at least two
(the ray space in dimension one is a singleton).
The \emph{Hilbert topology} means this metric topology.
Its pullback to $\Omega$ is only a pseudometric. The superscript
$\mathrm p$ distinguishes the projective metric from this pullback,
introduced explicitly in Appendix~\ref{app:cone-lifts}.
 For
$a\in A$, let
\[
 L_a(x)=a\circ x
\]
denote Jordan multiplication (left gyrotranslations on the cones will
be written $\mathcal L_a$), and define the quadratic representation
\[
 U_a=2L_a^2-L_{a^2}.
\]
We use the fundamental formula
\begin{equation}\label{eq:fundamental}
 U_{U_ab}=U_aU_bU_a.
\end{equation}
An element $a$ is invertible if and only if $U_a$ is invertible; in that case
\begin{equation}\label{eq:Uinverse}
 U_a^{-1}=U_{a^{-1}}.
\end{equation}
If $a$ and $b$ are invertible, then
\begin{equation}\label{eq:Uab-inverse}
 (U_ab)^{-1}=U_{a^{-1}}b^{-1}.
\end{equation}
The fundamental formula is \cite[Part~II, 5.2.3 (FFI), p.~202]{McCrimmon}.
The inverse recipe $x^{-1}=U_x^{-1}x$ and the $U$-inverse formula
are given in \cite[Part~II, 6.1.2--6.1.3, p.~212]{McCrimmon}. In particular,
invertibility of $U_{U_ab}=U_aU_bU_a$ gives invertibility of $U_ab$,
and the inverse recipe yields
$(U_ab)^{-1}=(U_aU_bU_a)^{-1}U_ab=U_{a^{-1}}b^{-1}$.
For every $x\in A$, $U_x$ preserves $A_+$ by
\cite[Proposition~3.3.6, p.~83]{HOS}.
Thus, for $a\in A_+^\circ$, $U_a$ is an order automorphism, since
its inverse $U_{a^{-1}}$ also preserves $A_+$; see also
\cite[Section~2.2, equation~(2.3)]{LRW}.
Inversion reverses order on $\Omega$ \cite[Lemma~3.5.3, p.~86]{HOS}.
Indeed, if $a\le b$ in $\Omega$, then $z=U_{b^{-1/2}}a\le e$.
Functional calculus gives $z^{-1}\ge e$, and \eqref{eq:Uab-inverse}
gives $z^{-1}=U_{b^{1/2}}a^{-1}$. Applying $U_{b^{-1/2}}$ gives
$a^{-1}\ge b^{-1}$. Inversion is involutive, so for every $\lambda>0$,
\[
 a\le\lambda b\quad\Longleftrightarrow\quad b^{-1}\le\lambda a^{-1}.
\]
These identities and order properties hold also in exceptional JB-algebras.

For $a\in\Om$, the spectrum lies in $(0,\infty)$, so $a^{-1}$,
$a^r$ for $r\in\mathbb R$, and $\log a$ belong to the norm-closed
unital JB-subalgebra generated by $a$ and $e$, and $a^r\in\Om$.
In this associative subalgebra, $U_xy=x^2\circ y$. In particular,
\begin{equation}\label{eq:U-on-powers}
 U_{a^\alpha}a^\beta=a^{2\alpha+\beta}
 \qquad(\alpha,\beta\in\mathbb R).
\end{equation}

\Needspace{4\baselineskip}
\begin{lemma}\label{lem:power-U}
Let $a\in A_+^\circ$.  Then
\[
 U_{a^\alpha}U_{a^\beta}=U_{a^{\alpha+\beta}}
\]
for every $\alpha,\beta\in\mathbb R$.
\end{lemma}

\begin{proof}
See Appendix~\ref{app:JB-calculus} for the functional-calculus argument.
\end{proof}

\Needspace{4\baselineskip}
\begin{lemma}[Continuity of functional calculus]
\label{lem:JB-functional-continuity}
Let $A$ be a unital JB-algebra, let $I\subset\mathbb R$ be a compact
interval, and let $f:I\to\mathbb R$ be continuous. The map
$x\mapsto f(x)$ is norm continuous on
\[
 D_I=\{x\in A:\sigma(x)\subset I\}.
\]
In particular, $\log:A_+^\circ\to A$ and $\exp:A\to A_+^\circ$
are norm continuous, and $(r,a)\mapsto a^r$ is jointly norm continuous
on $\mathbb R\times A_+^\circ$.
\end{lemma}

\begin{proof}
See Appendix~\ref{app:JB-calculus} for the functional-calculus argument.
\end{proof}

\Needspace{4\baselineskip}
\begin{lemma}[Bidual units and extreme points of order intervals]
\label{lem:shared-bidual-projections}
The canonical embeddings of unital JB-algebras $A,B$ into $A^{**},B^{**}$
are unital Jordan embeddings; in particular their Arens units are $e_A,e_B$.
In every unital JB-algebra, the extreme points of $[0,e]$ are exactly
the projections.
\end{lemma}
\begin{proof}
Put $M=A^{**}$ and $N=B^{**}$. Their Arens products make them
unital JBW-algebras and extend the original Jordan products on $A$
and $B$ \cite[Theorem~4.4.3, p.~104]{HOS}. We identify $A,B$ with
their canonical Jordan subalgebras in $M,N$. Separate weak* continuity
and weak* density extend $e_A\circ x=x$ from $A$ to $M$, so $e_M=e_A$,
and likewise $e_N=e_B$. 
The extreme points of $[0,e]$ in a unital JB-algebra are exactly
the projections. Indeed, if $a\in[0,e]$ is not a projection, then
$a^2$ and $2a-a^2$ are distinct elements of $[0,e]$ with midpoint $a$.
Conversely, if $p^2=p$ and $p\pm h\in[0,e]$, functional calculus gives
$(p\pm h)^2\le p\pm h$. Adding yields $2h^2\le0$, so $h=0$.
\end{proof}
\Needspace{8\baselineskip}
\subsection{Gyrocommutative gyrogroup structures}

For $a,b\in\Om$, define the operation for $t=1$ by
\begin{equation}\label{eq:oplus1}
 a\oplus b=U_{a^{1/2}}b.
\end{equation}
The bounded linear map $U_{a^{1/2}}$ and its inverse $U_{a^{-1/2}}$
both preserve $A_+$. Thus $U_{a^{1/2}}$ is an order automorphism and
a norm homeomorphism taking $A_+$ onto $A_+$, and hence taking
the interior $\Om$ onto $\Om$.
The left translation $\mathcal L_a(b)=a\oplus b$ therefore satisfies
\begin{equation}\label{eq:left-translations}
 \mathcal L_a=U_{a^{1/2}},
 \qquad
 \mathcal L_a^{-1}=U_{a^{-1/2}}=\mathcal L_{a^{-1}}.
\end{equation}

\Needspace{4\baselineskip}
\begin{lemma}\label{lem:Ka-b}
Let $a,b\in\Om$, put
\[
 d=a\oplus b=U_{a^{1/2}}b,
\]
and define
\[
 K_{a,b}=U_{d^{-1/2}}U_{a^{1/2}}U_{b^{1/2}}.
\]
Then $K_{a,b}$ is a Jordan automorphism of $A$. More generally, every
Jordan isomorphism $K:A\to B$ between unital JB-algebras is isometric,
preserves spectra, maps $\Omega_A$ onto $\Omega_B$, and commutes with
continuous functional calculus, in particular with powers and logarithms
as in \eqref{eq:Kpowers} and \eqref{eq:Klog}.
\end{lemma}

\begin{proof}
Write $K=K_{a,b}$. Its three factors are invertible by
\eqref{eq:Uinverse}. Put
$K^\sharp=U_{b^{1/2}}U_{a^{1/2}}U_{d^{-1/2}}$.
Successive applications of the fundamental formula \eqref{eq:fundamental}
give
\begin{equation}\label{eq:UGz}
 U_{Kz}=K U_z K^\sharp\qquad(z\in A).
\end{equation}
Since $U_xe=x^2$ and $d=U_{a^{1/2}}b$, \eqref{eq:U-on-powers} gives
$K(e)=U_{d^{-1/2}}d=e$. Thus \eqref{eq:UGz} at $z=e$ gives
$I=KK^\sharp$, so $K^\sharp=K^{-1}$.
Since $K^{-1}e=e$, applying both sides of \eqref{eq:UGz} to $e$
for general $z$ gives
\[
 (Kz)^2=U_{Kz}e=K U_z K^{-1}e=K(z^2).
\]
Polarization now shows that $K$ preserves the Jordan product.

Every Jordan isomorphism $K:A\to B$ sends $e_A$ to $e_B$ and satisfies
$U_{Kx}=KU_xK^{-1}$, directly from the definition of $U_x$.
The Jordan inverse $y$ of $x$ is characterized by
$x\circ y=e_A$ and $x^2\circ y=x$ \cite[3.2.9, p.~80]{HOS}.
Applying $K$ to these identities gives $(Kx)\circ(Ky)=e_B$ and
$(Kx)^2\circ(Ky)=Kx$. Applying $K^{-1}$ gives the converse, so $K$
preserves invertibility in both directions. Applying this to
$x-\lambda e_A$ for $\lambda\in\mathbb R$ gives $\sigma(Kx)=\sigma(x)$.
The spectral characterization of positivity and invertibility then gives
$K(\Omega_A)=\Omega_B$. Since
$\|x\|=\max_{\lambda\in\sigma(x)}|\lambda|$, it is an isometry
(see also \cite[Proposition~3.4.3, p.~85]{HOS}).
It commutes with polynomial evaluation, and uniform polynomial
approximation on the common spectrum extends this to continuous
functional calculus. In particular, for $c\in\Om$ and $r\in\mathbb R$,
\begin{equation}\label{eq:Kpowers}
 K(c^r)=(Kc)^r,
\end{equation}
and
\begin{equation}\label{eq:Klog}
 \log(Kc)=K(\log c).
\end{equation}
\end{proof}

\Needspace{4\baselineskip}
\begin{proposition}\label{prop:gyrogroup-JB}
The set $\Om$, equipped with \eqref{eq:oplus1}, is a gyrocommutative
gyrogroup. Its identity is $e$, its gyroinverse is $\ominus a=a^{-1}$,
and its gyration is the restriction to $\Om$ of the linear operator
\begin{equation}\label{eq:gyr-JB-t1}
 \gyr[a,b]
 =U_{(a\oplus b)^{-1/2}}U_{a^{1/2}}U_{b^{1/2}}.
\end{equation}
\end{proposition}

\begin{proof}
The gyrogroup axioms are verified directly in
Appendix~\ref{app:gyrogroup-verification}, using only the quadratic
representation identities and Lemma~\ref{lem:Ka-b}.
\end{proof}

For $t>0$, define the operation
\begin{equation}\label{eq:oplus-t}
 a\oplus_t b=\bigl(U_{a^{t/2}}b^t\bigr)^{1/t}
 \qquad(a,b\in\Om).
\end{equation}
The power maps and the order automorphism $U_{a^{t/2}}$ preserve $\Om$,
so this operation is well defined, and $\oplus_1=\oplus$.
Consider
\[
 F_t:\Om\to\Om,
 \qquad F_t(a)=a^t.
\]
Then
\begin{equation}\label{eq:Ft-gyroiso}
 F_t(a\oplus_t b)=F_t(a)\oplus F_t(b).
\end{equation}
Functional calculus gives $(c^t)^{1/t}=c$ for $c\in\Om$, so
$F_t^{-1}=F_{1/t}$. This transports the gyrocommutative gyrogroup
structure of Proposition~\ref{prop:gyrogroup-JB} to $(\Om,\oplus_t)$.
The identity is $e$, since $F_t(e)=e$, and the transported gyroinverse is
\[
 \ominus_t a=F_t^{-1}\bigl((F_ta)^{-1}\bigr)
             =F_t^{-1}(a^{-t})=a^{-1}.
\]
If
\[
 d=U_{a^{t/2}}b^t=(a\oplus_t b)^t,
\]
put $K=K_{a^t,b^t}$.  Transport of the gyration gives
$\gyr_t[a,b]=F_t^{-1}\circ K\circ F_t$.
By Lemma~\ref{lem:Ka-b} and \eqref{eq:Kpowers}, for $x\in\Om$,
\[
 F_t^{-1}(K(x^t))=(K(x^t))^{1/t}=((Kx)^t)^{1/t}=Kx.
\]
Consequently,
\begin{equation}\label{eq:gyr-t}
 \gyr_t[a,b]=U_{d^{-1/2}}U_{a^{t/2}}U_{b^{t/2}},
\end{equation}
the restriction of the Jordan automorphism $K_{a^t,b^t}$.
In particular, $\gyr_t[a,b]=\gyr[a^t,b^t]$ as maps on $\Om$,
and their linear extensions to $A$ are the same operator.

\Needspace{8\baselineskip}
\subsection{The standard-norm GGV structure}

For $a\in\Om$ and $r\in\mathbb R$, put
\begin{equation}\label{eq:phi-log}
 r\otimes a=a^r,\qquad \phi(a)=\log a.
\end{equation}
Functional calculus gives $\exp(A)\subset\Om$ and $\log\circ\exp=\id_A$,
with $\exp\circ\log=\id_\Om$. For later use, \eqref{eq:U-on-powers} gives
\begin{equation}\label{eq:GGV2-powers}
 U_{a^{r_1t/2}}a^{r_2t}=a^{t(r_1+r_2)}.
\end{equation}
For the full positive invertible cone of a unital $C^*$-algebra, the
operator-norm construction appears in \cite[Example~7]{AH2015} and
is a case of \cite[Example~20]{HatoriExamples2017}.
The following construction also applies to exceptional JB-algebras.

\Needspace{4\baselineskip}
\begin{theorem}\label{thm:JB-GGV}
Let $A$ be a unital JB-algebra and $\Om=A_+^\circ$. For every $t>0$,
$(\Om,\oplus_t,\otimes,\phi)$ is a standard-norm generalized
gyrovector space. Its gyrometric is
\begin{equation}\label{eq:rho-affine-log}
 \rho_{t,A}(a,b)=\frac1t\bigl\|\log(U_{a^{-t/2}}b^t)\bigr\|
 \qquad(a,b\in\Om).
\end{equation}
Its topology coincides with the norm topology on $\Om$ for every $t>0$.
In particular, $\Om$ is contractible and simply connected.
\end{theorem}

\begin{proof}
Fix $t>0$. The gyrogroup structure is supplied by
Proposition~\ref{prop:gyrogroup-JB} and its power transport.
Appendix~\ref{app:affine-axioms} verifies (GGV0)--(GGV7) and (GGVV)
directly from functional calculus. We give the order estimate and
topological argument here.

For (GGV8), put $\alpha=\|\log a\|$ and $\beta=\|\log b\|$.
By the order--spectrum correspondence recalled in
Subsection~\ref{sec:JB-preliminaries}, functional calculus in the
respective singly generated subalgebras gives
\[
 \exp(-t\alpha)e\le a^t\le\exp(t\alpha)e,
 \qquad
 \exp(-t\beta)e\le b^t\le\exp(t\beta)e.
\]
Applying the positive linear operator $U_{a^{t/2}}$ to the bounds for
$b^t$, and using $U_{a^{t/2}}e=a^t$ and the bounds for $a^t$, yields
\[
 \exp(-t(\alpha+\beta))e
 \le d:=U_{a^{t/2}}b^t
 \le\exp(t(\alpha+\beta))e.
\]
These estimates use only scalar functional calculus for individual
elements and positivity of $U_{a^{t/2}}$. The same order--spectrum
correspondence gives
$\sigma(d)\subset[e^{-t(\alpha+\beta)},e^{t(\alpha+\beta)}]$.
The spectral mapping theorem then yields
$\sigma(\log d)\subset[-t(\alpha+\beta),t(\alpha+\beta)]$.
Since $a\oplus_t b=d^{1/t}$, the spectral norm formula gives
\[
 \|\phi(a\oplus_t b)\|
 =\frac1t\|\log d\|
 \le\alpha+\beta
 =\|\phi(a)\|+\|\phi(b)\|.
\]
Thus (GGV8) holds, and all GGV axioms hold with real value-line structure.

It remains to identify the gyrometric and prove fixed-scalar continuity.
For fixed $a\in\Om$, the left translation
\[
 \mathcal L_{a^{-1}}^{(t)}(b)
 =a^{-1}\oplus_t b
 =\bigl(U_{a^{-t/2}}b^t\bigr)^{1/t}
\]
is a norm homeomorphism of $\Om$: the power maps are inverse norm
homeomorphisms by Lemma~\ref{lem:JB-functional-continuity}, and
$U_{a^{-t/2}}$ is a bounded linear order automorphism with bounded
inverse $U_{a^{t/2}}$. This translation sends $a$ to $e$.
Also, $\log:\Om\to A$ and $\exp:A\to\Om$ are inverse norm
homeomorphisms by functional calculus and the same lemma.
By \eqref{eq:Uab-inverse},
$(a\ominus_t b)^{-1}=(U_{a^{-t/2}}b^t)^{1/t}
=\mathcal L_{a^{-1}}^{(t)}(b)$. This is the automorphic inverse identity
recalled after Definition~\ref{def:GGV}. Since $\log(x^{-1})=-\log x$,
the defining gyrometric equals $\|\log(\mathcal L_{a^{-1}}^{(t)}(b))\|$,
which gives \eqref{eq:rho-affine-log}.
Consequently, for every $a\in\Om$ and every sequence $(b_n)$ in $\Om$,
\[
 \rho_{t,A}(a,b_n)\to0
 \quad\Longleftrightarrow\quad
 \|\mathcal L_{a^{-1}}^{(t)}(b_n)-e\|\to0
 \quad\Longleftrightarrow\quad \|b_n-a\|\to0.
\]
Both topologies are metrizable, so the gyrometric and norm topologies
coincide. Finally, Lemma~\ref{lem:JB-functional-continuity} makes every
fixed-scalar map $a\mapsto a^r$ norm continuous, and hence gyrometrically
continuous. Thus $(\Om,\oplus_t,\otimes,\phi)$ is a standard-norm GGV.
Corollary~\ref{cor:contractible} gives contractibility and hence simple
connectedness, with contraction $H(a,s)=a^{1-s}$.
\end{proof}

For each fixed $\oplus_t$ and coordinate map $\log$,
Proposition~\ref{prop:scalar-uniqueness} gives uniqueness of scalar
multiplication, here $r\otimes a=a^r$ for every $t$. The metric can
vary with $t$ even though its topology is fixed, as illustrated in
Corollary~\ref{cor:rho-Thompson}.

\Needspace{8\baselineskip}
\subsection{Thompson metric and power coordinates}

For the metric comparisons, $A,B$ are unital JB-algebras with units
$e_A,e_B$, and subscripts on metrics specify the algebra.
By a $\rho_t$-isometry we mean a map
$f:D\subset\Omega_A\to\Omega_B$ satisfying
$\rho_{t,B}(f(a),f(b))=\rho_{t,A}(a,b)$ for all $a,b\in D$;
surjectivity is not assumed unless stated.
Openness and connectedness refer to the common norm and gyrometric
topology of Theorem~\ref{thm:JB-GGV}.
Here \emph{affine} refers to the cones themselves, before passage to
positive rays.

Recall the definitions of $M(x/y)$ and $d_T$ in
\eqref{eq:order-ratio}--\eqref{eq:order-metrics}, and the gyrometric
formula \eqref{eq:rho-affine-log}.

\Needspace{4\baselineskip}
\begin{corollary}[Thompson's metric in power coordinates]\label{cor:rho-Thompson}
For every $t>0$ and $a,b\in A_+^\circ$,
\begin{equation}\label{eq:rho-Thompson}
 \rho_{t,A}(a,b)
 =\frac1t\bigl\|\log\bigl(U_{a^{-t/2}}b^t\bigr)\bigr\|
 =\frac1t\dT(a^t,b^t).
\end{equation}
In particular, $\rho_{1,A}=\dT$.
\end{corollary}

\begin{proof}
For $x,y\in A_+^\circ$, put $z=U_{x^{-1/2}}y$. The order isomorphism
$U_{x^{-1/2}}$ sends $x$ to $e_A$, so functional calculus gives
\[
 M(y/x)=\max\sigma(z),\qquad M(x/y)=1/\min\sigma(z).
\]
The spectral norm formula in Subsection~\ref{sec:JB-preliminaries} gives
$\dT(x,y)=\|\log z\|$, also the expression in \cite[Proposition~2.4]{LRW}.
Taking $x=a^t$, $y=b^t$ in this identity and using \eqref{eq:rho-affine-log}
gives \eqref{eq:rho-Thompson}. At $t=1$, Lemma~\ref{lem:real-value-line-rays}
and Theorem~\ref{thm:JB-GGV} show that $\dT$ is a metric inducing the
norm topology.
\end{proof}

If $a,b$ lie in a common norm-closed associative unital JB-subalgebra
(equivalently, they operator commute \cite[Theorem~3.13]{vandeWetering}), then
$U_{a^{-t/2}}b^t=(U_{a^{-1/2}}b)^t$, so $\rho_{t,A}(a,b)=\dT(a,b)$
for all $t>0$.

For $0<s<t$, the power $x\mapsto x^{s/t}$ preserves order on the positive
invertible cone of every unital JB-algebra \cite[Proposition~5]{WangWang2021}.
Hence $b^t\le\lambda a^t$ implies $b^s\le\lambda^{s/t}a^s$.
Taking infima gives $M(b^s/a^s)\le M(b^t/a^t)^{s/t}$. Interchanging
$a,b$ and taking the maximum of the two ratios in \eqref{eq:rho-Thompson}
yields $\rho_{s,A}(a,b)\le\rho_{t,A}(a,b)$. The inequality can be strict
even for real symmetric matrices.{\interfootnotelinepenalty=10000\relax
\footnote{\label{fn:strict-power-metrics}For $a=\operatorname{diag}(2,1/2)$
and $b=\left(\begin{smallmatrix}2&1\\1&1\end{smallmatrix}\right)$,
the matrix $a^{-t/2}b^ta^{-t/2}$ has determinant one and reciprocal
eigenvalues, so its trace is $2\cosh(t\rho_{t,A}(a,b))$.
Since $\operatorname{Tr}(a^{-1}b)=3$ and $\operatorname{Tr}(a^{-2}b^2)=37/4$,
the double-angle identity gives
$2\cosh(2\rho_{1,A}(a,b))=3^2-2=7<37/4=2\cosh(2\rho_{2,A}(a,b))$.}}

\Needspace{4\baselineskip}
\begin{proposition}\label{prop:power-conjugacy}
For $t>0$, write the power maps as
\[
 F_{t,A}(a)=a^t,\qquad F_{t,B}(b)=b^t.
\]
For a subset $D\subset A_+^\circ$ and a map $f:D\to B_+^\circ$, put
\[
 H=F_{t,B}\circ f\circ F_{t,A}^{-1}:F_{t,A}(D)\to B_+^\circ.
\]
Then $f$ is a $\rho_t$-isometry if and only if $H$ is a Thompson
isometry. In particular, when $D=A_+^\circ$, $f$ is a bijective
$\rho_t$-isometry if and only if $H$ is a bijective Thompson isometry.
\end{proposition}

\begin{proof}
The power maps are bijections with inverses $F_{1/t,A}$ and $F_{1/t,B}$.
By Corollary~\ref{cor:rho-Thompson}, for $a,b\in D$,
\[
\begin{aligned}
 d_{T,B}(H(a^t),H(b^t))&=t\,\rho_{t,B}(f(a),f(b)),\\
 d_{T,A}(a^t,b^t)&=t\,\rho_{t,A}(a,b).
\end{aligned}
\]
The common factor $t>0$ cancels, proving the equivalence of distance
preservation. When $D=A_+^\circ$, the bijectivity of the power maps
also gives the equivalence of bijectivity onto $B_+^\circ$.
\end{proof}

\Needspace{8\baselineskip}
\subsection{Projective GGV structure}
\label{subsec:projective-GGV}

For a unital JB-algebra $A$, put $\Omega=A_+^\circ$. Its positive ray space is
\[
 \mathbb P\Omega=\Omega/\mathbb R_{>0},\qquad
 \bar a=\{\alpha a:\alpha>0\}.
\]
Write $\pi_A:\Omega\to\mathbb P\Omega$, $\pi_A(a)=\bar a$, for the
canonical projection, and use analogous notation for other algebras.
We also use the quotient
\[
 [A]=A/\mathbb Re,\qquad [x]=x+\mathbb Re
\]
as a real vector space with the usual quotient operations and zero $[0]$.
The two quotients are related by
\[
 \log(\alpha a)=(\log\alpha)e+\log a
 \qquad(\alpha>0,\ a\in\Omega).
\]
Thus positive rescaling becomes translation by a multiple of $e$ in
logarithmic coordinates; this does not identify gyroaddition with
ordinary addition.

For $x\in A$, let $\lambda_-=\min\sigma(x)$, $\lambda_+=\max\sigma(x)$,
and $\mu(x)=(\lambda_++\lambda_-)/2$. Define the variation norm by
\begin{equation}\label{eq:variation-norm}
\begin{aligned}
 \|[x]\|_v&=\operatorname{diam}\sigma(x)=\lambda_+-\lambda_-\\
 &=2\inf_{\lambda\in\mathbb R}\|x-\lambda e\|
 =2\|x-\mu(x)e\|.
\end{aligned}
\end{equation}
Indeed, $\sigma(x+\lambda e)=\sigma(x)+\lambda$ for $\lambda\in\mathbb R$,
so the spectral width is independent of the representative.
The spectral norm formula gives
$\|x-\lambda e\|=\max\{|\lambda_+-\lambda|,|\lambda_--\lambda|\}
\ge(\lambda_+-\lambda_-)/2$,
with equality at $\lambda=\mu(x)$. Thus $\|\cdot\|_v$ is twice the
usual quotient norm. Since $\mathbb Re$ is closed, $([A],\|\cdot\|_v)$
is a Banach space with the quotient topology (see also \cite{LRW,RW}).
On $A$ the same spectral width is a seminorm with kernel $\mathbb Re$.

For $t>0$, write
\[
 q_t(a,b)=a\oplus_t b=\bigl(U_{a^{t/2}}b^t\bigr)^{1/t}
 \qquad(a,b\in\Omega).
\]
Define
\begin{equation}\label{eq:projective-operations}
 \bar a\poplus_t\bar b=\overline{q_t(a,b)},\qquad
 r\potimes\bar a=\overline{a^r}\quad(r\in\mathbb R),
\end{equation}
and
\begin{equation}\label{eq:projective-Phi}
 \Phi_A:\mathbb P\Omega\longrightarrow[A],\qquad
 \Phi_A(\bar a)=[\log a].
\end{equation}

\Needspace{4\baselineskip}
\begin{lemma}\label{lem:projective-well-defined}
The operations in \eqref{eq:projective-operations} and the map
\eqref{eq:projective-Phi} are well defined. Moreover, $\Phi_A$ is a bijection.
\end{lemma}

\begin{proof}
Theorem~\ref{thm:JB-GGV} gives $q_t(a,b),a^r\in\Omega$.
For $\alpha,\beta>0$, functional calculus, $U_{\lambda x}=\lambda^2U_x$,
and linearity of $U_x$ give
\[
 q_t(\alpha a,\beta b)=\alpha\beta q_t(a,b),\qquad
 (\alpha a)^r=\alpha^r a^r\quad(r\in\mathbb R).
\]
These prove independence of representatives for both operations.
The logarithmic scaling identity above does the same for $\Phi_A$.
Since $\exp(x+\lambda e)=e^\lambda\exp x$ and $\exp,\log$ are inverse
bijections, the well-defined map
\[
 \Psi_A:[A]\longrightarrow\mathbb P\Omega,\qquad
 \Psi_A([x])=\overline{\exp x}
\]
is the inverse of $\Phi_A$.
\end{proof}

\Needspace{4\baselineskip}
\begin{theorem}\label{thm:projective-JB-GGV}
Let $A$ be a unital JB-algebra of dimension at least two. For every $t>0$,
\[
 (\mathbb P\Omega,\poplus_t,\potimes,\Phi_A)
\]
is a standard-norm generalized gyrovector space. It is contractible
and simply connected. Its gyrometric topology is the topology transported
by $\Phi_A$ from $([A],\|\cdot\|_v)$. Equivalently, any state $\omega$
identifies it with the norm topology on the section $\omega(a)=1$.
In particular the topology is independent of $t$ and of the chosen state.
\end{theorem}

\begin{proof}
Fix $t>0$. The algebraic verification in
Appendix~\ref{app:projective-axioms} gives the projective gyrogroup,
with identity $\bar e$, inverse $\overline{a^{-1}}$, and gyrations
\[
 \gyr_t^{\mathrm p}[\bar a,\bar b]\bar c
 =\overline{K_{a^t,b^t}c}.
\]
The operator $K_{a^t,b^t}$ is independent of the positive
representatives of $\bar a,\bar b$. That verification also gives
(GGV0)--(GGV7), (GGVV), and the coordinate identity
\[
 \Phi_A(r\potimes\bar a)=r\Phi_A(\bar a).
\]
We prove (GGV8) and the topological assertions here.

For (GGV8), choose $\widetilde a=e^{-\mu(\log a)}a$ and
$\widetilde b=e^{-\mu(\log b)}b$. These represent $\bar a,\bar b$,
and \eqref{eq:variation-norm} gives
$2\|\log\widetilde a\|=\|[\log a]\|_v$ and the analogous equality for $b$.
Using $\|[z]\|_v\le2\|z\|$, well-definedness, and affine (GGV8) from
Theorem~\ref{thm:JB-GGV},
\[
\begin{aligned}
 \|\Phi_A(\bar a\poplus_t\bar b)\|_v
 &=\|[\log(\widetilde a\oplus_t\widetilde b)]\|_v\\
 &\le2\|\log(\widetilde a\oplus_t\widetilde b)\|\\
 &\le2(\|\log\widetilde a\|+\|\log\widetilde b\|)\\
 &=\|\Phi_A(\bar a)\|_v+\|\Phi_A(\bar b)\|_v.
\end{aligned}
\]
Thus all GGV axioms hold. Write $\prho_{t,A}$ for the resulting
gyrometric.

For the topology choose a state $\omega$ on $A$. Such a functional exists:
Hahn--Banach extends $\lambda e\mapsto\lambda$ with norm one, and
$0\le x\le e$ gives $\|e-x\|\le1$, hence $\omega(x)=1-\omega(e-x)\ge0$.
By scaling this proves positivity on $A_+$ (see also
\cite[Lemmas~1.2.2 and~1.2.4, p.~8]{HOS}). Put
\[
 \Sigma_\omega=\{a\in\Omega:\omega(a)=1\},\qquad
 s_\omega(\bar a)=\frac{a}{\omega(a)}.
\]
For $a\in\Omega$, the bound $a\ge\varepsilon e$ for some $\varepsilon>0$
gives $\omega(a)\ge\varepsilon>0$; faithfulness is not required.
Thus each ray has exactly one representative in $\Sigma_\omega$,
and $s_\omega$ is a bijection with inverse $\pi_A|_{\Sigma_\omega}$.
Compare the gyrometric topology with the topology transported by
$s_\omega$ from the norm topology on $\Sigma_\omega$.

At the identity, $\prho_{t,A}(\bar e,\bar b)=\|[\log b]\|_v$.
For $b\in\Sigma_\omega$, write $\lambda_-,\lambda_+$ for its spectral
endpoints. The bounds $\lambda_-e\le b\le\lambda_+e$ give
$0<\lambda_-\le1\le\lambda_+$, so
\[
\begin{aligned}
 \|\log b\|&=\max\{-\log\lambda_-,\log\lambda_+\}\\
 &\le\log\lambda_+-\log\lambda_-=\|[\log b]\|_v\le2\|\log b\|.
\end{aligned}
\]
By continuity of $\log,\exp$ in Lemma~\ref{lem:JB-functional-continuity},
for sequences $(b_n)$ in $\Sigma_\omega$,
\[
 \prho_{t,A}(\bar e,\bar b_n)\to0
 \quad\Longleftrightarrow\quad
 \|\log b_n\|\to0
 \quad\Longleftrightarrow\quad
 \|b_n-e\|\to0.
\]
Normalization is needed to deduce norm convergence from gyrometric
convergence: the unnormalized elements $ne$ all represent $\bar e$.

For fixed $a\in\Omega$, the projective left translation
$\mathcal L_{\bar a}^{(t)}$ is represented on $\Sigma_\omega$ by
\[
 b\longmapsto\frac{q_t(a,b)}{\omega(q_t(a,b))}.
\]
This is norm continuous by Lemma~\ref{lem:JB-functional-continuity},
bounded linearity of $U_{a^{t/2}}$, and positivity of the denominator.
The left inverse property represents its inverse by the same formula
with $a^{-1}$ in place of $a$, so it is a norm homeomorphism.
Projective left translations are gyrometric isometries by
\cite[Proposition~3.5]{AHrevisited}. For any $\bar a\in\mathbb P\Omega$,
put $\tau=\mathcal L_{\overline{a^{-1}}}^{(t)}$, so $\tau\bar a=\bar e$.
For every sequence $(\bar b_n)$, the equivalence at $\bar e$ and the
norm-homeomorphism property of $\tau$ give
\[
\begin{aligned}
 \prho_{t,A}(\bar a,\bar b_n)\to0
 &\quad\Longleftrightarrow\quad
 \|s_\omega(\tau\bar b_n)-e\|\to0\\
 &\quad\Longleftrightarrow\quad
 \|s_\omega(\bar b_n)-s_\omega(\bar a)\|\to0.
\end{aligned}
\]
Since both topologies are metrizable, their agreement on convergent
sequences at every point proves equality. Hence $s_\omega$ identifies the
$\prho_{t,A}$ topology with the norm topology on $\Sigma_\omega$
for every $t>0$.

On this section scalar multiplication is represented by
\begin{equation}\label{eq:section-scalar}
 S_\omega(r,a)=\frac{a^r}{\omega(a^r)}
 \qquad(r\in\mathbb R,\ a\in\Sigma_\omega).
\end{equation}
Lemma~\ref{lem:JB-functional-continuity} and positivity of the denominator
make $S_\omega$ jointly norm continuous. Through $s_\omega$, this proves joint
gyrometric continuity of projective scalar multiplication, and hence
the standard-norm property.
Corollary~\ref{cor:contractible} gives contractibility and simple
connectedness, with $H(\bar a,s)=\overline{a^{1-s}}$.
On $\Sigma_\omega$ the coordinate map $a\mapsto[\log a]$ is continuous
by continuity of $\log$ and the quotient map. The bounded linear map
$x\mapsto x-\omega(x)e$ vanishes on $\mathbb Re$ and so induces a continuous
linear section $j_\omega$ of the quotient map. More explicitly, choose
a representative $x$ with $\mu(x)=0$. Since $|\omega(x)|\le\|x\|$,
\eqref{eq:variation-norm} gives
\[
 \|j_\omega([x])\|=\|x-\omega(x)e\|
 \le2\|x\|=\|[x]\|_v.
\]
The inverse $\Phi_A^{-1}$ is represented on $\Sigma_\omega$ by
\[
 [x]\longmapsto
 \frac{\exp(j_\omega([x]))}{\omega(\exp(j_\omega([x])))}.
\]
This is continuous by Lemma~\ref{lem:JB-functional-continuity} and
positivity of the denominator. Thus $\Phi_A$ is a homeomorphism
between the gyrometric topology on $\mathbb P\Omega$ (equivalently,
the norm topology transported by $s_\omega$) and the variation-norm
topology on $[A]$.
\end{proof}

For maps between projective cones, the subscript again specifies
the algebra: a $\prho_t$-isometry satisfies
$\prho_{t,B}(f(\bar a),f(\bar b))=\prho_{t,A}(\bar a,\bar b)$
for all points of its domain; surjectivity is assumed only when stated.
\Needspace{8\baselineskip}
\subsection{Hilbert geometry and power conjugacy}
\label{subsec:Hilbert-global}

Recall the projective gyrometric $\prho_{t,A}$ constructed in
Theorem~\ref{thm:projective-JB-GGV}. We derive its relation to
Hilbert's metric from \eqref{eq:order-ratio}--\eqref{eq:order-metrics};
see also \cite[Proposition~2.4]{LRW}.

\Needspace{4\baselineskip}
\begin{proposition}\label{prop:prho-Hilbert}
Let $A$ be a unital JB-algebra of dimension at least two.
For every $t>0$ and $a,b\in A_+^\circ$,
\begin{equation}\label{eq:prho-Hilbert}
 \prho_{t,A}(\bar a,\bar b)
 =\frac1t\bigl\|[\log\bigl(U_{a^{-t/2}}b^t\bigr)]\bigr\|_v
 =\frac1t\dH(\overline{a^t},\overline{b^t}).
\end{equation}
For every $t>0$, the metric $\prho_{t,A}$ induces the Hilbert topology.
\end{proposition}

\begin{proof}
By Theorem~\ref{thm:projective-JB-GGV}, $\ominus_t\bar b=\overline{b^{-1}}$.
Put $z=U_{a^{-t/2}}b^t$. Since inversion commutes with the $1/t$ power
by functional calculus, \eqref{eq:Uab-inverse} gives
$(a\oplus_t b^{-1})^{-1}=z^{1/t}$. By the definition of the gyrometric
and \eqref{eq:projective-Phi},
\[
 \prho_{t,A}(\bar a,\bar b)
 =\|[\log(a\oplus_t b^{-1})]\|_v
 =\frac1t\|[-\log z]\|_v=\frac1t\|[\log z]\|_v.
\]
For $x,y\in\Omega$, set $z=U_{x^{-1/2}}y$,
$\lambda_-=\min\sigma(z)>0$, and $\lambda_+=\max\sigma(z)$.
The order calculation in Corollary~\ref{cor:rho-Thompson} gives
$M(y/x)=\lambda_+$ and $M(x/y)=1/\lambda_-$. Thus
\begin{equation}\label{eq:Hilbert-JB-formula}
 \dH(\bar x,\bar y)
 =\log(\lambda_+/\lambda_-)=\operatorname{diam}\sigma(\log z)=\|[\log z]\|_v.
\end{equation}
Taking $x=a^t$, $y=b^t$ proves \eqref{eq:prho-Hilbert}.
In particular $\prho_{1,A}=\dH$, so $\dH$ is a metric on rays by
Lemma~\ref{lem:real-value-line-rays}. Its topology is the common
topology of Theorem~\ref{thm:projective-JB-GGV}.
\end{proof}

For $0<s<t$, multiplying the two order-ratio bounds following
Corollary~\ref{cor:rho-Thompson} and applying \eqref{eq:order-metrics}
and \eqref{eq:prho-Hilbert} gives
$\prho_{s,A}(\bar a,\bar b)\le\prho_{t,A}(\bar a,\bar b)$.
For the determinant-one $2\times2$ matrices in
footnote~\ref{fn:strict-power-metrics}, the relative matrix has reciprocal
eigenvalues. Its logarithmic spectral width is twice its logarithmic
spectral norm, so $\prho_{t,A}(\bar a,\bar b)=2\rho_{t,A}(a,b)$.
At $t=1,2$ the respective values are $2\operatorname{arcosh}(3/2)$
and $\operatorname{arcosh}(37/8)$. Thus the common homeomorphism $\Phi_A$
need not give the distance between two rays by taking the norm of their
coordinate difference.

\Needspace{4\baselineskip}
\begin{proposition}[Power conjugacy for the projective family]
\label{prop:projective-power-conjugacy}
Let $A,B$ be unital JB-algebras of dimension at least two, and let $t>0$.
Put $F_{t,A}^{\mathrm p}(\bar a)=\overline{a^t}$ and define
$F_{t,B}^{\mathrm p}$ similarly. For $D\subset\mathbb P A_+^\circ$
and a map $f:D\to\mathbb P B_+^\circ$, define
\[
 H=F_{t,B}^{\mathrm p}\circ f\circ(F_{t,A}^{\mathrm p})^{-1}.
\]
Then $f$ is a $\prho_t$-isometry if and only if $H$ is a Hilbert isometry
on $F_{t,A}^{\mathrm p}(D)$. For $D=\mathbb P A_+^\circ$, the same
equivalence holds for bijective isometries onto the entire target.
\end{proposition}

\begin{proof}
By Lemma~\ref{lem:projective-well-defined}, the projective power maps
are well defined. Their inverses are the projective $1/t$ power maps.
By Proposition~\ref{prop:prho-Hilbert}, $F_{t,A}^{\mathrm p}$ is an
isometric bijection from $(\mathbb P A_+^\circ,t\prho_{t,A})$ onto
$(\mathbb P A_+^\circ,d_{H,A}^{\mathrm p})$, and likewise for $B$.
Thus, for $\bar a,\bar b\in D$,
\[
\begin{aligned}
 d_{H,B}^{\mathrm p}\bigl(H(F_{t,A}^{\mathrm p}\bar a),
                         H(F_{t,A}^{\mathrm p}\bar b)\bigr)
 &=t\prho_{t,B}(f(\bar a),f(\bar b)),\\
 d_{H,A}^{\mathrm p}(F_{t,A}^{\mathrm p}\bar a,F_{t,A}^{\mathrm p}\bar b)
 &=t\prho_{t,A}(\bar a,\bar b).
\end{aligned}
\]
The common factor $t>0$ cancels. The bijectivity of the power maps
also gives the asserted equivalence of bijectivity. Projecting
\eqref{eq:Ft-gyroiso} shows that $F_{t,A}^{\mathrm p}$ intertwines
$\poplus_t$ with $\poplus_1$. It preserves scalar multiplication because
$(a^r)^t=(a^t)^r$, and $\Phi_A F_{t,A}^{\mathrm p}=t\Phi_A$.
\end{proof}

\Needspace{8\baselineskip}
\subsection{Logarithmic linearization and infinitesimal gyrations}
\label{subsec:linearization}

\Needspace{4\baselineskip}
\begin{lemma}[Infinitesimal gyrations on a JB-cone]
\label{lem:JB-infinitesimal-gyration}
Let $C$ be a unital JB-algebra. Write $L_xy=x\circ y$ and
$D_{x,y}=[L_x,L_y]$. For each $w\in C$ one has
$U_{\exp(w/2)}=\exp(L_w)$. For $x,y\in C$,
\[
 \log\bigl(U_{\exp(hx/2)}\exp(hy)\bigr)=h(x+y)+O(h^3)
 \quad(h\to0),
\]
where the estimate is in the norm of $C$ for fixed $x,y$.
The Jordan automorphisms
$K_{a,b}$ of Lemma~\ref{lem:Ka-b} satisfy
\begin{equation}\label{eq:infinitesimal-gyration}
 K_{\exp(hx),\exp(hy)}
 =I+\frac{h^2}{2}D_{x,y}+O(h^3)
 \qquad(h\to0)
\end{equation}
in operator norm.
\end{lemma}

\begin{proof}
Set $R(h)=U_{\exp(hx/2)}$. Lemma~\ref{lem:power-U} gives
$R(h+k)=R(h)R(k)$, while $U_a=2L_a^2-L_{a^2}$ gives $R'(0)=L_x$.
The identity $R(h+k)-R(h)=R(h)(R(k)-I)$ gives $R'(h)=R(h)L_x$.
Hence $R(h)\exp(-hL_x)$ has zero derivative and value $I$ at $0$.
This proves $U_{\exp(w/2)}=\exp(L_w)$ for every $w\in C$. Put
$d_h=U_{\exp(hx/2)}\exp(hy)$. Expanding in norm gives
\[
 d_h=e+h(x+y)+\frac{h^2}{2}(x+y)^2+O(h^3).
\]
Since $d_h\to e$, we have $\|d_h-e\|<1$ for sufficiently small
$|h|$, so the logarithm series converges in norm.

The quadratic terms in
$\log d_h=(d_h-e)-(d_h-e)^2/2+O(h^3)$ cancel, so
$\log d_h=h(x+y)+O(h^3)$. Applying the exponential identity to
$w=-\log d_h$ gives $U_{d_h^{-1/2}}=\exp(-L_{\log d_h})$.
Put $X=L_x$, $Y=L_y$, and $Z=X+Y$. Since $\|L_w\|\le\|w\|$,
we have $L_{\log d_h}=hZ+O(h^3)$ in operator norm. Hence
\[
\begin{aligned}
 \exp(-L_{\log d_h})
 &=I-hZ+\frac{h^2}{2}Z^2+O(h^3),\\
 \exp(hX)\exp(hY)
 &=I+hZ+h^2\left(\frac12X^2+XY+\frac12Y^2\right)+O(h^3).
\end{aligned}
\]
By \eqref{eq:gyr-JB-t1}, multiplying these expansions yields
\[
\begin{aligned}
 K_{\exp(hx),\exp(hy)}
 &=\exp(-L_{\log d_h})\exp(hX)\exp(hY)\\
 &=I+h^2\left(\frac12X^2+XY+\frac12Y^2-\frac12Z^2\right)+O(h^3)\\
 &=I+\frac{h^2}{2}(XY-YX)+O(h^3),
\end{aligned}
\]
which proves the assertion, with all remainders in operator norm.
\end{proof}

For fixed $a,b\in A_+^\circ$, apply Lemma~\ref{lem:JB-infinitesimal-gyration}
with $x=-\log a$, $y=\log b$ and $h=t$. Dividing the logarithmic expansion
by $t$ and taking norms in \eqref{eq:rho-Thompson} gives
\[
 \rho_{t,A}(a,b)=\|\log b-\log a\|+O(t^2)\qquad(t\downarrow0).
\]

We recover the Jordan structure from a normalized projective GGV
isometry in two stages. Lemma~\ref{lem:projective-log-triple} obtains a
linear variation-norm isometry preserving the quotient Lie triple
operation. Proposition~\ref{prop:projective-Jordan-reconstruction} uses
this additional identity to construct the Jordan isomorphism directly.
For a unital JB-algebra $C$, use the operators $L_x$ and $D_{x,y}$ of
Lemma~\ref{lem:JB-infinitesimal-gyration}.
Since $D_{x,e}=D_{e,y}=0$ and $D_{x,y}e=0$, the element
$D_{x,y}z\in C$ itself depends only on $[x],[y],[z]$.

\Needspace{4\baselineskip}
\begin{lemma}[Logarithmic linearization and infinitesimal gyrations]
\label{lem:projective-log-triple}
Let $A,B$ be unital JB-algebras of dimension at least two, and let
$g:\mathbb P A_+^\circ\to\mathbb P B_+^\circ$ be a bijective
$\prho_1$-isometry with $g(\bar e_A)=\bar e_B$. Then
\[
 S=\Phi_B\circ g\circ\Phi_A^{-1}:[A]\longrightarrow[B]
\]
is a surjective real linear variation-norm isometry. Moreover,
\begin{equation}\label{eq:quotient-Lie-triple}
 S([D_{x,y}z])=[D_{u,v}w],
 \qquad [u]=S[x],\quad[v]=S[y],\quad[w]=S[z].
\end{equation}
\end{lemma}

\begin{proof}
Theorem~\ref{thm:projective-JB-GGV} supplies GGVs, and $g$ is a
surjection preserving their gyrometrics, as required by
\cite[Corollary~4.5]{AHrevisited}. That corollary gives preservation
of gyroaddition and all real scalar multiplications after a left
gyrotranslation. Since $g(\bar e_A)=\bar e_B$, the translation is
the identity, so $g$ itself preserves both operations.
Thus $S$ is real homogeneous, continuous, and bijective; continuity
follows from the homeomorphisms $\Phi_A,\Phi_B$.
The radial identity
\[
 \prho_{1,A}(\bar e_A,\bar a)=\|\Phi_A(\bar a)\|_v
\]
and its counterpart for $B$ give $\|S\xi\|_v=\|\xi\|_v$.

For $C=A,B$, transport gyroaddition through $\Phi_C$ by setting
\[
 \xi\star\eta
 =\Phi_C\bigl(\Phi_C^{-1}(\xi)\poplus_1\Phi_C^{-1}(\eta)\bigr).
\]
This defines an operation on $[C]$ independently of representatives;
in logarithmic coordinates it is
\[
 [x]\star[y]=[\log(U_{\exp(x/2)}\exp y)].
\]
Lemma~\ref{lem:JB-infinitesimal-gyration} gives
$\log(U_{\exp(hx/2)}\exp(hy))=h(x+y)+O(h^3)$ for fixed $x,y$.
Since $\|[z]\|_v\le2\|z\|$, taking $h=1/n$ gives
$n((\xi/n)\star(\eta/n))\to\xi+\eta$ in variation norm.
Homogeneity, continuity, and preservation of $\star$ now give
\[
\begin{aligned}
 S(\xi+\eta)
 &=\lim_{n\to\infty}S\bigl(n((\xi/n)\star(\eta/n))\bigr)\\
 &=\lim_{n\to\infty}n\bigl((S\xi/n)\star(S\eta/n)\bigr)=S\xi+S\eta.
\end{aligned}
\]
Together with real homogeneity this proves linearity of $S$;
its norm preservation then makes it isometric.

A gyroaddition isomorphism intertwines gyrations. Here $t=1$, so
$K^{(1)}_{a,b}=K_{a,b}$. By Theorem~\ref{thm:projective-JB-GGV} the
projective gyration is independent of representatives, so any lifts
$u,v,w$ in \eqref{eq:quotient-Lie-triple} may be used. Each $K$ is a
Jordan automorphism and commutes with $\exp$ and $\log$, so in
logarithmic coordinates this says
\[
 S([K_{\exp(hx),\exp(hy)}z])
 =[K_{\exp(hu),\exp(hv)}w].
\]
Write the remainder in \eqref{eq:infinitesimal-gyration} from
Lemma~\ref{lem:JB-infinitesimal-gyration} as $R_h$, with
$\|R_h\|_{\mathrm{op}}=O(|h|^3)$. For fixed $z\in A$ and $h\ne0$,
the quotient-norm bound in \eqref{eq:variation-norm} gives
\[
 \left\|\left[\frac{2R_hz}{h^2}\right]\right\|_v
 \le\frac{4\|R_h\|_{\mathrm{op}}\|z\|}{h^2}
 =O(|h|)\longrightarrow0.
\]
The same estimate applies in $B$. Since $S$ is bounded and linear,
subtracting $S[z]=[w]$, dividing by $h^2/2$, and letting $h\to0$
proves \eqref{eq:quotient-Lie-triple}.
\end{proof}

\Needspace{8\baselineskip}
\section{Reconstruction of the Jordan structure}
\label{sec:reconstruction}

\label{subsec:projective-reconstruction}

\Needspace{4\baselineskip}
\begin{remark}[Relation to the Hilbert classification at $t=1$]\label{rem:hilbert-classification-relation}
At $t=1$, Theorem~\ref{thm:projective-global-classification} recovers the known Hilbert-isometry
classification for unital JB-algebras of Roelands and Wortel
\cite[Theorem~5.1]{RW}, extending the JBW-algebra result
\cite[Theorem~4.21]{LRW}.
For general $t>0$, Proposition~\ref{prop:projective-power-conjugacy}
identifies the two classifications by conjugating with the projective
power maps $F_{t,C}^{\mathrm p}:\bar a\mapsto\overline{a^t}$,
$C=A,B$.
Our proof starts with normalization by a left $\poplus_1$-gyrotranslation
so that $g(\bar e_A)=\bar e_B$. The GGV Mazur--Ulam theorem then gives
operation preservation, and Lemma~\ref{lem:projective-log-triple} extracts
the additional quotient Lie triple identity used in
Proposition~\ref{prop:projective-Jordan-reconstruction}.

The common-sign argument has a precedent in the orthogonal-simplex and
projection-chain arguments of \cite[Theorem~4.9 and Lemma~4.10]{LRW}.
For general JB-algebras, \cite[Sections~3--4]{RW} uses the
facial geometry of dual balls and the atomic parts of the biduals.
Their passage from projection orthoisomorphisms to Jordan isomorphisms
uses the Bunce--Wright extension theorem and treats the type $I_2$
summand separately; see also \cite[Section~4.3]{LRW}.
In the Euclidean Jordan case, \cite[Remark~4.20]{LRW} already constructs
a lift by spectral decomposition. Its linearity follows from quotient
linearity using the trace-zero subspaces, and the unital linear-isometry
result then gives a Jordan isomorphism.

Here infinitesimal gyrations recover Peirce components on the full
biduals. The centered lift in
\eqref{eq:centered-Jordan-lift} requires no trace: its norm continuity
and spectral approximation first give square preservation, and
Lemma~\ref{lem:square-preserving-lift} then proves additivity. Restriction
to the original algebras gives the Jordan isomorphism for general
unital JB-algebras, without invoking the extension or linear-isometry
classification theorems above. The global formula is known; the proof
uses the extra information carried by the GGV operations.
\end{remark}

\Needspace{8\baselineskip}
\subsection{Projection chains and square-preserving lifts}

We use the standard Peirce and spectral theory of JBW-algebras
\cite[Sections~2.6 and~4.2]{HOS}. Write $\mathcal P(M)=\{p\in M:p^2=p\}$.
For a projection $p$, write
$M_\alpha(p)=\{x:p\circ x=\alpha x\}$ and $P^p_\alpha$ for its
Peirce projections, where $\alpha=0,1/2,1$. An element $x$
\emph{operator commutes} with $p$ if $[L_x,L_p]=0$. This is equivalent
to $P^p_{1/2}x=0$. Indeed, $[L_p,L_x]p=(L_p^2-L_p)x=-P^p_{1/2}x/4$.
Conversely, if $x\in M_1(p)\oplus M_0(p)$, the Peirce multiplication
rules show that $L_x$ leaves each eigenspace of $L_p$ invariant.
Hence $L_x$ commutes with $L_p$.
The commutant $\{p\}'=M_1(p)\oplus M_0(p)$ is a weak* closed
unital Jordan subalgebra and contains the spectral projections of
its elements \cite[Lemma~4.2.5, pp.~99--100]{HOS}. A nonzero projection is minimal if it
has no nonzero proper subprojection; in a JBW-algebra this is
equivalent to $M_1(p)=\mathbb Rp$: in the JBW corner $M_1(p)$
with unit $p$, spectral approximation by projections
\cite[Proposition~4.2.3, p.~99]{HOS} gives this conclusion when its
only projections are $0,p$.

The following projection-chain conclusion is \cite[Lemma~4.10]{LRW}.
We give a direct proof using Peirce symmetries, without a type
decomposition. Its role is to connect the orientations of different
three-projection partitions in the reconstruction below.

\Needspace{4\baselineskip}
\begin{lemma}[Connecting projections by Peirce symmetries]
\label{lem:Peirce-projection-chain}
Let $M$ be a unital JBW-algebra and put
\[
 \mathcal P_*=
 \{p\in\mathcal P(M):0<p<e,\ e-p\text{ is not minimal}\}.
\]
If $\mathcal P_*\ne\varnothing$, any two of its elements are joined
by a finite sequence $p=p_0,\ldots,p_n=q$ in $\mathcal P_*$ with
$p_i\perp p_{i+1}$ and $p_i+p_{i+1}<e$.
\end{lemma}

\begin{proof}
Write $p\sim q$ when such a finite chain exists. This is an
equivalence relation: chains of length zero, reversal, and
concatenation give reflexivity, symmetry, and transitivity.
Every nonzero subprojection $r\le p\in\mathcal P_*$ also lies in
$\mathcal P_*$, since $e-r\ge e-p$ contains a nonminimal projection.
If $p\perp q$ and $p+q<e$, they are adjacent. If $p+q=e$, both are
nonminimal, so $0<r<p$ and $0<s<q$ give the chain $p,s,r,q$. If $p<q$, then $e-q\in\mathcal P_*$
because $q$ is nonminimal, and $p\sim e-q\sim q$; the case $q<p$
follows by symmetry. Operator commuting projections $p,q$ generate
an associative subalgebra \cite[Theorem~3.13]{vandeWetering}, in which
$p\circ q=p\wedge q$ is a projection. If it is zero, one of the two orthogonal cases above
applies; otherwise the containment case gives
$p\sim p\circ q\sim q$ (with reflexivity when the meet equals
one of the projections).

For arbitrary $p,q\in\mathcal P_*$, put $h=(p-q)^2$.
Writing $q=q_1+w+q_0$ in the Peirce decomposition relative to $p$
gives $h=p-q_1+q_0$. Thus $h$ operator commutes with $p$, and,
by symmetry, with $q$. If $h$ is nonscalar, JBW spectral theory
provides a nontrivial spectral projection $z$ of $h$, and $z$
operator commutes with both $p$ and $q$
\cite[Proposition~4.2.3 and Lemma~4.2.5, pp.~99--100]{HOS}.
At least one of $z,e-z$ belongs to $\mathcal P_*$: otherwise both
are minimal, and a projection commuting with $z$ belongs to
$\{0,z,e-z,e\}$, contradicting $p\in\mathcal P_*$.
Replacing $z$ by its complement if necessary gives $p\sim z\sim q$.

Suppose instead that $(p-q)^2=\lambda e$. Since $\|p-q\|\le1$,
$0\le\lambda\le1$. The Peirce components satisfy
\[
 q=(1-\lambda)p+\lambda(e-p)+w,\qquad
 w\in M_{1/2}(p),\qquad w^2=\lambda(1-\lambda)e.
\]
The last equality follows by expanding $q^2=q$: the product of
the diagonal part with $w$ is $w/2$.
At either endpoint $w^2=0$, hence $\|w\|^2=\|w^2\|=0$. This gives
$p=q$ or $q=e-p$, respectively. Otherwise
$v=w/\sqrt{\lambda(1-\lambda)}$ satisfies $v^2=e$ and
$v\circ p=v/2$. Hence
\[
 U_vp=2v\circ(v\circ p)-v^2\circ p=e-p.
\]
The fundamental formula gives $U_v^2=I$, $U_ve=e$, and
$(U_vx)^2=U_v(x^2)$; polarization makes $U_v$ a Jordan
automorphism. It exchanges $p,e-p$, so $p=U_v(e-p)$ is nonminimal
because $e-p$ is. Choose $0<r<p$ and put $z=r+U_vr$.
The automorphism preserves projections and order, hence
$U_vr\le U_vp=e-p$; since $r\le p$, the summands are orthogonal.
Moreover,
$e-z=(p-r)+U_v(p-r)$ is a sum of two nonzero orthogonal projections,
so $z\in\mathcal P_*$. We have $U_vz=z$. For the projection
$s=(e+v)/2$, the Peirce eigenvalues give
$P^s_{1/2}=4(L_s-L_s^2)$. Since $L_v=2L_s-I$ and $v^2=e$,
\[
 U_v=2L_v^2-I=8L_s^2-8L_s+I=I-2P^s_{1/2}.
\]
Thus $P^s_{1/2}z=0$, so $z$ operator commutes with $s$, hence
with $v=2s-e$. Since $r\le p$ and $U_vr\le e-p$, their sum $z$ lies
in $M_1(p)\oplus M_0(p)$ and operator commutes with $p$, hence with
$q=(1-\lambda)p+\lambda(e-p)+\sqrt{\lambda(1-\lambda)}v$.
Thus $p\sim z\sim q$ in this case as well.
\end{proof}

The next lemma is an algebraic lifting step: it derives additivity
from square preservation once linearity modulo the scalar direction and
scalar-translation covariance are known. It does not use a norm or a
GGV structure. This permits the reconstruction below to establish
square preservation before proving additivity.

\Needspace{4\baselineskip}
\begin{lemma}[A square-preserving lift is linear]
\label{lem:square-preserving-lift}
Let $E,F$ be nonzero unital real Jordan algebras, and let
$R:E/\mathbb Re_E\to F/\mathbb Re_F$ be injective and real linear.
Suppose $K:E\to F$ satisfies, for $x\in E$ and $r,\lambda\in\mathbb R$,
\[
 [Kx]=R[x],\qquad
 K(rx+\lambda e_E)=rKx+\lambda e_F,\qquad
 K(x^2)=(Kx)^2.
\]
Then $K$ is an injective unital real linear Jordan homomorphism.
If $R$ is surjective, $K$ is a Jordan isomorphism.
\end{lemma}

\begin{proof}
Put $u=Kx$, $v=Ky$. Quotient linearity gives real $\alpha,\beta$ with
$K(x+y)=u+v+\alpha e_F$ and $K(x-y)=u-v+\beta e_F$.
Apply $R$ to the class of
$(x+y)^2+(x-y)^2=2x^2+2y^2$ and use linearity of $R$ and square preservation.
The expansion in $F$ is
\[
\begin{aligned}
 &(u+v+\alpha e_F)^2+(u-v+\beta e_F)^2-2u^2-2v^2\\
 &\qquad=2(\alpha+\beta)u+2(\alpha-\beta)v
              +(\alpha^2+\beta^2)e_F.
\end{aligned}
\]
The scalar term vanishes in the quotient; dividing by $2$ gives
\[
 (\alpha+\beta)[u]+(\alpha-\beta)[v]=0.
\]
If $[x],[y]$ are independent, injectivity of $R$ gives
$\alpha=\beta=0$. If $[x]=0$, write $x=\lambda e_E$.
Taking $r=0$ in the covariance identity gives
$K(\lambda e_E)=\lambda e_F$, so $K(x+y)=Kx+Ky$; the case
$[y]=0$ follows by symmetry. Otherwise $[x]\ne0$ and
$[y]=r[x]$ for some $r$, hence $y=rx+\lambda e_E$. Then
\[
 K(x+y)=(1+r)Kx+\lambda e_F=Kx+Ky.
\]
Thus $K$ is additive and, by the assumed homogeneity, real linear.
Also $K(e_E)=e_F$ by taking $r=0$, $\lambda=1$. Polarization gives
\[
\begin{aligned}
 K(x\circ y)
 &=\tfrac12\bigl(K((x+y)^2)-K(x^2)-K(y^2)\bigr)\\
 &=Kx\circ Ky.
\end{aligned}
\]
If $Kx=0$, then $R[x]=0$, so injectivity of $R$ gives
$x=\lambda e_E$; scalar covariance yields $0=\lambda e_F$,
hence $\lambda=0$ since $e_F\ne0$. Thus $K$ is injective.
If $R$ is onto, choose $x$ with $[Kx]=[b]$ for any $b\in F$.
Writing $Kx=b+\lambda e_F$ gives $K(x-\lambda e_E)=b$.
\end{proof}

The conclusion of the following proposition is known under the weaker
assumption that $S$ is a surjective linear variation-norm isometry
\cite[Theorem~4.4]{RW}. Here the extra identity
\eqref{eq:quotient-Lie-triple}, supplied by
Lemma~\ref{lem:projective-log-triple}, gives a direct reconstruction.
The relation
\[
 D_{p,x}p=-\tfrac14P^p_{1/2}x
\]
recovers the Peirce $1/2$ component from infinitesimal gyrations. We use
it to preserve operator commutativity of projections, choose a common
sign, and construct a square-preserving lift. The biduals supply the
finite spectral approximations needed for this last step.

\Needspace{8\baselineskip}
\subsection{The common sign and the centered lift}

\Needspace{4\baselineskip}
\begin{proposition}[Jordan reconstruction from infinitesimal gyrations]
\label{prop:projective-Jordan-reconstruction}
Let $A,B$ be unital JB-algebras of dimension at least two.
Suppose $S:[A]\to[B]$ is a surjective real linear variation-norm
isometry satisfying \eqref{eq:quotient-Lie-triple}. Then there are
$\varepsilon\in\{-1,1\}$ and a Jordan isomorphism $J:A\to B$ such that
\[
 S[x]=\varepsilon[Jx]\qquad(x\in A).
\]
\end{proposition}

\begin{proof}
\emph{Passage to the biduals.}
Put $M=A^{**}$ and $N=B^{**}$, which are unital JBW-algebras
with separately weak* continuous products
\cite[Theorem~4.4.3, p.~104]{HOS} and units $e_M=e_A$, $e_N=e_B$,
as in Lemma~\ref{lem:shared-bidual-projections}.
For a closed subspace $Y$ of a Banach space $X$, the canonical
identification is $(X/Y)^{**}\cong X^{**}/Y^{\perp\perp}$.
Under the canonical embedding in $X^{**}$, a finite-dimensional
subspace is weak* closed, so $Y^{\perp\perp}=Y$ for the scalar
lines. Consequently,
\[
 (A/\mathbb Re_A)^{**}\cong M/\mathbb Re_A,\qquad
 (B/\mathbb Re_B)^{**}\cong N/\mathbb Re_B.
\]
Formula~\eqref{eq:variation-norm} applies also to the JB-algebras $M,N$,
so the variation norm is twice the usual quotient norm on all four algebras. Thus these identifications give
$[A]^{**}=[M]$ and $[B]^{**}=[N]$ isometrically.
In particular $\widehat S=S^{**}:[M]\to[N]$ is a surjective linear variation-norm
isometry, with inverse $(S^{-1})^{**}$.
Write $j_X:X\to X^{**}$ for the canonical embedding of a Banach
space. These identifications send $j_{[A]}([a])$ to
$[j_A(a)]\in[M]$, and similarly for $B$.
Thus $S^{**}j_{[A]}=j_{[B]}S$ says that $\widehat S$ extends $S$.
We identify $A,B$ with their canonical images below.
Write $q_C:C\to[C]$ for the quotient map. Choose a bounded linear
functional $\omega$ on $B$ with $\omega(e_B)=1$ by Hahn--Banach.
The section $j_\omega([b])=b-\omega(b)e_B$ gives the bounded lift
$V=j_\omega S q_A:A\to B$. Thus $q_BV=S q_A$. Under the canonical
identifications, $q_M=q_A^{**}$ and $q_N=q_B^{**}$, so
$q_NV^{**}=\widehat S q_M$. Write \eqref{eq:quotient-Lie-triple} using
$Vx,Vy,Vz$ as representatives. Goldstine's theorem gives weak*
approximating nets from $A$ for elements of $M$. Extend first in
$x$, then in $y$, and finally in $z$, fixing at each step the
previously extended variables as elements of $M$. The products,
the quotient maps, $\widehat S=S^{**}$, and $V^{**}$ have the required
weak* continuity. In particular, for fixed $x,z\in M$,
\[
 D_{x,y}z=(L_xL_z-L_{x\circ z})y
\]
is weak* continuous in $y$, and the same holds in $N$.
The extended identity holds for these lifts and hence for arbitrary lifts:
in $N$ the value of $D_{u,v}w$ depends only on the three classes, as recalled before Lemma~\ref{lem:projective-log-triple}.

\emph{Projections and their orientation.}
The extreme points of the variation-norm unit ball of $[M]$ are
exactly the classes of nontrivial projections. Here is the elementary
argument, also recorded in \cite[Lemma~4.1]{LRW}.
If $[p]=([a]+[b])/2$ with both norms at most one, subtract the
minimum spectral value from each representative. Then
$\min\sigma(a)=\min\sigma(b)=0$ and their maximum spectral values
are their variation norms, at most one, so $0\le a,b\le e$.
Now $a+b=2p+\lambda e$ and $0\le a+b\le2e$ force $\lambda=0$.
Projections are extreme in $[0,e]$ in every unital JB-algebra,
by Lemma~\ref{lem:shared-bidual-projections},
so $a=b=p$.
Conversely, $[M]\ne0$ since $\dim A\ge2$. If $\|\xi\|_v<1$,
choose $\eta$ with $0<\|\eta\|_v<1-\|\xi\|_v$; the distinct
points $\xi\pm\eta$ lie in the unit ball and have midpoint $\xi$,
so $\xi$ is not extreme. Normalize a class of norm one to
$0\le x\le e$ with
$\min\sigma(x)=0$, $\max\sigma(x)=1$. If $x$ is not a projection,
functional calculus puts $x^2$ and $2x-x^2$ in $[0,e]$, so their
classes lie in the unit ball, and
\[
 [x]=\frac12\bigl([x^2]+[2x-x^2]\bigr).
\]
These classes are distinct: otherwise $x-x^2=\lambda e$ for some
$\lambda\in\mathbb R$, and the spectral value $0$ forces $\lambda=0$,
contrary to $x^2\ne x$. Thus $[x]$ is not extreme.

A nontrivial projection is uniquely determined by its class:
if $p=q+\lambda e$ for nontrivial projections $p,q$, then
$\sigma(p)=\sigma(q)=\{0,1\}$ and spectral translation forces
$\lambda=0$. A surjective linear isometry maps extreme points of the unit
ball bijectively onto extreme points. Consequently $\widehat S[p]=[\theta(p)]$
defines a bijection on nontrivial projections. Set $\theta(0)=0$, $\theta(e)=e$.
Linearity gives
$\theta(e-p)=e-\theta(p)$ for nontrivial $p$.
The identity $D_{p,x}p=-P^p_{1/2}x/4$ implies
\[
 \widehat S([P^p_{1/2}x])=[P^{\theta(p)}_{1/2}y],\qquad [y]=\widehat S[x].
\]
Since a scalar in a Peirce $1/2$ space is zero, operator
commutativity of projections is preserved in both directions.

For a partition $p+a+b=e$ into nonzero orthogonal projections,
the three image projections commute and their sum is $\lambda e$
by linearity of $\widehat S$. By \cite[Theorem~3.13]{vandeWetering}, the first two generate an
associative JB-subalgebra. Adjoining the unit preserves associativity,
so the resulting unital JB-subalgebra is isomorphic to $C(X,\mathbb R)$
for a compact $X$ \cite[Theorem~3.2.2]{HOS}.
It contains the third projection, obtained by subtraction from
$\lambda e$. Each projection takes only the values $0,1$, so
$\lambda\in\{0,1,2,3\}$. The values $0$ and $3$ would make all
three projections $0$ or all three $e$, contrary to nontriviality.
Thus the sum is $e$ or $2e$; call these orientations positive
and negative, respectively.
Two partitions sharing a projection have the same orientation.
Indeed, after interchanging them if necessary, write
$p+a+b=e=p+c+d$ with the first positive and the second negative,
and put $q=\theta(p)$. Then
\[
 \theta(a)+\theta(b)=e-q,\qquad
 (e-\theta(c))+(e-\theta(d))=q,
\]
so $\theta(a)\le e-q\le\theta(c),\theta(d)$.
Comparable projections operator commute; the reverse direction of
commutativity preservation therefore makes $a,c$ and $a,d$ commute.
Since $\theta(c)-\theta(a)$ is a projection, isometry gives
\[
 \|[a-c]\|_v=\|[\theta(a)-\theta(c)]\|_v\le1,
\]
and the same argument gives $\|[a-d]\|_v\le1$.
For incomparable commuting projections $a,c$, the projections
$a\circ(e-c)$ and $c\circ(e-a)$ are nonzero and orthogonal,
and their difference is $a-c$. Hence $\sigma(a-c)$ contains both
$1$ and $-1$. The bounds above thus make $a$ comparable with
both $c,d$. As $c,d$ are nonzero and orthogonal, $a\le c$
would make $a,d$ incomparable. Hence $c\le a$, and similarly
$d\le a$, contradicting $c+d=e-p=a+b$ with $b\ne0$.

Every member of a triple partition belongs to $\mathcal P_*$, since its
complement is a sum of two nonzero orthogonal projections.
Choose one member from each of two partitions. If they coincide,
the partitions already share a projection; otherwise join them by
a chain from Lemma~\ref{lem:Peirce-projection-chain}. Each adjacent
pair gives the triple partition
$p_i+p_{i+1}+(e-p_i-p_{i+1})=e$ with all three terms nonzero.
Successive partitions share a projection, and the first and last
share a projection with the original partitions. Thus all triple
partitions have the same orientation, and one choice
$\varepsilon=\pm1$ gives the following map on nontrivial projections:
$\Theta(p)=\theta(p)$ in the positive orientation and
$\Theta(p)=e-\theta(p)$ in the negative orientation. Then
$[\Theta(p)]=\varepsilon \widehat S[p]$ and each image triple sums to $e$:
in the negative orientation its sum is $3e-2e=e$. The three images are
pairwise orthogonal, since
the sum of any two is the complement of the third and hence a
projection. The map $\Theta$ also preserves complements.
If there are no triple partitions, every pair of nonzero orthogonal
projections is complementary, and either choice, $\Theta(p)=\theta(p)$
or $\Theta(p)=e-\theta(p)$ on nontrivial projections, works.
Set $\Theta(0)=0$, $\Theta(e)=e$ in all cases.
For orthogonal $p,q$, additivity is immediate if either is zero;
if $p+q=e$, it follows from complement preservation. Otherwise
$p,q,e-p-q$ form a triple partition, so
\[
 \Theta(p)+\Theta(q)=e-\Theta(e-p-q)=\Theta(p+q).
\]
Thus $\Theta$ preserves orthogonal sums and orthogonality, the
latter because the sum of the two image projections is a projection.

\emph{Squares before additivity.}
For any unital JB-algebra $C$, write
\[
 \mu_C(x)=\frac{\max\sigma(x)+\min\sigma(x)}2,\qquad
 \kappa_C([x])=x-\mu_C(x)e_C.
\]
Spectral translation and scaling show that the second map is well
defined and real homogeneous; negative scalars exchange the two
spectral endpoints. The endpoint bounds give
$|\mu_C(x)-\mu_C(y)|\le\|x-y\|$. For any $\lambda\in\mathbb R$,
\[
\begin{aligned}
 \|\kappa_C([x])-\kappa_C([y])\|
 &\le\|x-y-\lambda e_C\|
       +|\mu_C(x-\lambda e_C)-\mu_C(y)|\\
 &\le2\|x-y-\lambda e_C\|.
\end{aligned}
\]
Taking the infimum over $\lambda$ and using
\eqref{eq:variation-norm}, which makes the variation norm twice
the usual quotient norm, gives
$\|\kappa_C(\xi)-\kappa_C(\eta)\|\le\|\xi-\eta\|_v$. Define
\begin{equation}\label{eq:centered-Jordan-lift}
 \mathcal J(x)=\kappa_N(\varepsilon \widehat S[x])+\mu_M(x)e_N.
\end{equation}
It is real homogeneous, commutes with scalar translations, and
lifts $\varepsilon \widehat S$. The preceding estimates give
\[
\begin{aligned}
 \|\mathcal J(x)-\mathcal J(y)\|
 &\le\|\varepsilon \widehat S([x]-[y])\|_v+|\mu_M(x)-\mu_M(y)|\\
 &\le\|[x-y]\|_v+\|x-y\|\le3\|x-y\|.
\end{aligned}
\]
Let $x=\sum_i\lambda_i p_i$ with nonzero orthogonal $p_i$ summing
to $e$, where the real coefficients need not be distinct.
Preservation of orthogonal pairs and sums makes the $\Theta(p_i)$
orthogonal and their sum $e$. Their nonzero classes follow from
injectivity of $\widehat S$ when $p_i\ne e$; the one-term case uses
$\Theta(e)=e$. Thus $y=\sum_i\lambda_i\Theta(p_i)$ has class
$\varepsilon \widehat S[x]$ and the same spectral endpoints as $x$.
Since every $y\in N$ satisfies
$y=\kappa_N([y])+\mu_N(y)e_N$, this identifies
$\mathcal J(x)=\sum_i\lambda_i\Theta(p_i)$.
Applying the same argument to the coefficients $\lambda_i^2$
gives $\mathcal J(x^2)=\mathcal J(x)^2$.
By \cite[Proposition~4.2.3, p.~99]{HOS}, finite real linear combinations
of orthogonal projections approximate every element of $M$ in norm.
Discarding zero projections and, if necessary, adjoining
$e-\sum_i p_i$ with coefficient $0$ makes the underlying projections
a partition of $e$ without changing the linear combination.
The estimate
\[
 \|u^2-v^2\|\le(\|u\|+\|v\|)\|u-v\|
\]
and norm continuity of $\mathcal J$ therefore extend square
preservation to all of $M$. No weak* continuity of $\mu_C$ or
$\kappa_C$ is asserted or used.
Lemma~\ref{lem:square-preserving-lift} now makes $\mathcal J:M\to N$
a Jordan isomorphism.

\emph{Restriction to the original algebras.}
The canonical embeddings identify isometrically the closed unital
Jordan algebra generated by an element with that generated by its
image. The continuous functional calculus therefore preserves its
spectrum \cite[Theorem~3.2.4, p.~79]{HOS}.
Under these embeddings, $\mu_M|_A=\mu_A$, $\kappa_N|_{[B]}=\kappa_B$,
and $e_N=e_B$. Since $\widehat S$ extends $S$ as above,
\eqref{eq:centered-Jordan-lift} maps $A$ into $B$, with restriction
$J=\mathcal J|_A$ given by
\[
 J(a)=\kappa_B(\varepsilon S[a])+\mu_A(a)e_B.
\]
Define
\[
 Q(b)=\kappa_A(\varepsilon S^{-1}[b])+\mu_B(b)e_A\qquad(b\in B).
\]
The identities $\mu_C(\kappa_C(\xi))=0$ and
$\mu_C(z+\lambda e_C)=\mu_C(z)+\lambda$ give
\[
\begin{aligned}
 [Ja]&=\varepsilon S[a],&\qquad \mu_B(Ja)&=\mu_A(a),\\
 [Qb]&=\varepsilon S^{-1}[b],& \mu_A(Qb)&=\mu_B(b).
\end{aligned}
\]
Using $\varepsilon^2=1$, both compositions are therefore
\[
\begin{aligned}
 Q(Ja)&=\kappa_A([a])+\mu_A(a)e_A=a,\\
 J(Qb)&=\kappa_B([b])+\mu_B(b)e_B=b.
\end{aligned}
\]
Thus $J$ is onto $B$ and injective. Since the canonical embeddings
are unital Jordan homomorphisms, $J$ inherits linearity and square
preservation from $\mathcal J$, and $J(e_A)=e_B$.
Hence $J:A\to B$ is a Jordan isomorphism lifting $\varepsilon S$,
that is, $[Ja]=\varepsilon S[a]$.
\end{proof}

\Needspace{8\baselineskip}
\subsection{Affine reconstruction}

\Needspace{4\baselineskip}
\begin{lemma}[Jordan reconstruction from commutator preservation]
\label{lem:affine-Jordan-reconstruction}
Let $A,B$ be unital JB-algebras and let $S:A\to B$ be a
surjective real linear norm isometry satisfying
\begin{equation}\label{eq:affine-Lie-triple}
 S(D_{x,y}z)=D_{Sx,Sy}Sz\qquad(x,y,z\in A).
\end{equation}
Then $s=S(e_A)$ is a central symmetry in $B$, and
$J(x)=s\circ S(x)$ is a Jordan isomorphism $A\to B$.
In particular, $S(x)=s\circ J(x)$ for all $x\in A$.
\end{lemma}

\begin{proof}
Since $D_{e_A,y}=0$, \eqref{eq:affine-Lie-triple} and surjectivity
give $[L_s,L_b]c=0$ for all $b,c\in B$. Thus $s$ is central.
The unit $e_A$ is an extreme point of the closed unit ball: if
$e_A=(u+v)/2$ with $\|u\|,\|v\|\le1$, then $e_A-u$ and $e_A-v$
are positive and sum to zero by the spectral order bounds. Properness of
$A_+$ then gives $u=v=e_A$ (Subsection~\ref{sec:JB-preliminaries}).
Since $S$ maps the closed unit ball onto the closed unit ball, $s$
is also extreme. If $s^2\ne e_B$, put $h=(e_B-s^2)/2\ne0$.
For $-1\le\lambda\le1$,
\[
 \left|\lambda\pm\frac{1-\lambda^2}{2}\right|\le1.
\]
Functional calculus therefore gives $\|s\pm h\|\le1$, contradicting
extremality of $s$. Consequently $s^2=e_B$.

Applying $[L_s,L_x]=0$ to $s$ gives
$L_s^2x=s\circ(x\circ s)=x\circ s^2=L_{s^2}x$ for every $x\in B$.
Thus $L_s^2=L_{s^2}=I$. As $\|s\|=1$, the operator
$L_s$ and its inverse are contractions, so $L_s$ is an isometry.
Thus $J=L_sS$ is a surjective real linear norm isometry with
$J(e_A)=e_B$.

Put $M=A^{**}$ and $N=B^{**}$, with the canonical unital Jordan
embeddings and units described in Lemma~\ref{lem:shared-bidual-projections}.
The map
$\widehat J=J^{**}:M\to N$ is again a unital surjective linear
isometry, with inverse $(J^{-1})^{**}$ and $\widehat J|_A=J$.
The spectral norm formula gives
\[
 [0,e_M]=\left\{x\in M:
       \left\|x-\frac{e_M}{2}\right\|\le\frac12\right\},
\]
and similarly in $N$. Hence $\widehat J$ maps $[0,e_M]$ onto $[0,e_N]$.

By Lemma~\ref{lem:shared-bidual-projections}, the extreme points of
each order interval are precisely its projections.
It follows that $\widehat J$ maps projections to projections. If
$p,q$ are orthogonal projections, then $p+q$ is a projection;
linearity implies that $\widehat J(p)+\widehat J(q)$ is a projection,
so $\widehat J(p)\circ\widehat J(q)=0$.
For every finite spectral sum $x=\sum_i\lambda_i p_i$ with pairwise
orthogonal projections $p_i$, we therefore have
\[
 \widehat J(x^2)=\sum_i\lambda_i^2\widehat J(p_i)
 =\left(\sum_i\lambda_i\widehat J(p_i)\right)^2
 =\widehat J(x)^2.
\]
Such sums approximate every element of $M$ in norm
\cite[Proposition~4.2.3, p.~99]{HOS}, so continuity gives square
preservation on all of $M$. Restricting to $A$ and polarizing shows
that $J$ is a Jordan isomorphism. Finally, $L_s^2=I$ gives $S=L_sJ$.
\end{proof}

\Needspace{8\baselineskip}
\section{Global and local classifications of cone isometries}
\label{sec:Thompson-geometry}

We prove Theorem~C for the Thompson and Hilbert metrics and obtain the
corresponding classifications for every $t>0$. Throughout this section,
$A,B$ are unital JB-algebras. Metrics and topologies are those of
Section~\ref{sec:cone-linearization}.
We first classify global maps by GGV normalization and the reconstruction
results of Section~\ref{sec:reconstruction}. Only after this step do we
apply Theorem~A to maps on connected open subsets.

\Needspace{8\baselineskip}
\subsection{Positive-cone isometries}

\Needspace{4\baselineskip}
\begin{remark}[Relation to the Thompson classification at $t=1$]\label{rem:thompson-classification-relation}
At $t=1$, the following theorem recovers the known Thompson-isometry
classification for unital JB-algebras in \cite[Theorem~3.2]{LRW},
which extends the $C^*$-algebraic classification of Hatori and Moln\'ar
\cite[Theorem~9]{HatoriMolnar}. For general $t>0$, Proposition~\ref{prop:power-conjugacy}
identifies the classification of $\rho_t$-isometries with this Thompson
classification by the power maps $a\mapsto a^t$.
We give a direct proof for every $t>0$ within the GGV framework of
the extension theorem, without invoking the Thompson classification.
The logarithmic map of the identity-preserving normalization has
the form $S(x)=s\circ J(x)$, with $s$ a central symmetry and $J$ a
Jordan isomorphism, by Lemma~\ref{lem:affine-Jordan-reconstruction}. This is the form in the linear JB-isometry
theorem of Isidro and Rodr\'iguez-Palacios \cite[Theorem~1.4]{IRP}.
In Lemma~\ref{lem:affine-Jordan-reconstruction}, the commutator identity
is used only to establish centrality of $s$. For the remaining unital
case, we reprove the result by the bidual argument used in
\cite[Theorem~4, pp.~341--342]{WrightYoungson1978}. Thus the required
linear form is obtained without invoking the full linear-isometry
classification.
The converse follows directly from the order metric.
\end{remark}

\Needspace{4\baselineskip}
\begin{theorem}[Classification of positive-cone gyrometric bijections]\label{thm:global-classification}
Let $A,B$ be unital JB-algebras and let $t>0$. A map
$f:A_+^\circ\to B_+^\circ$ is a bijective $\rho_t$-isometry if and only if
there exist a Jordan isomorphism $J:A\to B$, a central projection $p\in B$,
and $c\in B_+^\circ$ such that
\begin{equation}\label{eq:global-form}
 f(a)=
 \left[
 U_{c^{t/2}}
 \bigl(pJ(a^t)+p^\perp J(a^{-t})\bigr)
 \right]^{1/t},
 \qquad a\in A_+^\circ,
\end{equation}
where $p^\perp=e_B-p$. The data $c,J,p$ are unique, and $c=f(e_A)$.
\end{theorem}

\begin{proof}
Suppose first that $f$ is a bijective $\rho_t$-isometry, and put
$c=f(e_A)$. By Theorem~\ref{thm:JB-GGV}, both cones are GGVs, and
$f$ is surjective and preserves their gyrometrics. These are the
hypotheses of the GGV Mazur--Ulam theorem
\cite[Corollary~4.5]{AHrevisited}; the additional standard-norm
continuity will be used for the local extension in
Subsection~\ref{subsec:affine-local}. Left gyrotranslation is an
isometry by \cite[Proposition~3.5]{AHrevisited}, and the corollary shows that
\begin{equation}\label{eq:normalize-g}
 g(a)=c^{-1}\oplus_t f(a)
 =\bigl(U_{c^{-t/2}}f(a)^t\bigr)^{1/t}
\end{equation}
is a bijective gyrometric isometry fixing the identity and preserving
gyroaddition and scalar multiplication for every real scalar, by
equations~(4.3)--(4.4) of that corollary. In particular,
since $r\otimes a=a^r$,
\begin{equation}\label{eq:g-powers}
 g(a^r)=g(a)^r\qquad(r\in\mathbb R).
\end{equation}
The inverse normalization is $f(a)^t=U_{c^{t/2}}g(a)^t$.
Define
\begin{equation}\label{eq:S-log}
 S:A\longrightarrow B,\qquad S(x)=\log g(\exp x).
\end{equation}
The inverse bijections $\exp$ and $\log$ make $S$ bijective with $S(0)=0$,
and \eqref{eq:g-powers} gives
$g(\exp(x/n))=\exp(S(x)/n)$ for $n\ge1$.

Fix $x,y\in A$, put $h=t/n$ and $u=U_{\exp(-hx/2)}\exp(hy)$.
Lemma~\ref{lem:JB-infinitesimal-gyration} applied to $-x,y$ gives
$\log u=h(y-x)+O(h^3)$. Thus \eqref{eq:rho-affine-log} yields
\begin{equation}\label{eq:infinitesimal}
 n\rho_{t,A}(\exp(x/n),\exp(y/n))
 =h^{-1}\|\log u\|\longrightarrow\|x-y\|.
\end{equation}
Applying \eqref{eq:infinitesimal} in $B$ and using the power and isometry
properties of $g$ gives
\[
\begin{aligned}
 \|S(x)-S(y)\|
 &=\lim_{n\to\infty}
 n\rho_{t,B}(g(\exp(x/n)),g(\exp(y/n)))\\
 &=\lim_{n\to\infty}
 n\rho_{t,A}(\exp(x/n),\exp(y/n))=\|x-y\|.
\end{aligned}
\]
Thus $S$ is a surjective norm isometry fixing zero, so the classical
Mazur--Ulam theorem already gives real linearity. We also give a
direct argument from the preserved GGV operation, which applies to the
projective quotient in Lemma~\ref{lem:projective-log-triple}. Define
\[
 x\star_t y=\frac1t\log\bigl(U_{\exp(tx/2)}\exp(ty)\bigr).
\]
The same first-order expansion gives
$n((x/n)\star_t(y/n))\to x+y$. Since $S$ is real homogeneous by
\eqref{eq:g-powers} and preserves $\star_t$ by gyroaddition
preservation, the continuity of the isometry $S$ yields
\[
\begin{aligned}
 S(x+y)
 &=\lim_{n\to\infty}S\bigl(n((x/n)\star_t(y/n))\bigr)\\
 &=\lim_{n\to\infty}n\bigl((Sx/n)\star_t(Sy/n)\bigr)=Sx+Sy.
\end{aligned}
\]
This also proves real linearity directly.

The map $g$ preserves gyroaddition and inverses (take $r=-1$ in
\eqref{eq:g-powers}). The identity
$\gyr[a,b]z=\ominus(a\oplus b)\oplus(a\oplus(b\oplus z))$,
which follows from left gyroassociativity and cancellation, therefore
shows that $g$ intertwines gyrations.
Equations~\eqref{eq:gyr-t}, \eqref{eq:Klog}, and \eqref{eq:g-powers}
therefore give, for $x,y,z\in A$ and real $h$,
\[
 S\bigl(K_{\exp(thx),\exp(thy)}z\bigr)
 =K_{\exp(thSx),\exp(thSy)}Sz.
\]
Since $S$ is bounded and linear, we may subtract $Sz$, divide by
$t^2h^2/2$, and let $h\to0$. Lemma~\ref{lem:JB-infinitesimal-gyration} yields
\eqref{eq:affine-Lie-triple}. By
Lemma~\ref{lem:affine-Jordan-reconstruction}, $s=S(e_A)$ is a central
symmetry and $J(x)=s\circ S(x)$ is a Jordan isomorphism, with
$S(x)=s\circ J(x)$.
Put $p=(e_B+s)/2$ and $p^\perp=e_B-p$. These are orthogonal central
projections, and $s=p-p^\perp$.
Write $pB=\{p\circ b:b\in B\}$, and similarly for $p^\perp B$.
As in Subsection~\ref{sec:JB-preliminaries}, products with central
elements are Jordan products; thus $pJ(a)=p\circ J(a)=U_pJ(a)$,
since $U_p=L_p$ for a central projection.
The central projection splits $B$ into the direct summands $pB$ and
$p^\perp B$ \cite[2.5.7, p.~47]{HOS}. Polynomials and hence continuous
functional calculus act componentwise. Applying Lemma~\ref{lem:Ka-b}
to $J$ gives
\begin{equation}\label{eq:normalized-global-form}
 g(a)=\exp(S(\log a))=pJ(a)+p^\perp J(a^{-1}).
\end{equation}
The right-hand side of \eqref{eq:normalized-global-form} commutes
with all positive powers by componentwise functional calculus.
Consequently, by \eqref{eq:rho-Thompson}, a bijection
$f:A_+^\circ\to B_+^\circ$ satisfying $f(e_A)=e_B$ is a
$\rho_t$-isometry for some $t>0$ if and only if it is a
$\rho_t$-isometry for every $t>0$.
The condition $f(e_A)=e_B$ is essential. For the matrices $a,b$ in
footnote~\ref{fn:strict-power-metrics}, the congruence
$H(x)=a^{-1/2}xa^{-1/2}$ is a positively homogeneous order automorphism.
It preserves both order ratios and hence is a bijective $\rho_1$-isometry.
Since $H(a)=I$ and
$\rho_t(I,y)=\|\log y\|$ for every $t>0$, we have
\[
 \rho_2(H(a),H(b))=\rho_1(H(a),H(b))
 =\rho_1(a,b)<\rho_2(a,b).
\]

For general $c$, using \eqref{eq:g-powers} at $r=t$ and undoing the normalization gives
\[
 f(a)^t
 =U_{c^{t/2}}g(a^t)
 =U_{c^{t/2}}\bigl(pJ(a^t)+p^\perp J(a^{-t})\bigr).
\]
Taking the positive $t$-th root proves \eqref{eq:global-form}.

Evaluation of \eqref{eq:global-form} at $e_A$ gives $f(e_A)=c$,
since $J(e_A)=e_B$ and $U_{c^{t/2}}e_B=c^t$. This determines $g$
and $S$ through \eqref{eq:normalize-g} and \eqref{eq:S-log}.
For any representing data in \eqref{eq:global-form}, componentwise
functional calculus gives \eqref{eq:normalized-global-form} and
$S(x)=(2p-e_B)\circ J(x)$. Thus the data are recovered by
\[
 s=S(e_A),\qquad p=\frac{e_B+s}{2},\qquad
 J(x)=s\circ S(x)\quad(x\in A).
\]
Here $J(e_A)=e_B$, and the last equality uses centrality of $s$ and
$s^2=e_B$. This proves uniqueness.

Conversely, the power conjugate of \eqref{eq:global-form} is
\[
 H(x)=U_{c^{t/2}}\bigl(pJ(x)+p^\perp J(x^{-1})\bigr).
\]
The order isomorphisms $J$ and $U_{c^{t/2}}$ preserve the order ratios
$M$. By order reversal under inversion in Subsection~\ref{sec:JB-preliminaries},
inversion interchanges the two order ratios and preserves $d_T$.
In $B=pB\oplus p^\perp B$, each order ratio is the maximum of the
component ratios, so $d_T$ is the maximum of the component distances
(omitting any zero summand). Inverting only the $p^\perp B$ component
is therefore an involutive isometry. Hence $H$ is a bijective Thompson
isometry, and Proposition~\ref{prop:power-conjugacy} gives the result.
\end{proof}

At $t=1$, Corollary~\ref{cor:rho-Thompson} and
Theorem~\ref{thm:global-classification} give the Thompson classification.
The bijective Thompson isometries are precisely the maps
\[
 f(a)=U_{c^{1/2}}\bigl(pJ(a)+p^\perp J(a^{-1})\bigr),
\]
with $c=f(e_A)$, a Jordan isomorphism $J:A\to B$, and a central
projection $p\in B$ with $p^\perp=e_B-p$.

\Needspace{8\baselineskip}
\subsection{Projective-cone isometries}

\Needspace{4\baselineskip}
\begin{theorem}[Classification of projective-cone gyrometric bijections]
\label{thm:projective-global-classification}
Let $A$ and $B$ be unital JB-algebras of dimension at least two.
Fix $t>0$. A map
\[
 f:\mathbb P A_+^\circ\longrightarrow\mathbb P B_+^\circ
\]
is a bijective $\prho_t$-isometry if and only if there exist a Jordan
isomorphism $J:A\to B$, a sign $\varepsilon\in\{-1,1\}$, and
$c\in B_+^\circ$ such that
\begin{equation}\label{eq:projective-global-form}
 f(\bar a)=
 \overline{
 \left[U_{c^{t/2}}J(a^{\varepsilon t})\right]^{1/t}}
 \qquad(a\in A_+^\circ).
\end{equation}
A sign is \emph{admissible} for $f$ if some $J,c$ realize
\eqref{eq:projective-global-form} with that sign.
Suppose now that $f$ satisfies these equivalent conditions.
The ray $\bar c=f(\bar e_A)$ is unique, and for each admissible sign
$\varepsilon$ the Jordan isomorphism $J$ is unique. Both signs are
admissible exactly when projective inversion on $\mathbb P A_+^\circ$
is induced by a Jordan automorphism of $A$. This occurs exactly for
$A$ Jordan isomorphic to $\mathbb R^2$ or to a spin factor.
Here a spin factor means $\mathbb R\oplus W$ with $W$ a real Hilbert
space of dimension at least two, possibly infinite dimensional,
as in \cite[Section~6.1, especially Lemma~6.1.3, p.~136]{HOS}.
\end{theorem}

\begin{proof}
Write $P_A=F_{t,A}^{\mathrm p}$ and $P_B=F_{t,B}^{\mathrm p}$.
By Proposition~\ref{prop:projective-power-conjugacy}, a bijective
$\prho_t$-isometry $f$ gives a bijective $\prho_1$-isometry
$H=P_B\circ f\circ P_A^{-1}$. Choose a representative
$c\in B_+^\circ$ of $f(\bar e_A)$. Then
$H(\bar e_A)=\overline{c^t}$ and $\ominus H(\bar e_A)=\overline{c^{-t}}$.
Since $\bar u\poplus_1\bar y=\overline{U_{u^{1/2}}y}$ by
\eqref{eq:projective-operations}, the normalization is
\[
 g(\bar x)=\overline{c^{-t}}\poplus_1 H(\bar x)
          =\overline{U_{c^{-t/2}}y},\qquad \bar y=H(\bar x).
\]
Left gyrotranslations are gyrometric isometries
\cite[Proposition~3.5]{AHrevisited}, so $g$ is a bijective
$\prho_1$-isometry and $g(\bar e_A)=\bar e_B$.
Lemma~\ref{lem:projective-log-triple} and
Proposition~\ref{prop:projective-Jordan-reconstruction} give a Jordan
isomorphism $J:A\to B$ and a common sign $\varepsilon$ with
$[Jh]=\varepsilon\Phi_Bg\Phi_A^{-1}([h])$. Since
$\varepsilon^2=1$,
\[
 \Phi_Bg\Phi_A^{-1}([h])=\varepsilon[Jh]\qquad(h\in A).
\]
By Lemma~\ref{lem:Ka-b}, $J$ commutes with continuous functional calculus.
For $x=\exp h\in A_+^\circ$, exponentiation therefore gives
\[
 g(\bar x)=\overline{\exp(\varepsilon Jh)}
           =\overline{J(x^\varepsilon)}.
\]
Applying its inverse left translation gives
$H(\bar x)=\overline{c^t}\poplus_1 g(\bar x)
=\overline{U_{c^{t/2}}J(x^\varepsilon)}$.
Substitution in $f=P_B^{-1}\circ H\circ P_A$ proves
\eqref{eq:projective-global-form}.

The map $J$ takes squares onto squares, and $U_{c^{t/2}}$ is positive
with positive inverse $U_{c^{-t/2}}$. Thus they are order isomorphisms
and preserve the interiors of the positive cones, so the expression
inside the root is positive and invertible.
For $\alpha,\beta>0$, the identities
$U_{\lambda z}=\lambda^2U_z$ and
$(\lambda z)^s=\lambda^sz^s$ for $z>0$, $\lambda>0$, give
\[
\begin{aligned}
 &\left[U_{(\beta c)^{t/2}}J((\alpha a)^{\varepsilon t})\right]^{1/t}\\
 &\qquad=\beta\alpha^\varepsilon
   \left[U_{c^{t/2}}J(a^{\varepsilon t})\right]^{1/t}.
\end{aligned}
\]
Changing the representatives of $\bar a$ and $\bar c$ therefore
preserves the image ray. Evaluation at $a=e_A$ gives
$f(\bar e_A)=\bar c$, proving uniqueness of this ray.
Conversely, a map of the stated form has power conjugate
$H(\bar x)=\overline{U_{c^{t/2}}J(x^\varepsilon)}$.
The order isomorphisms $J$ and $U_{c^{t/2}}$ preserve the ratios
$M(x/y)$. By order reversal under inversion in
Subsection~\ref{sec:JB-preliminaries}, $x^{-1}\le\lambda y^{-1}$
if and only if $y\le\lambda x$. Hence
\[
 M(x^{-1}/y^{-1})=M(y/x).
\]
By the Hilbert metric formula \eqref{eq:order-metrics}, $J$,
$U_{c^{t/2}}$, and inversion when $\varepsilon=-1$ induce bijective
Hilbert isometries on rays. Thus $H$ is one. Since
$\prho_1=\dH$ by Proposition~\ref{prop:prho-Hilbert},
Proposition~\ref{prop:projective-power-conjugacy} proves the converse.

Uniqueness of $\bar c$ was shown above. The scaling identity lets us
use the same representative $c$ in two formulas without changing
$J,\varepsilon$ or $J',\varepsilon'$. Since $P_A,P_B$ are bijections,
the formulas give the same $H$. The order isomorphism $U_{c^{t/2}}$
acts bijectively on rays; cancelling it and applying $J'^{-1}$ gives
\[
 \overline{K(x^\varepsilon)}=\overline{x^{\varepsilon'}}
 \qquad(x\in A_+^\circ),\qquad K=J'^{-1}J.
\]
If the signs agree, bijectivity of $x\mapsto x^\varepsilon$ makes
$K$ fix each positive ray. For nonscalar $x\in A_+^\circ$, write
$Kx=\lambda x$. Then
\[
 K(x+e_A)=\lambda x+e_A=\mu(x+e_A)
\]
for some $\mu>0$. Linear independence of $x,e_A$ gives
$\lambda=\mu=1$. The map $K$ also fixes scalars, and $A_+^\circ$
spans $A$, since
$a=(a+(\|a\|+1)e_A)-(\|a\|+1)e_A$ is a difference of positive
invertible elements. Hence $K=I$, proving uniqueness of $J$.
If the signs differ, putting $y=x^\varepsilon$ gives
$\overline{Ky}=\overline{y^{-1}}$, so $K$ induces projective inversion.
Conversely, for such a $K$,
\[
 \overline{(J\circ K)(a^{-\varepsilon t})}
   =\overline{J(a^{\varepsilon t})}.
\]
Thus replacing $(J,\varepsilon)$ by $(J\circ K,-\varepsilon)$ leaves
the ray of the radicand, and hence $f$, unchanged. Whether both signs
are admissible depends only on $A$.

To identify the sign ambiguity directly, suppose a Jordan
automorphism $K$ induces projective inversion. For every $x\in A$,
functional calculus gives
$\exp(Kx)=K(\exp x)=\lambda\exp(-x)$ for some $\lambda>0$.
Taking logarithms gives $Kx=(\log\lambda)e_A-x$, so $K+I$ takes
values in $\mathbb Re_A$. Define the real linear functional $\tau$ by
\[
 \tau(x)e_A=\tfrac12(Kx+x).
\]
Since $K$ is isometric, $|\tau(x)|\le\|x\|$; moreover
$\tau(e_A)=1$. Thus $W=\ker\tau$ is closed and
$Kz=z$ if and only if $z=\tau(z)e_A$.
For $w\in W$, we have $Kw=-w$ and
$K(w^2)=(Kw)^2=w^2$, hence $w^2=\tau(w^2)e_A$.
Polarization gives $w\circ z=\langle w,z\rangle e_A$, where
$\langle w,z\rangle=\tau(w\circ z)$. Since $w^2\ge0$ and
$\|w^2\|=\|w\|^2$, this is a positive definite inner product with
$\langle w,w\rangle=\|w\|^2$. The closed subspace $W$ is therefore
complete in its inner product norm, hence a real Hilbert space.
Writing $x=\tau(x)e_A+(x-\tau(x)e_A)$ identifies $A$ with
$\mathbb Re_A\oplus W$, whose product is
\[
 (s,w)\circ(r,z)=(sr+\langle w,z\rangle,sz+rw).
\]
Here $W\ne0$ since $\dim A\ge2$. If $\dim W=1$, choose a unit
vector $w\in W$. The nonzero orthogonal projections
$p_\pm=(e_A\pm w)/2$ give $A=\mathbb Rp_+\oplus\mathbb Rp_-
\cong\mathbb R^2$. Otherwise $A$ is a spin factor in the convention
stated above.
Conversely, in a spin factor the inverse of $(s,w)$, $s>\|w\|$, is
$(s,-w)/(s^2-\|w\|^2)$, and $\sigma(s,w)=(s,-w)$ is a Jordan
automorphism inducing projective inversion. In $\mathbb R^2$ take
$\sigma$ to be the coordinate interchange, which has the same
property. If $\varphi$ is a
Jordan isomorphism from $A$ onto either algebra, then
$K=\varphi^{-1}\sigma\varphi$ induces projective inversion on $A$,
because Jordan isomorphisms commute with inversion.
\end{proof}

Suppose that $\dim A\ge2$ and that the unit of $A^{**}$ has no decomposition
into three nonzero pairwise orthogonal projections. In the proof of
Proposition~\ref{prop:projective-Jordan-reconstruction}, either sign then
works. Taking $B=A$ and $S=\mathrm{id}_{[A]}$ in
\eqref{eq:quotient-Lie-triple}, the choice $\varepsilon=-1$ gives a
Jordan automorphism $J$ with $[Jx]=-[x]$ for $x\in A$. Exponentiating
yields $\overline{J(a)}=\overline{a^{-1}}$ for $a\in A_+^\circ$.
Theorem~\ref{thm:projective-global-classification} therefore identifies
$A$ up to Jordan isomorphism as $\mathbb R^2$ or a spin factor.

\Needspace{4\baselineskip}
\begin{corollary}[Hilbert isometries]\label{cor:Hilbert-global}
Let $A,B$ be unital JB-algebras of dimension at least two.
A map
$f:\mathbb P A_+^\circ\to\mathbb P B_+^\circ$ is a bijective Hilbert isometry
if and only if
\[
 f(\bar a)=\overline{U_bJ(a^\varepsilon)}
\]
for some Jordan isomorphism $J:A\to B$, some
$\varepsilon\in\{-1,1\}$, and some $b\in B_+^\circ$.
The ray $\bar b$ is uniquely determined by $f(\bar e_A)=\overline{b^2}$.
For each admissible sign $J$ is unique, and the sign ambiguity is as in
Theorem~\ref{thm:projective-global-classification}.
\end{corollary}

\begin{proof}
Take $t=1$ in Theorem~\ref{thm:projective-global-classification};
Proposition~\ref{prop:prho-Hilbert} gives $\prho_1=\dH$.
Put $b=c^{1/2}$, equivalently $c=b^2$, using that the positive
square-root map is a bijection of $B_+^\circ$. If
$\overline{b'^2}=\overline{b^2}$, then $b'^2=\lambda b^2$ for some
$\lambda>0$. Uniqueness and homogeneity of the positive square root
give $b'=\sqrt{\lambda}\,b$, proving uniqueness of the ray.
The assertions about $J$ and the sign follow from the same theorem.
\end{proof}

\Needspace{4\baselineskip}
\begin{remark}[Affine and projective classification]\label{rem:affine-projective-descent}
At $t=1$, the affine formula \eqref{eq:global-form} is
\[
 f(a)=U_{c^{1/2}}\bigl(pJ(a)+p^\perp J(a^{-1})\bigr),
\]
where $J:A\to B$ is a Jordan isomorphism, $p$ is a central projection
of $B$, $p^\perp=e_B-p$, and $c=f(e_A)\in B_+^\circ$.
It \emph{descends to rays}, meaning
$f(\alpha a)\in\mathbb R_{>0}f(a)$ for every $a\in A_+^\circ$
and $\alpha>0$, exactly when $p=0$ or $p=e_B$.
Indeed, choose $\alpha>0$ with $\alpha\ne1$. If $f$ descends,
$f(\alpha e_A)$ and $f(e_A)$ must lie on the same ray.
Applying $U_{c^{-1/2}}$ gives
\[
 \alpha p+\alpha^{-1}p^\perp=\lambda e_B
\]
for some $\lambda>0$. If both projections are nonzero, comparison
in the two central summands forces $\alpha=\lambda=\alpha^{-1}$,
contrary to the choice of $\alpha$. Conversely, $p=e_B$ or $p=0$
makes rescaling the input by $\alpha$ rescale the output by
$\alpha$ or $\alpha^{-1}$, respectively. These two cases give
$\varepsilon=1$ and $\varepsilon=-1$ in
Corollary~\ref{cor:Hilbert-global}, with $b=c^{1/2}$. Conversely, taking
$c=b^2$ and the corresponding $p$ lifts every bijective Hilbert isometry
to a ray-preserving bijective Thompson isometry.
The same criterion holds for $t>0$ by
Proposition~\ref{prop:power-conjugacy}: the power map and its inverse
preserve rays, since $(\alpha a)^t=\alpha^ta^t$.

For $A=B$ of dimension at least two and $t=1$, the projectivities
$\bar a\mapsto\overline{U_bJ(a)}$ form the whole Hilbert isometry group
if $A\cong\mathbb R^2$ or $A$ is a spin factor. In all other cases they
form a subgroup of index two, with the other coset obtained by composing
with inversion. This follows from Theorem~\ref{thm:projective-global-classification}
and the description of linear order automorphisms in
\cite[Proposition~2.3]{LRW}. See also \cite[Theorem~5.3]{RW}.
For the finite-dimensional and JBW-algebra cases, see
\cite{Bosche2012,Walsh2018} and \cite[Proposition~4.23]{LRW}, respectively.
\end{remark}

\Needspace{8\baselineskip}
\subsection{Local extension and classification on positive cones}
\label{subsec:affine-local}

Fix $t>0$.

\Needspace{4\baselineskip}
\begin{theorem}[Affine Mankiewicz-type extension theorem]
\label{thm:JB-local-global}
Let $U\subset A_+^\circ$ and $V\subset B_+^\circ$ be nonempty open
subsets, with $U$ connected. If
\[
 T:U\longrightarrow V
\]
is a bijective $\rho_t$-isometry, then $T$ extends uniquely to a bijective
$\rho_t$-isometry
\[
 \widetilde T:A_+^\circ\longrightarrow B_+^\circ.
\]
\end{theorem}

\begin{proof}
The isometry is a homeomorphism onto $V$, so $V=T(U)$ is connected.
By Theorem~\ref{thm:JB-GGV}, both affine cones are standard-norm GGVs,
and Corollary~\ref{cor:standard-local-global} applies.
\end{proof}

\Needspace{4\baselineskip}
\begin{corollary}[Affine local classification]
\label{cor:JB-local-classification}
Let $U\subset A_+^\circ$ and $V\subset B_+^\circ$ be nonempty open
subsets, with $U$ connected. A map $T:U\to V$ is a bijective $\rho_t$-isometry if and only if
there exist a Jordan isomorphism $J:A\to B$, a central projection $p\in B$,
and $c\in B_+^\circ$ such that
\begin{equation}\label{eq:local-form}
 T(a)=
 \left[
 U_{c^{t/2}}
 \bigl(pJ(a^t)+p^\perp J(a^{-t})\bigr)
 \right]^{1/t},\qquad a\in U,
\end{equation}
and the map defined on all of $A_+^\circ$ by the right-hand side of
\eqref{eq:local-form} maps $U$ onto $V$.
In this case the same formula on all of $A_+^\circ$ is the unique global
extension $\widetilde T$. The data $c,J,p$ are unique, and
$c=\widetilde T(e_A)$, whether or not $e_A$ belongs to $U$.
\end{corollary}

\begin{proof}
Extend $T$ by Theorem~\ref{thm:JB-local-global} and apply
Theorem~\ref{thm:global-classification} to $\widetilde T$. Conversely,
each set of data in the statement defines a global bijective isometry
whose restriction maps $U$ onto $V$. Any two choices representing $T$
therefore define the same global extension by Theorem~\ref{thm:JB-local-global},
so their parameters agree by Theorem~\ref{thm:global-classification}.
\end{proof}

\Needspace{4\baselineskip}
\begin{corollary}[Local-to-global theorem for Thompson's metric]
\label{cor:Thompson-local}
Let $T:U\to V$ be a bijective Thompson isometry between nonempty open
subsets of $A_+^\circ$ and $B_+^\circ$, with $U$ connected. Openness refers
to the norm topology, equivalently the Thompson metric topology by
Corollary~\ref{cor:rho-Thompson}. Then $T$ extends uniquely to a global
bijective Thompson isometry $\widetilde T$ of the form
\[
 a\longmapsto U_{c^{1/2}}
 \bigl(pJ(a)+p^\perp J(a^{-1})\bigr).
\]
The representing data $c\in B_+^\circ$, the Jordan isomorphism $J:A\to B$,
and the central projection $p\in B$ are unique, with $c=\widetilde T(e_A)$.
\end{corollary}

\begin{proof}
Apply Corollary~\ref{cor:JB-local-classification} at $t=1$, using
$\rho_{1,A}=\dT$ and the corresponding identity in $B$.
\end{proof}

\Needspace{4\baselineskip}
\begin{remark}\label{rem:affine-cstar-classification}
When $A=\mathcal A_{\sa}$ and $B=\mathcal B_{\sa}$ are the self-adjoint
parts of unital $C^*$-algebras $\mathcal A$ and $\mathcal B$,
Corollary~\ref{cor:JB-local-classification} recovers the local
classification and extension theorem of Hatori for every $t>0$
\cite[Theorem~8]{HatoriExtension}. The Jordan isomorphism $J$ on the
self-adjoint parts extends complex-linearly to a Jordan $*$-isomorphism,
and the central projections of $\mathcal B_{\sa}$ are exactly those
of $\mathcal B$. Indeed, if $p$ is a central projection in $\mathcal B_{\sa}$,
then $U_p=L_p$ gives $pxp=(px+xp)/2$ for $x\in\mathcal B_{\sa}$.
Multiplication on the left and right by $p$ yields $px=xp$.
Complex linearity then gives commutation with all of $\mathcal B$;
the converse is immediate.
\end{remark}

\Needspace{8\baselineskip}
\subsection{Local extension and classification on projective cones}
\label{subsec:projective-local}

For each $t>0$, open subsets below are taken in the $\prho_t$ topology. By
Proposition~\ref{prop:prho-Hilbert} and
Theorem~\ref{thm:projective-JB-GGV}, this is the projective Hilbert
topology and, under a state section $s_\omega$, the norm topology on
$\Sigma_\omega$. The following result is a projective analogue of
Mankiewicz's extension theorem for normed spaces \cite{Mankiewicz}.

\Needspace{4\baselineskip}
\begin{theorem}[Projective Mankiewicz-type extension theorem]
\label{thm:projective-local-global}
Let $A,B$ be unital JB-algebras of dimension at least two and let $t>0$.
Let $U\subset\mathbb P A_+^\circ$ and $V\subset\mathbb P B_+^\circ$
be nonempty open subsets, with $U$ connected. If
\[
 T:U\longrightarrow V
\]
is a bijective $\prho_t$-isometry, then it extends uniquely to a bijective
$\prho_t$-isometry
\[
 \widetilde T:\mathbb P A_+^\circ\longrightarrow\mathbb P B_+^\circ.
\]
\end{theorem}

\begin{proof}
Since $T$ is a homeomorphism onto $V$, the image $V=T(U)$ is connected.
Both projective cones are standard-norm GGVs by
Theorem~\ref{thm:projective-JB-GGV}, so
Corollary~\ref{cor:standard-local-global} applies.
\end{proof}

\Needspace{4\baselineskip}
\begin{corollary}[Projective local classification]
\label{cor:projective-local-classification}
Let $A,B$ be unital JB-algebras of dimension at least two and let $t>0$.
Let $U\subset\mathbb P A_+^\circ$ and $V\subset\mathbb P B_+^\circ$
be nonempty open subsets, with $U$ connected. A map
$T:U\to V$ is a bijective $\prho_t$-isometry if and only if there exist a
Jordan isomorphism $J:A\to B$, a sign $\varepsilon\in\{-1,1\}$, and
$c\in B_+^\circ$ such that
\begin{equation}\label{eq:projective-local-form}
 T(\bar a)=
 \overline{
 \left[U_{c^{t/2}}J(a^{\varepsilon t})\right]^{1/t}}
 \qquad(\bar a\in U),
\end{equation}
and the map defined on all of $\mathbb P A_+^\circ$ by the right-hand
side maps $U$ onto $V$.
In that case the same formula on the whole projective cone is the unique
global extension $\widetilde T$, and $\bar c=\widetilde T(\bar e_A)$,
whether or not $\bar e_A\in U$. The uniqueness of $J$ for each sign
and the sign ambiguity are as in
Theorem~\ref{thm:projective-global-classification}.
\end{corollary}

\begin{proof}
Theorem~\ref{thm:projective-local-global} first gives the unique global
extension $\widetilde T$. Classifying it by
Theorem~\ref{thm:projective-global-classification} and restricting to
$U$ yields \eqref{eq:projective-local-form}; its image of $U$ is
$T(U)=V$.
Conversely, the reverse implication of
Theorem~\ref{thm:projective-global-classification} makes the displayed
global map a bijective $\prho_t$-isometry. Its restriction to $U$ is
injective and has image $V$, so it is a bijective isometry $U\to V$.
By Theorem~\ref{thm:projective-local-global}, that global map is the
unique extension of $T$. Any two parameter choices representing $T$
therefore define the same global map $\widetilde T$, so the parameter
uniqueness and sign ambiguity follow from
Theorem~\ref{thm:projective-global-classification}.
\end{proof}

\Needspace{4\baselineskip}
\begin{corollary}[Local-to-global theorem for Hilbert's metric]
\label{cor:Hilbert-local}
Let $A,B$ be unital JB-algebras of dimension at least two.
Let $T:U\to V$ be a bijective Hilbert isometry between nonempty
open subsets of $\mathbb P A_+^\circ$ and $\mathbb P B_+^\circ$
in the Hilbert topology, with $U$ connected. Then $T$ extends uniquely
to a global bijective Hilbert isometry $\widetilde T$ of the form
\[
 \bar a\longmapsto\overline{U_bJ(a^\varepsilon)},
\]
where $J:A\to B$ is a Jordan isomorphism,
$\varepsilon\in\{-1,1\}$, and $b\in B_+^\circ$.
Here $\widetilde T(\bar e_A)=\overline{b^2}$.
The ray $\bar b$ and $J$ for each admissible sign are unique, with
sign ambiguity as in Corollary~\ref{cor:Hilbert-global}.
Conversely, every global map of this form restricts to a bijective
Hilbert isometry $U\to V$ whenever it maps $U$ onto $V$.
\end{corollary}

\begin{proof}
Apply Corollary~\ref{cor:projective-local-classification} at $t=1$,
using $\prho_1=\dH$ on both projective cones by
Proposition~\ref{prop:prho-Hilbert}, and put $b=c^{1/2}$.
Corollary~\ref{cor:Hilbert-global} applied to $\widetilde T$ gives
the value at $\bar e_A$ and parameter uniqueness. Its converse gives
the assertion about restrictions.
\end{proof}

\Needspace{8\baselineskip}
\subsection{The Jordan algebra determined by local metric geometry}

\Needspace{4\baselineskip}
\begin{corollary}\label{cor:local-Jordan-both}
Let $A,B$ be unital JB-algebras and fix $t>0$.
Let $U\subset A_+^\circ$ and $V\subset B_+^\circ$ be nonempty open
subsets, with $U$ connected. If there is a bijective $\rho_t$-isometry
$T:U\to V$, then $A$ and $B$ are Jordan isomorphic.
The same implication holds for a bijective $\prho_t$-isometry between
such subsets of the projective cones when both algebras have dimension
at least two.
\end{corollary}

\begin{proof}
Apply Corollary~\ref{cor:JB-local-classification} in the affine case and
Corollary~\ref{cor:projective-local-classification} in the projective case.
\end{proof}

\Needspace{4\baselineskip}
\begin{remark}[One-dimensional algebras]\label{rem:one-dimensional-projective}
The projective implication also holds when the projective cone of a
one-dimensional algebra is equipped with its necessarily zero metric,
consistent with $[\mathbb R]=0$.
If both algebras are one-dimensional they are Jordan isomorphic.
If only one is, its projective cone is a singleton, whereas
Theorem~\ref{thm:projective-JB-GGV} identifies the other projective cone
homeomorphically with a nonzero normed space. Every nonempty open
subset of such a space is infinite, so no bijection exists between
the indicated open sets.
\end{remark}

\Needspace{8\baselineskip}
\subsection{Trace-one sections and density matrices}

Every positive invertible ray in $M_n(\mathbb C)_{\sa}$ has a unique
trace-one representative. Write $\mathcal P_n$ for the positive
definite matrices and $\mathcal D_n^\circ=\{a\in\mathcal P_n:
\operatorname{Tr}a=1\}$. The section
$s_{\operatorname{Tr}}(\bar a)=a/\operatorname{Tr}a$ identifies
$\mathbb P\mathcal P_n$ with $\mathcal D_n^\circ$. Equip
$\mathcal D_n^\circ$ with the gyrometric $\prho_t$ transported by
this section. Its topology is the relative norm topology by
Theorem~\ref{thm:projective-JB-GGV}, applied with the state
$\operatorname{Tr}/n$ and followed by rescaling by $1/n$. For $t=1$
the transported distance is Hilbert's metric.

\Needspace{4\baselineskip}
\begin{corollary}[Full-rank density matrices]\label{cor:density-local}
Let $n\ge2$ and $t>0$. Every bijective $\prho_t$-isometry between
nonempty connected open subsets of $\mathcal D_n^\circ$ extends uniquely
to a global bijective isometry of $\mathcal D_n^\circ$. These global
isometries are precisely the maps
\[
 \widetilde T(a)=\frac{X(a)}{\operatorname{Tr}X(a)},\qquad
 X(a)=\bigl(c^{t/2}W\,\chi(a)^{\varepsilon t}W^*c^{t/2}\bigr)^{1/t},
\]
where $c=\widetilde T(I/n)\in\mathcal D_n^\circ$, $W$ is unitary,
$\varepsilon\in\{-1,1\}$, and $\chi$ is the identity or transpose in
a fixed orthonormal basis. For $n\ge3$ the sign is unique. For $n=2$
one may always take $\varepsilon=1$. For each admissible sign, $\chi$
is unique and $W$ is unique up to a scalar of modulus one.
\end{corollary}

\begin{proof}
Transfer Corollary~\ref{cor:projective-local-classification} through
$s_{\operatorname{Tr}}$. Appendix~\ref{app:density-parameters}
identifies the matrix parameters and proves the uniqueness assertions.
\end{proof}

For the Albert algebra $H_3(\mathbb O)$, the same normalization by the
canonical trace identifies the projective positive cone with its
trace-one section, a standard-norm GGV with identity $I_3/3$
\cite[Chapters~II--III and~V]{FarautKoranyi}.

\addtocontents{toc}{\protect\newpage}
\appendix

\Needspace{8\baselineskip}
\section{Auxiliary GGV constructions and continuity}
\label{app:GGV-constructions}

This appendix collects structural criteria and constructions used
in the examples. The construction proofs require the GGV identities
and metric preliminaries of Section~\ref{sec:continuation}, but not
its continuation or classification conclusions. The accompanying
remarks explain the continuity and connectedness hypotheses of the
extension results.

\Needspace{8\baselineskip}
\subsection{Scalar continuity and uniqueness}

\Needspace{4\baselineskip}
\begin{lemma}[Dense parameters]\label{lem:dense-parameters}
Let $G$ be a GGV with real value-line structure and let
$D\subset\mathbb R$ be dense. If $S_q$ is continuous for every
$q\in D$, then $G$ is standard-norm. In particular, either the rational
or the dyadic scalar maps suffice.
\end{lemma}

\begin{proof}
Lemma~\ref{lem:real-value-line-rays} gives, for every $M>0$ and $q,r\in\mathbb R$,
\[
 \sup_{\rho(e,a)\le M}\rho(S_q(a),S_r(a))\le M|q-r|.
\]
Fix $r\in\mathbb R$ and choose $q_n\in D$ with $q_n\to r$. The continuous
maps $S_{q_n}$ converge to $S_r$ uniformly on every bounded set.
Since every point has a bounded neighborhood, $S_r$ is continuous.
Thus $G$ is standard-norm.
\end{proof}

\Needspace{4\baselineskip}
\begin{proposition}[Uniqueness of the scalar multiplication]
\label{prop:scalar-uniqueness}
Fix a gyrogroup operation and a coordinate map $\phi$. There is at most
one scalar multiplication which makes these data a GGV with real
value-line structure.
\end{proposition}

\begin{proof}
Let $\otimes$ and $\odot$ be two such scalar multiplications.
The gyrometric $\rho(a,b)=\|\phi(a\ominus b)\|$ depends only on
$\oplus$ and $\phi$, so both scalar multiplications have the same metric.
By (GGV1)--(GGV3), both maps $b\mapsto 2\otimes b$ and $b\mapsto 2\odot b$
equal $b\mapsto b\oplus b$ and are bijective, with inverses given by
multiplication by $1/2$. Thus, for every $a\in G$, the equation
$b\oplus b=a$ has a unique solution, and hence
\[
 \tfrac12\otimes a=\tfrac12\odot a.
\]
Repeated halving and (GGV2) show agreement at every nonnegative dyadic
rational. For negative dyadic rationals, use
$(-r)\otimes a=\ominus(r\otimes a)$
\cite[Proposition~2.1(v)]{AHrevisited}, which holds for both scalar
multiplications. For fixed $a$, the maps $r\mapsto r\otimes a$ and
$r\mapsto r\odot a$ are continuous in the common metric by
Lemma~\ref{lem:real-value-line-rays}. Since the dyadic rationals are dense
in $\mathbb R$, they agree for every real $r$.
\end{proof}

Recall that a gyrogroup with a topology is called a
\emph{topological gyrogroup} if gyroaddition is jointly continuous and
inversion is continuous \cite[Definition~1]{Atiponrat2017}.
The gyrometric topologies considered here are Hausdorff.
A topological gyrogroup is \emph{strongly topological} if, in addition,
the identity has a neighborhood base $\mathcal U$ such that
$\gyr[a,b](U)=U$ for every $U\in\mathcal U$ and $a,b\in G$;
see \cite[p.~5116]{BaoLin2019}. In the proposition below, the base
can be chosen symmetric, that is, $\ominus U=U$ for every $U\in\mathcal U$.

\Needspace{4\baselineskip}
\begin{proposition}[Topological characterization]
\label{prop:standard-topological}
Let $G$ be a GGV with real value-line structure, equipped with its
gyrometric topology. The following conditions are equivalent:
\begin{enumerate}[label=\textup{(\roman*)}]
\item $G$ is a standard-norm GGV.
\item The underlying gyrogroup $(G,\oplus)$ is topological.
\item The underlying gyrogroup $(G,\oplus)$ is strongly topological.
\end{enumerate}
When these conditions hold, scalar multiplication is jointly continuous
and $G$ is contractible.
\end{proposition}

\begin{proof}
Write $R_a(x)=x\oplus a$ for right gyrotranslation by $a$.
We first record identities which do not require
continuity. Put
\[
 D_a(x)=x\boxplus(\ominus a)
       =x\oplus\gyr[x,a](\ominus a).
\]
Coaddition is commutative in a gyrocommutative gyrogroup. Hence the
gyromidpoint formula in Definition~\ref{def:gyromidpoint} gives
\[
 D_a(x)=2\otimes P(\ominus a,x)
 =S_2\bigl(L_{\ominus a}(S_{1/2}(L_a(x)))\bigr).
\]
Since $L_a^{-1}=L_{\ominus a}$ and $S_2^{-1}=S_{1/2}$, the map $D_a$
is bijective. The left loop property and the left gyroassociative law
\cite[Definition~2.7, (G5) and (G3), p.~17]{Ungar2008} give
\[
 \begin{aligned}
 D_a(R_a(x))
 &=(x\oplus a)\oplus\gyr[x\oplus a,a](\ominus a)\\
 &=(x\oplus a)\oplus\gyr[x,a](\ominus a)\\
 &=x\oplus\bigl(a\oplus(\ominus a)\bigr)=x.
 \end{aligned}
\]
Thus $R_a=D_a^{-1}$ and $R_a^{-1}=D_a$. Inverting the displayed
composition for $D_a$, we obtain
\begin{equation}\label{eq:right-translation-scalar}
 \begin{aligned}
 R_a&=L_{\ominus a}\circ S_2\circ L_a\circ S_{1/2},\\
 S_2\circ L_a&=L_a\circ R_a\circ S_2.
 \end{aligned}
\end{equation}

Suppose \textup{(i)} holds. Left translations are isometries, and
$S_2$ and $S_{1/2}$ are continuous. The first identity in
\eqref{eq:right-translation-scalar} shows that every right translation
is continuous. For $x,x',y,y'\in G$, left invariance gives
\[
 \rho(x'\oplus y',x\oplus y)
 \le \rho(y',y)+\rho(R_y(x'),R_y(x)).
\]
Thus gyroaddition is jointly continuous. Inversion is the continuous
map $S_{-1}$, so \textup{(ii)} holds.

Now suppose \textup{(ii)} holds. The gyrator identity
\cite[Theorem~2.10(10), p.~19]{Ungar2008}, applied to $x,a,\ominus a$,
gives $\gyr[x,a](\ominus a)=\ominus(x\oplus a)\oplus x$. Hence
\[
 R_a^{-1}(x)=D_a(x)=x\oplus\bigl(\ominus(x\oplus a)\oplus x\bigr).
\]
Both $R_a$ and its inverse are continuous, since gyroaddition and
inversion are continuous. Also $L_a^{-1}=L_{\ominus a}$, so left and
right translations are homeomorphisms. This argument does not use
continuity of $S_{1/2}$. The map $S_2(x)=x\oplus x$ is a continuous
bijection. By (GGV7) and bijectivity of $S_2$, we have
$S_2(B_\rho(e,\varepsilon))=B_\rho(e,2\varepsilon)$ for every
$\varepsilon>0$. Using left invariance and the second identity in
\eqref{eq:right-translation-scalar}, we obtain
\[
 S_2(B_\rho(a,\varepsilon))
 =L_a\bigl(R_a(B_\rho(e,2\varepsilon))\bigr).
\]
This set is open. Therefore $S_2$ is open, and its inverse $S_{1/2}$
is continuous. The maps $S_m$, $m\in\mathbb Z$, are continuous by
repeated gyroaddition and inversion. It follows from (GGV3) that every
dyadic scalar map $S_{m/2^n}=S_m\circ(S_{1/2})^n$ is continuous.
Lemma~\ref{lem:dense-parameters} now gives continuity of every $S_r$,
proving \textup{(i)}.

Finally, symmetry of the metric gives $\rho(e,x)=\rho(x,e)=\|\phi(x)\|$.
Since $\ominus x=(-1)\otimes x$, (GGV7) yields
$\rho(e,\ominus x)=\rho(e,x)$, so the balls at $e$ are symmetric.
By (GGV0) and the fact that each $\gyr[a,b]$ is an automorphism,
\[
 \gyr[a,b](B_\rho(e,\varepsilon))=B_\rho(e,\varepsilon)
 \qquad(a,b\in G,\ \varepsilon>0).
\]
These balls form a symmetric invariant neighborhood base, proving
\textup{(ii)}$\Rightarrow$\textup{(iii)}. The converse is part of the
definition. The last assertion follows from
Proposition~\ref{prop:joint-scalar-continuity} and
Corollary~\ref{cor:contractible}.
\end{proof}

\Needspace{4\baselineskip}
\begin{remark}[The associative case]\label{rem:associative-standard}
If a GGV $G$ with real value-line structure has associative $\oplus$, then
$(G,\oplus,\otimes)$ is a real normed space with $N(a)=\|\phi(a)\|$.
For topological gyrovector spaces, distributivity when all gyrations are
trivial is proved in \cite[Theorem~4.1]{KimDistributivity2015}.
Here it follows without assuming continuity of the fixed-scalar maps.

Indeed, associativity makes all gyrations trivial, and gyrocommutativity
makes $(G,\oplus)$ an abelian group.  Translation invariance and the
triangle inequality give
\[
 \rho(x\oplus y,x'\oplus y')\le\rho(x,x')+\rho(y,y'),
\]
so addition is jointly continuous. By (GGV1)--(GGV3), $S_m$ is additive
for every nonnegative integer $m$. For $m<0$, use
$S_m=\ominus\circ S_{-m}$ and additivity of inversion in an abelian group.
For $n>0$, the inverse of $S_n$ is the additive map $S_{1/n}$.
Thus $S_{m/n}=S_m\circ S_{1/n}$ is additive.
Joint continuity of
addition and the scalar-ray continuity in
Lemma~\ref{lem:real-value-line-rays} then give
\[
 r\otimes(a\oplus b)=(r\otimes a)\oplus(r\otimes b)
 \qquad(r\in\mathbb R).
\]
The remaining vector-space axioms are (GGV1)--(GGV3), and $N$ is a norm
by injectivity of $\phi$, $\phi(e)=0$, and (GGV7)--(GGV8).
By \eqref{eq:reflection}, inversion is the isometric point reflection
$\sigma_e$. With joint continuity of addition, this gives \textup{(ii)}
of Proposition~\ref{prop:standard-topological}, so $G$ is standard-norm
and contractible.
\end{remark}

\Needspace{4\baselineskip}
\begin{remark}[Simple connectedness and monodromy]
\label{rem:real-value-line-problem}
Under real value-line structure alone,
Proposition~\ref{prop:metric-geodesic} gives path connectedness. We do not
know whether contractibility or simple connectedness also follows.
Fixed-scalar continuity ensures continuity of the canonical contraction
by Proposition~\ref{prop:joint-scalar-continuity} and
Corollary~\ref{cor:contractible}. In Example~\ref{ex:fixed-scalar-counterexample}, the doubling map is
discontinuous, yet the scalars in $[0,1]$ still give a contraction.

A separate question is whether nontrivial continuation monodromy occurs.
Remark~\ref{rem:associative-standard} and Corollary~\ref{cor:contractible}
show that nontrivial monodromy requires a nonassociative gyrogroup with
a discontinuous fixed-scalar map on the side where the loop is taken, namely
$G_1$ for $T$ and $G_2$ for $T^{-1}$. For a fixed gyrogroup and coordinate map,
Proposition~\ref{prop:scalar-uniqueness} gives uniqueness of the scalar
multiplication under real value-line structure, so replacing the scalar
multiplication cannot restore continuity while retaining that structure.
The exact extension criterion for a given bijective isometry between
connected open subsets is Theorem~\ref{thm:two-sided-monodromy}.
\end{remark}

\Needspace{4\baselineskip}
\begin{remark}[Necessity of connectedness]\label{rem:connectedness}
Connectedness cannot be dropped from
Theorem~\ref{thm:general-local-global}, even within the standard-norm
class. See Remark~4 and Example~8 of \cite{Mankiewicz} and
Example~10 of \cite{HatoriExtension}. The following example
illustrates the obstruction.

Equip $X=\mathbb R^2$ with the supremum norm and its usual vector-space
operations. This is a standard-norm GGV. Put
\[
 A=(-3,-2)\times(-1/4,1/4),\qquad
 B=(2,3)\times(-1/4,1/4),\qquad U=A\cup B,
\]
and define $T:U\to U$ by
\[
 T(x,y)=
 \begin{cases}
 (x,y),&(x,y)\in A,\\
 (x,-y),&(x,y)\in B.
 \end{cases}
\]
The set $U$ is nonempty and open but not connected. The map $T$ is
bijective and preserves distances within each component. If
$p=(x,y)\in A$ and $q=(u,v)\in B$, then
$|u-x|>4$, whereas $|y-v|<1/2$ and $|y+v|<1/2$. Hence
\[
 \|p-q\|_\infty=|u-x|=\|T(p)-T(q)\|_\infty.
\]
Thus $T$ is a bijective isometry of $U$.

Suppose that $T$ extended to a surjective isometry
$\widetilde T:X\to X$. The maps $\widetilde T$ and
$\operatorname{id}_X$ are local gyrometric isometries on the connected
space $X$ and agree on the nonempty open set $A$.
Lemma~\ref{lem:identity-principle} forces
$\widetilde T=\operatorname{id}_X$, contradicting
\[
 T(5/2,1/8)=(5/2,-1/8).
\]
Therefore $T$ has no global surjective isometric extension.

The identity-principle step in Theorem~\ref{thm:two-sided-monodromy}
extends agreement with the initial germ only across its connected
component. Here the identity map extends the germ on $A$ but not $T|_B$.
Thus connectedness cannot be omitted in general, although particular
isometries on disconnected open sets may still extend.

Hatori's example \cite[Example~10]{HatoriExtension} uses two disjoint
open balls in $C_{\mathbb R}(\{1,2\})$ with the supremum norm and
different isometries on the two components. Exponentiating each coordinate
transfers the examples to the positive invertible cone of $\mathbb C^2$,
since
\[
 d_T(a,b)=\max_{i=1,2}|\log a_i-\log b_i|.
\]
Theorem~\ref{thm:JB-GGV} in Section~\ref{sec:cone-linearization} gives the
standard-norm GGV structure, and Corollary~\ref{cor:rho-Thompson} in
Section~\ref{sec:cone-linearization} identifies its gyrometric at $t=1$
with $d_T$.
\end{remark}

\Needspace{8\baselineskip}
\subsection{Standardization and transport}

\Needspace{4\baselineskip}
\begin{proposition}[Radial standardization]\label{prop:radial-standardization}
Let $(G,\oplus,\otimes,\phi_0)$ be a GGV with coordinate map
$\phi_0:G\to(V,\|\cdot\|)$ such that $\phi_0(e)=0$. Recall
that its one-dimensional norm-value space is
\[
 L=\|\phi_0(G)\|=\{\,\pm\|\phi_0(x)\|:x\in G\,\},
\]
with the operations $\oplus'$ and $\otimes'$ from \textup{(GGVV)}.
Suppose that there is a vector-space isomorphism
\[
  \chi:(L,\oplus',\otimes')\longrightarrow(\mathbb R,+,\cdot)
\]
that is strictly increasing for the usual order on $L\subset\mathbb R$,
the order used in (GGV8) of Definition~\ref{def:GGV} \cite{AHrevisited}.
Define $\psi:G\to V$ by $\psi(e)=0$ and, for $a\ne e$,
\begin{equation}\label{eq:radial-standardization}
 \psi(a)=
 \frac{\chi(\|\phi_0(a)\|)}{\|\phi_0(a)\|}\,\phi_0(a).
\end{equation}
Then $(G,\oplus,\otimes,\psi)$ is a GGV with real value-line structure.  Its gyrometric is
\begin{equation}\label{eq:standardized-metric}
 \rho_\psi(a,b)=\chi\bigl(\|\phi_0(a\ominus b)\|\bigr).
\end{equation}
Every other increasing linear isomorphism $L\to\mathbb R$ is $c\chi$ for
a unique $c>0$ and multiplies this gyrometric by $c$.
\end{proposition}

\begin{proof}
Since $\phi_0(e)=0$, \cite[Proposition~2.1(i)]{AHrevisited} identifies
the additive identity of $L$ with $0$, so $\chi(0)=0$.
For $a\ne e$, injectivity of $\phi_0$ and strict monotonicity of $\chi$
make the coefficient in \eqref{eq:radial-standardization} positive. Hence
\[
 \|\psi(a)\|=\chi(\|\phi_0(a)\|),\qquad
 \frac{\psi(a)}{\|\psi(a)\|}
 =\frac{\phi_0(a)}{\|\phi_0(a)\|}\quad(a\ne e),
\]
and the norm identity also holds at $e$.

If $\psi(a)=\psi(b)$, the norm identity and injectivity of $\chi$
give $\|\phi_0(a)\|=\|\phi_0(b)\|$. If this value is zero, then
$a=b=e$. Otherwise the coefficients in \eqref{eq:radial-standardization}
are equal and positive, so $\phi_0(a)=\phi_0(b)$ and $a=b$.
Thus $\psi$ is injective.
The nonnegative part of $L$ is precisely
$\{\|\phi_0(a)\|:a\in G\}$. Since $\chi$ is an increasing bijection
with $\chi(0)=0$, it maps this set onto $[0,\infty)$: for $t\ge0$,
strict monotonicity gives $\chi^{-1}(t)\ge0$. The norm identity therefore
gives $\{\pm\|\psi(a)\|:a\in G\}=\mathbb R$.
We equip this value line with the ordinary real operations, proving
(GGVV) with the real value line.

The norm identity transfers (GGV0) directly. The original
gyrogroup and scalar operations give (GGV1)--(GGV3), (GGV5), and (GGV6).
For (GGV7), linearity of $\chi$ for the given value-line operations gives
\[
\begin{aligned}
 \|\psi(r\otimes a)\|
 &=\chi\bigl(\|\phi_0(r\otimes a)\|\bigr)
 =\chi\bigl(|r|\otimes'\|\phi_0(a)\|\bigr)\\
 &=|r|\,\chi\bigl(\|\phi_0(a)\|\bigr)=|r|\|\psi(a)\|.
\end{aligned}
\]
For $a\ne e$ and $r\ne0$, the original (GGV7) applied to $r$ and $|r|$ gives
\[
 \|\phi_0(r\otimes a)\|=|r|\otimes'\|\phi_0(a)\|
 =\|\phi_0(|r|\otimes a)\|.
\]
This value is nonzero, since $\chi$ maps it to $|r|\|\psi(a)\|>0$.
In particular, $|r|\otimes a\ne e$, and the original (GGV4) shows that
$\phi_0(|r|\otimes a)$ and $\phi_0(a)$ have the same normalized direction.
The direction identity above and the new (GGV7) therefore give
\[
 \frac{\psi(|r|\otimes a)}{\|\psi(r\otimes a)\|}
 =\frac{\psi(|r|\otimes a)}{\|\psi(|r|\otimes a)\|}
 =\frac{\psi(a)}{\|\psi(a)\|},
\]
where the denominators are positive since
$\|\psi(r\otimes a)\|=|r|\|\psi(a)\|>0$. This is (GGV4).
The comparison in (GGV8) involves only nonnegative values: the sum
on its right is carried by $\chi$ to a sum of nonnegative reals.
No identification of the additive inverse in $L$ with ordinary
real negation is needed. The original (GGV8), monotonicity, and additivity of $\chi$
give the new (GGV8):
\[
 \|\psi(a\oplus b)\|
 \le\chi\bigl(\|\phi_0(a)\|\oplus'\|\phi_0(b)\|\bigr)
 =\|\psi(a)\|+\|\psi(b)\|.
\]
Thus the new GGV has real value-line structure.
Applying the norm identity to $a\ominus b$ proves
\eqref{eq:standardized-metric}. For any other such isomorphism $\eta$,
the increasing real-linear automorphism $\eta\circ\chi^{-1}$ of
$\mathbb R$ is multiplication by its value $c>0$ at $1$.
\end{proof}

The original gyrometric $\rho_0$ has range $L\cap[0,\infty)$, since
$\rho_0(c,e)=\|\phi_0(c)\|$. Thus \eqref{eq:standardized-metric} gives
$\rho_\psi=\rho_0$ exactly when $\chi$ fixes $L\cap[0,\infty)$ pointwise.
For a GGV already having real value-line structure, the choice
$\chi=\id_{\mathbb R}$ gives $\psi=\phi_0$, as in Example~\ref{ex:normed-space}.

After radial standardization, the standard-norm property amounts to
fixed-scalar continuity for the new gyrometric or, equivalently by
Proposition~\ref{prop:standard-topological}, continuity of gyroaddition
and inversion. Here all topological notions refer to $\rho_\psi$.
Write $B_0(a,s)=\{b:\rho_0(a,b)<s\}$ for every real $s>0$.
Since $B_{\rho_\psi}(a,t)=B_0(a,\chi^{-1}(t))$ for $t>0$, the
$\rho_\psi$-topology always has as a base the balls of radii in $L\cap(0,\infty)$.
For arbitrary real radii, consider the topology in which $U$ is open
if every $a\in U$ has some $B_0(a,s)\subset U$. This topology agrees
with the $\rho_\psi$-topology exactly when $\inf(L\cap(0,\infty))=0$.
Indeed, in that case every $B_0(a,s)$ contains a ball about $a$ of
positive radius in $L$. If the infimum is positive, sufficiently small
real radii give singletons, whereas the standardized space contains a
nonconstant continuous scalar ray through $e$ by Lemma~\ref{lem:real-value-line-rays}.

This criterion concerns neighborhoods at their centers. To have
\emph{every} $B_0(a,s)$ open in the standardized metric, the increasing
function $h=\chi^{-1}|_{[0,\infty)}$ must be right-continuous there.
Indeed, $B_0(a,s)=\{b:h(\rho_\psi(a,b))<s\}$ is open under that condition.
Conversely, if $h(t)<s<\lim_{u\downarrow t}h(u)$, then
$B_0(e,s)$ contains the point at distance $t$ on a unit-speed scalar
ray but excludes all subsequent points of that ray.

\Needspace{4\baselineskip}
\begin{proposition}[Transport of a standard-norm GGV structure]\label{prop:transport-standard}
Let $G_1$ be a set equipped with maps
\[
 \oplus_1:G_1\times G_1\longrightarrow G_1,
 \qquad
 \otimes_1:\mathbb R\times G_1\longrightarrow G_1,
\]
and let
\[
 (G_2,\oplus_2,\otimes_2,\phi_2)
\]
be a standard-norm GGV with coordinate map $\phi_2:G_2\to V$.
Suppose that $\Theta:G_1\to G_2$ is a bijection such that
\[
 \Theta(a\oplus_1 b)=\Theta(a)\oplus_2\Theta(b),
 \qquad
 \Theta(r\otimes_1 a)=r\otimes_2\Theta(a)
\]
for all $a,b\in G_1$ and $r\in\mathbb R$.  Put
\[
 \phi_1=\phi_2\circ\Theta.
\]
Then $(G_1,\oplus_1,\otimes_1,\phi_1)$ is a standard-norm GGV, and $\Theta$
is an isometry for the resulting gyrometrics.  In particular, $G_1$ is
contractible and simply connected. Both underlying gyrogroups are
strongly topological for their gyrometric topologies. If the coordinate
map $\phi_2$ is continuous for the gyrometric topology of $G_2$, then
$\phi_1$ is continuous for the gyrometric topology of $G_1$.
\end{proposition}

\begin{proof}
The intertwining identities and bijectivity of $\Theta$ transport
the gyrogroup and scalar operations from $G_2$ to $G_1$.
Writing $e_2$ for the identity of $G_2$, we obtain
\[
 e_1=\Theta^{-1}(e_2),\qquad
 \ominus_1a=\Theta^{-1}(\ominus_2\Theta(a)),
\]
\[
 \gyr_1[a,b]
 =\Theta^{-1}\circ\gyr_2[\Theta(a),\Theta(b)]\circ\Theta.
\]
These identities transfer the gyrocommutative gyrogroup axioms and
(GGV1)--(GGV3), (GGV5), and (GGV6).
The map $\phi_1=\phi_2\circ\Theta$ is injective and has image
$\phi_2(G_2)$. Equip the common norm-value line with the operations
of $G_2$, namely the ordinary real operations. The intertwining
identities then transfer (GGVV), (GGV0), (GGV4), (GGV7), and (GGV8).
For (GGV4), injectivity gives $\Theta(a)\ne e_2$ when $a\ne e_1$;
for $r\ne0$, (GGV7) in $G_2$ gives
$\|\phi_2(r\otimes_2\Theta(a))\|=|r|\|\phi_2(\Theta(a))\|>0$.
Thus the denominators are nonzero.

The inverse identity above gives
$\Theta(a\ominus_1b)=\Theta(a)\ominus_2\Theta(b)$, so
\[
 \rho_1(a,b)
 =\|\phi_2(\Theta(a)\ominus_2\Theta(b))\|
 =\rho_2(\Theta(a),\Theta(b)).
\]
Thus $\Theta$ is a bijective isometry and a homeomorphism.
For $S_{i,r}(a)=r\otimes_i a$, the identity
\[
 S_{1,r}=\Theta^{-1}\circ S_{2,r}\circ\Theta
\]
transfers fixed-scalar continuity, proving that $G_1$ is standard-norm.
Corollary~\ref{cor:contractible} and
Proposition~\ref{prop:standard-topological} give the stated topological
consequences. If $\phi_2$ is continuous, then so is
$\phi_1=\phi_2\circ\Theta$.
\end{proof}

\Needspace{8\baselineskip}
\subsection{Products and bounded mapping spaces}
\label{sec:products-functions}

Finite and countable $\ell^p$-products, $1\le p\le\infty$, and bounded
mapping spaces preserve real value-line structure. For bounded maps on
an infinite set, the standard-norm criterion is uniform continuity of
each fixed scalar map on bounded subsets of the target GGV
(Proposition~\ref{prop:bounded-maps-continuity}).
Example~\ref{ex:fixed-scalar-counterexample} also gives a countable $\ell^\infty$-product of
standard-norm factors with a discontinuous fixed-scalar map.
Real value-line structure allows the component norms to be combined:
their scalar homogeneity and triangle inequalities use the same
ordinary real operations.

For a GGV with real value-line structure, the coordinate map sends the
identity to $0$ by \cite[Proposition~2.1(i)]{AHrevisited}, since the
additive identity of its value line is the ordinary real zero.
Also, the gyroassociative law and left cancellation give
\[
 \gyr[x,y]z=\ominus(x\oplus y)\oplus\bigl(x\oplus(y\oplus z)\bigr).
\]
Thus gyrations in products and mapping spaces are computed
componentwise or pointwise. Their inverses are given by
$\gyr[x,y]^{-1}=\gyr[y,x]$ \cite[Theorem~2.34]{Ungar2008}.

For earlier product constructions using sums of normalized norms, see
\cite[Propositions~3.3 and~3.8]{Abe2016}.

\Needspace{4\baselineskip}
\begin{proposition}[Finite and countable $\ell^p$-products]\label{prop:lp-products}
Let $I=\{1,\ldots,m\}$ with $m\ge1$ or $I=\mathbb N$. For each
$j\in I$, let $(G_j,\oplus_j,\otimes_j,\phi_j)$ be a GGV with real
value-line structure, identity $e_j$, coordinate map
$\phi_j:G_j\to(V_j,\|\cdot\|_j)$, and gyrometric $\rho_j$.
For $1\le p\le\infty$, set
\[
 G_p=\bigl\{x\in\textstyle\prod_{j\in I}G_j:
            (\rho_j(e_j,x_j))_{j\in I}\in\ell^p(I)\bigr\},
 \qquad V_p=\bigl(\textstyle\bigoplus_{j\in I}V_j\bigr)_p,
\]
where $\|(v_j)\|_{V_p}=\|(\|v_j\|_j)\|_{\ell^p(I)}$.
With componentwise operations and coordinate map
$\Phi(x)=(\phi_j(x_j))_{j\in I}$, $G_p$ is a GGV with real value-line
structure and gyrometric
\[
 \rho_p(x,y)=\|(\rho_j(x_j,y_j))_{j\in I}\|_{\ell^p(I)}.
\]
If $I$ is finite or $p<\infty$, then $G_p$ is standard-norm if and only
if every factor is standard-norm. In that case scalar multiplication is
jointly continuous and $G_p$ is contractible and simply connected.
For $I=\mathbb N$ and $p=\infty$, standard-norm factors alone do not
suffice (Example~\ref{ex:fixed-scalar-counterexample}).
\end{proposition}

\begin{proof}
Since $\rho_j(e_j,x_j)=\|\phi_j(x_j)\|_j$, $\Phi$ takes values in $V_p$
and is injective. Taking the $\ell^p(I)$-norm in the coordinate
identities (GGV7) and inequalities (GGV8) gives
\[
 \|\Phi(r\otimes x)\|_{V_p}=|r|\,\|\Phi(x)\|_{V_p},\qquad
 \|\Phi(x\oplus y)\|_{V_p}\le\|\Phi(x)\|_{V_p}+\|\Phi(y)\|_{V_p}.
\]
The inequality uses monotonicity on nonnegative tuples and the triangle
inequality, also for the supremum norm. Thus $G_p$ contains $(e_j)_j$
and is closed under the operations, including inverses since
$\ominus_j x_j=(-1)\otimes_j x_j$ \cite[Proposition~2.1(v)]{AHrevisited}.
Gyrations and their inverses preserve each coordinate norm by (GGV0),
so they restrict to automorphisms of $G_p$. The gyrogroup identities
and (GGV1)--(GGV3), (GGV5), and (GGV6) hold componentwise.
For $t>0$, (GGV4) and (GGV7) in each factor give
$\phi_j(t\otimes_j x_j)=t\phi_j(x_j)$, including at $e_j$ since
$\phi_j(e_j)=0$. This gives (GGV4) for $\Phi$ using the displayed
norm identity. Embedding a single factor with all other coordinates
equal to their identities preserves the coordinate norm. Thus the
signed norm-value set is $\mathbb R$, and (GGVV) holds with ordinary
real operations. Applying the norm to $\Phi(x\ominus y)$ gives the
stated formula for $\rho_p$.

Each single-coordinate embedding is isometric and commutes with
scalar multiplication, so the standard-norm property of $G_p$ implies
that of each factor. Conversely, suppose all factors are standard-norm.
For finite $I$ the $\ell^p$ metric induces the finite product topology
for every $p$, including $p=\infty$. Each fixed-scalar map is then
continuous by coordinatewise continuity.

Now let $I=\mathbb N$, $p<\infty$, and $\rho_p(x^{(n)},x)\to0$.
Fix $r\in\mathbb R$ and put
$d_{j,n}=\rho_j(r\otimes_j x^{(n)}_j,r\otimes_j x_j)$.
For each $j$, coordinatewise convergence and fixed-scalar continuity
give $d_{j,n}\to0$. For the tail, the triangle inequality through $e_j$
and (GGV7) give
\[
 d_{j,n}\le |r|\rho_j(x^{(n)}_j,x_j)+2|r|\rho_j(e_j,x_j).
\]
Writing $T_N(x)=(\sum_{j>N}\rho_j(e_j,x_j)^p)^{1/p}$, we obtain
\[
 \rho_p(r\otimes x^{(n)},r\otimes x)
 \le \Bigl(\sum_{j=1}^N d_{j,n}^p\Bigr)^{1/p}
       + |r|\rho_p(x^{(n)},x)+2|r|T_N(x).
\]
For fixed $N$ the first two terms tend to zero, and $T_N(x)\to0$ as
$N\to\infty$. This proves fixed-scalar continuity. Joint continuity
and contractibility now follow from
Proposition~\ref{prop:joint-scalar-continuity} and
Corollary~\ref{cor:contractible}.
\end{proof}

\Needspace{4\baselineskip}
\begin{proposition}[Bounded mapping spaces and fixed-scalar continuity]
\label{prop:bounded-maps-continuity}
Let $(G,\oplus,\otimes,\phi)$ be a GGV with real value-line structure,
identity $e$, coordinate map $\phi:G\to(V,\|\cdot\|)$, and gyrometric
$\rho$. For a nonempty set $\Lambda$, put
\[
 \mathcal B(\Lambda,G)=
 \{f:\Lambda\to G:\sup_{s\in\Lambda}\rho(e,f(s))<\infty\}.
\]
With pointwise operations and coordinate map
$\Phi(f)=\phi\circ f\in\ell^\infty(\Lambda,V)$, this is a GGV with
real value-line structure and gyrometric
\begin{equation}\label{eq:sup-gyrometric}
 \rho_\infty(f,g)=\sup_{s\in\Lambda}\rho(f(s),g(s)).
\end{equation}
For infinite $\Lambda$, $\mathcal B(\Lambda,G)$ is standard-norm if
and only if every $S_r:a\mapsto r\otimes a$, $r\in\mathbb R$, is
uniformly continuous on each $\rho$-bounded subset of $G$.
For finite $\Lambda$, $\mathcal B(\Lambda,G)$ is standard-norm if and
only if $G$ is standard-norm.
\end{proposition}

\begin{proof}
Since $\rho(e,a)=\|\phi(a)\|$, boundedness of $f$ is equivalent to
$\Phi(f)\in\ell^\infty(\Lambda,V)$. The proof of the GGV axioms in
Proposition~\ref{prop:lp-products} for $p=\infty$ applies to the
supremum over any index set $\Lambda$. Indeed, (GGV7)--(GGV8) give
closure under pointwise scalar multiplication and gyroaddition, and
gyrations and their inverses preserve the coordinate norms pointwise.
Positive scalar homogeneity of $\Phi$ follows pointwise from that of
$\phi$. Constant functions realize all signed norm values since
$\Lambda$ is nonempty. Evaluating $\|\Phi(f\ominus g)\|_\infty$ gives
\eqref{eq:sup-gyrometric}. For finite $\Lambda$ the standard-norm
criterion now follows from the finite product case of
Proposition~\ref{prop:lp-products}.

Write $\mathcal S_r(f)=S_r\circ f$. If $S_r$ is uniformly continuous
on bounded sets and $\rho_\infty(f_n,f)\to0$, all the ranges of $f_n$
and $f$ lie in one bounded ball. Uniform continuity on that ball gives
$\rho_\infty(\mathcal S_r(f_n),\mathcal S_r(f))\to0$.

Conversely, suppose $\Lambda$ is infinite and some $S_r$ fails to be
uniformly continuous on a bounded set. Choose $a_n,b_n$ in that set
and $\varepsilon>0$ with $\rho(a_n,b_n)\to0$ but
$\rho(S_r(a_n),S_r(b_n))\ge\varepsilon$. At distinct points
$s_n\in\Lambda$, set $f(s_n)=a_n$ and take $f=e$ elsewhere. Obtain
$g_k$ from $f$ by replacing $a_n$ by $b_n$ at $s_n$ for all $n\ge k$.
These functions are bounded and satisfy
\[
 \rho_\infty(g_k,f)=\sup_{n\ge k}\rho(a_n,b_n)\longrightarrow0,
 \qquad
 \rho_\infty(\mathcal S_r(g_k),\mathcal S_r(f))\ge\varepsilon.
\]
Thus $\mathcal S_r$ is discontinuous at $f$.
\end{proof}

\Needspace{5\baselineskip}
\begin{example}[Failure of uniform continuity on bounded sets]
\label{ex:fixed-scalar-counterexample}
The bounded-set condition holds in normed spaces since $S_r$ is
$|r|$-Lipschitz. To see how it can fail in a standard-norm GGV, take
the unit disk $D\subset\mathbb C$ with the structure of
Proposition~\ref{prop:bsd-ggv} for $E=\mathbb C$. Its coordinate map
$L$ satisfies $L(e^{i\alpha}\tanh t)=e^{i\alpha}t$ for $t\ge0$,
so $S_2$ doubles $t$ and preserves the angle. Write $d_M=k_D$.
The disk automorphism taking $z$ to $0$ gives
$d_M(z,w)=\operatorname{arctanh}|(w-z)/(1-\bar z w)|$.
Since $|1-\bar z w|^2-|w-z|^2=(1-|z|^2)(1-|w|^2)$, we obtain
\[
 \sinh d_M(z,w)=\frac{|w-z|}{\sqrt{(1-|z|^2)(1-|w|^2)}}.
\]
In particular, for points of modulus $\tanh t$ with angle
$\theta\in[0,\pi]$ at the origin, the right-hand side is
$2\tanh t\sin(\theta/2)/(1-\tanh^2 t)
=\sinh(2t)\sin(\theta/2)$.

Give the $j$th copy of $D$ the coordinate map $L/j$. By
Proposition~\ref{prop:radial-standardization}, applied with
$\chi(u)=u/j$, its gyrometric is $d_M/j$ and its scalar maps are
unchanged. By Proposition~\ref{prop:lp-products}, the
$\ell^p$-product $P$ of these disks is standard-norm for $1\le p<\infty$.
In the $j$th disk take $t=j$ and choose
$\sin(\theta_j/2)=\sinh(j)/\sinh(4j)$. Let $a_j,b_j\in P$ have these
two values in coordinate $j$ and zero elsewhere. Both have distance
$1$ from the identity. Since they differ only in coordinate $j$,
$\rho_p(a_j,b_j)$ is the distance of those two disk values divided by
$j$, and tends to zero, since
\[
 \sinh\bigl(j\rho_p(a_j,b_j)\bigr)
 =\frac{\sinh(2j)\sinh(j)}{\sinh(4j)}
 =\frac{\sinh(j)}{2\cosh(2j)}\longrightarrow0.
\]
Under $S_2$ the radius parameter becomes $2j$ and the angle stays
$\theta_j$. Thus $\sinh(j\rho_p(S_2a_j,S_2b_j))=\sinh(j)$, so the
image distance is $1$. Thus $S_2$ fails to be uniformly continuous
on the unit sphere about the identity in $P$. By
Proposition~\ref{prop:bounded-maps-continuity}, $\mathcal B(\Lambda,P)$
fails to be standard-norm for infinite $\Lambda$.

For the countable $\ell^\infty$-product of the same rescaled disks, write
\[
 P_\infty=\left\{z=(z_j)\in D^{\mathbb N}:
       \sup_{j\ge1}\frac{d_M(0,z_j)}{j}<\infty\right\},
\]
with gyrometric
\[
 \rho_\infty(z,w)=\sup_{j\ge1}\frac{d_M(z_j,w_j)}{j}.
\]
Let $a$ have in every coordinate $j$ the nonzero disk value of $a_j$,
and let $b^{(j)}$ agree with $a$ except that coordinate $j$ is replaced
by the disk value of $b_j$. The points $a,b^{(j)}$ have distance $1$
from the identity, and their $S_2$-images have distance $2$, so all
belong to $P_\infty$. By the same calculation,
$\rho_\infty(b^{(j)},a)\to0$ but
$\rho_\infty(S_2b^{(j)},S_2a)=1$.
Thus $S_2$ is discontinuous at $a$, and this $\ell^\infty$-product
is not standard-norm.

Nevertheless $P_\infty$ is contractible. In the disk coordinates
$z=e^{i\theta}\tanh t$, the line element $|dz|/(1-|z|^2)$ has square
$dt^2+\tfrac14\sinh^2(2t)\,d\theta^2$. For $0\le s\le1$ the radial
map $S_s(t,\theta)=(st,\theta)$ multiplies lengths by at most $s$,
since $\sinh(2st)\le s\sinh(2t)$. Hence it is $s$-Lipschitz for
$d_M$, and the same estimate holds for $\rho_\infty$ after division
by $j$ in each coordinate and passage to the supremum. Along each
ray the radial distance is linear in the parameter, so
\[
 \rho_\infty(S_s x,S_u y)
 \le s\rho_\infty(x,y)+|s-u|\rho_\infty(0,y)
 \qquad (0\le s,u\le1).
\]
Thus $(x,s)\mapsto S_sx$ is jointly continuous, with $S_1$ the identity
and $S_0$ constant. In particular $P_\infty$ satisfies the simple
connectivity hypothesis of Theorem~\ref{thm:general-local-global}
although it fails the standard-norm hypothesis of
Corollary~\ref{cor:standard-local-global}.
\end{example}

\Needspace{8\baselineskip}
\section{Cone axioms and operator examples}
\label{app:cone-details}

The first four subsections complete the algebraic verifications
used in Section~\ref{sec:cone-linearization}. They use quadratic
representation identities and functional calculus, not any isometry
classification. The last subsection gives the operator and matrix examples.

\Needspace{8\baselineskip}
\subsection{Functional calculus and quadratic powers}
\label{app:JB-calculus}

\begin{proof}[Proof of Lemma~\ref{lem:power-U}]
Products of quadratic representations are compositions of bounded
linear operators on $A$. By the operator power-associativity rule
\cite[Part~II, Theorem~5.2.2(2), p.~201]{McCrimmon},
\[
 U_{f(a)}U_{g(a)}=U_{(fg)(a)}
\]
for real polynomials $f,g$.

The map $x\mapsto L_x$ is bounded and linear, and $x\mapsto x^2$
is continuous by continuity of the Jordan product. Hence
$U_x=2L_x^2-L_{x^2}$ depends continuously on $x$ in operator norm.

Fix $\alpha,\beta\in\mathbb R$ and choose a compact interval
$I\subset(0,\infty)$ containing $\sigma(a)$. By polynomial
approximation, choose real polynomials $f_n,g_n$ converging uniformly
on $I$ to $\lambda\mapsto\lambda^\alpha$ and
$\lambda\mapsto\lambda^\beta$, respectively. Their products converge
uniformly to $\lambda\mapsto\lambda^{\alpha+\beta}$.
Continuous functional calculus therefore gives
\[
 f_n(a)\longrightarrow a^\alpha,\qquad
 g_n(a)\longrightarrow a^\beta,\qquad
 (f_ng_n)(a)\longrightarrow a^{\alpha+\beta}
\]
in the norm of $A$. Passing to the limit in
$U_{f_n(a)}U_{g_n(a)}=U_{(f_ng_n)(a)}$, using operator-norm continuity and
continuity of composition in $\mathcal L(A)$, proves
$U_{a^\alpha}U_{a^\beta}=U_{a^{\alpha+\beta}}$ on all of $A$.
\end{proof}
\begin{proof}[Proof of Lemma~\ref{lem:JB-functional-continuity}]
Suppose $x_n\to x$ in $D_I$. Given $\varepsilon>0$, choose a real
polynomial $P$ with $\sup_{\lambda\in I}|f(\lambda)-P(\lambda)|
<\varepsilon/3$. Functional calculus gives
\[
\begin{aligned}
 \|f(x_n)-f(x)\|
 &\le\|f(x_n)-P(x_n)\|+\|P(x_n)-P(x)\|\\
 &\qquad+\|P(x)-f(x)\|.
\end{aligned}
\]
The first and third terms are less than $\varepsilon/3$, and the
middle term tends to zero because polynomial evaluation is norm
continuous. This proves the first assertion.

If $a_n\to a\in A_+^\circ$, choose $c>0$ with $a\ge ce$.
For all sufficiently large $n$,
\[
 a_n\ge a-\|a_n-a\|e\ge(c/2)e.
\]
Boundedness of the norms then places the spectra of $a$ and all
sufficiently large $a_n$ in a common compact interval
$[m,M]\subset(0,\infty)$. The first assertion applied to $\log$
gives $\log a_n\to\log a$. Likewise, every norm-convergent sequence
in $A$ has its spectra in a common interval $[-R,R]$, so the same
assertion applied to $\exp$ proves continuity of the exponential.
Finally, functional calculus gives $a^r=\exp(r\log a)$.
Continuity of scalar multiplication in $A$ and of $\log$ and $\exp$
therefore proves joint continuity of $(r,a)\mapsto a^r$.
\end{proof}
\Needspace{8\baselineskip}
\subsection{The gyrogroup on a positive cone}
\label{app:gyrogroup-verification}

\begin{proof}[Proof of Proposition~\ref{prop:gyrogroup-JB}]
We verify the left identity, left inverses, left gyroassociative law,
gyration automorphisms, and left loop property (G1)--(G5), followed by
gyrocommutativity (G6) \cite[Definitions~2.7 and~2.8]{Ungar2008}.
The left translations are bijective by \eqref{eq:left-translations}.
Since $U_e=I$ and $U_xe=x^2$, we have
$e\oplus a=a$ and $a\oplus e=a$. By \eqref{eq:U-on-powers},
\[
 a^{-1}\oplus a=U_{a^{-1/2}}a=e,
 \qquad a\oplus a^{-1}=U_{a^{1/2}}a^{-1}=e.
\]

Put $d=a\oplus b$ and define
\[
 \gyr[a,b]=\mathcal L_d^{-1}\mathcal L_a\mathcal L_b
 =U_{d^{-1/2}}U_{a^{1/2}}U_{b^{1/2}}.
\]
By Lemma~\ref{lem:Ka-b}, $K=\gyr[a,b]$ is a Jordan automorphism
mapping $\Om$ onto itself.
The identity $\mathcal L_d\gyr[a,b]=\mathcal L_a\mathcal L_b$ now gives,
for $c\in\Om$,
\[
 (a\oplus b)\oplus\gyr[a,b]c=a\oplus(b\oplus c),
\]
the left gyroassociative law. Also $K(U_xy)=U_{Kx}K(y)$ by the
definition of $U_x$. Together with \eqref{eq:Kpowers}, this gives
\[
\begin{aligned}
 K(x\oplus y)
 &=K(U_{x^{1/2}}y)
 =U_{K(x^{1/2})}K(y)\\
 &=U_{(Kx)^{1/2}}K(y)
 =Kx\oplus Ky.
\end{aligned}
\]
Thus each gyration is an automorphism of $(\Om,\oplus)$.

To prove the left loop property, we first establish the operator
identity
\begin{equation}\label{eq:LaLbLa}
 \mathcal L_a\mathcal L_b\mathcal L_a
 =\mathcal L_{a\oplus(b\oplus a)}.
\end{equation}
Put $x=U_{a^{1/2}}b^{1/2}$. Since $U_{a^{1/2}}$ preserves $\Om$,
the element $x$ is positive and invertible. The fundamental formula
gives
\[
 U_x
 =U_{a^{1/2}}U_{b^{1/2}}U_{a^{1/2}}
 =\mathcal L_a\mathcal L_b\mathcal L_a.
\]
Applying this to $e$ gives
\[
\begin{aligned}
 x^2
 &=U_xe
 =U_{a^{1/2}}U_{b^{1/2}}U_{a^{1/2}}e\\
 &=U_{a^{1/2}}U_{b^{1/2}}a
 =a\oplus(b\oplus a).
\end{aligned}
\]
Since $x$ is positive, functional calculus in the subalgebra generated
by $x,e$ gives $(x^2)^{1/2}=x$. Thus $x=(a\oplus(b\oplus a))^{1/2}$. Hence
$U_x=\mathcal L_{a\oplus(b\oplus a)}$, proving
\eqref{eq:LaLbLa}.

Again let $d=a\oplus b$. Since
$a^{-1}\oplus d=\mathcal L_a^{-1}(\mathcal L_a b)=b$,
\eqref{eq:LaLbLa} gives
\[
 \mathcal L_{d\oplus b}
 =\mathcal L_d\mathcal L_a^{-1}\mathcal L_d.
\]
Substitution in the defining formula for the gyration gives
\[
 \gyr[d,b]
 =(\mathcal L_d\mathcal L_a^{-1}\mathcal L_d)^{-1}
      \mathcal L_d\mathcal L_b
 =\mathcal L_d^{-1}\mathcal L_a\mathcal L_b
 =\gyr[a,b].
\]
This proves the left loop property.

Finally, the fundamental formula gives
$U_d=U_{a^{1/2}}U_bU_{a^{1/2}}$, hence $d^2=U_{a^{1/2}}U_ba$.
The same formula with inner argument $e$ gives $U_{x^2}=U_x^2$,
in particular $U_{b^{1/2}}^2=U_b$. By \eqref{eq:U-on-powers},
\[
\begin{aligned}
 \gyr[a,b](b\oplus a)
 &=U_{d^{-1/2}}U_{a^{1/2}}U_ba\\
 &=U_{d^{-1/2}}d^2=d=a\oplus b.
\end{aligned}
\]
This proves the gyrocommutative law and completes the verification.
\end{proof}
\Needspace{8\baselineskip}
\subsection{The affine GGV axioms}
\label{app:affine-axioms}

\begin{proof}[Verification of (GGV0)--(GGV7) and (GGVV) in Theorem~\ref{thm:JB-GGV}]
Fix $t>0$. Proposition~\ref{prop:gyrogroup-JB} and its power transport give the gyrocommutative
gyrogroup $(\Om,\oplus_t)$. Since the Jordan automorphism
$K=\gyr_t[a,b]$ is isometric, \eqref{eq:Klog} gives
$\|\phi(Kc)\|=\|K(\log c)\|=\|\phi(c)\|$, proving (GGV0).
The functional-calculus identities $a^1=a$ and
$(a^{r_2})^{r_1}=a^{r_1r_2}$ give (GGV1) and (GGV3).

For (GGV2), taking the $1/t$ power in \eqref{eq:GGV2-powers} gives
$(r_1\otimes a)\oplus_t(r_2\otimes a)=a^{r_1+r_2}=(r_1+r_2)\otimes a$.

Choose $0\ne h\in A$ by the standing hypothesis. For every $s\ge0$,
\[
 \left\|\log\exp\!\left(\frac{sh}{\|h\|}\right)\right\|=s.
\]
Thus the signed norm-value set is $\mathbb R$; with ordinary addition
and scalar multiplication this proves (GGVV). For (GGV4) and (GGV7), functional calculus gives
\[
 \phi(|r|\otimes a)=|r|\log a,
 \qquad \|\phi(r\otimes a)\|=|r|\|\log a\|.
\]
The second identity is (GGV7). If $a\ne e$ and $r\ne0$, its right-hand
side is positive, and division gives (GGV4). Identity~\eqref{eq:Kpowers}
is precisely (GGV5).

For (GGV6), put $x=a^{r_1}$ and $y=a^{r_2}$. By
\eqref{eq:GGV2-powers}, $d=U_{x^{t/2}}y^t=a^{t(r_1+r_2)}$, so
\eqref{eq:gyr-t} yields
\[
 \gyr_t[x,y]
 =U_{a^{-t(r_1+r_2)/2}}U_{a^{r_1t/2}}U_{a^{r_2t/2}}.
\]
Apply Lemma~\ref{lem:power-U} twice. The sum of the three exponents is
zero, so the product is $U_{a^0}=U_e=I$. This is an operator identity
on all of $A$, not merely on the subalgebra generated by $a$.
In particular, the gyration is the identity on $\Om$.

\end{proof}

\Needspace{8\baselineskip}
\subsection{The projective gyrogroup and scalar axioms}
\label{app:projective-axioms}

\begin{proof}[Algebraic verification for Theorem~\ref{thm:projective-JB-GGV}]
Fix $t>0$. Lemma~\ref{lem:projective-well-defined} makes the operations
well defined and $\Phi_A$ a bijection onto $([A],\|\cdot\|_v)$.
For $a,b\in\Omega$, put
\[
 d=U_{a^{t/2}}b^t,\qquad
 K^{(t)}_{a,b}=K_{a^t,b^t}=U_{d^{-1/2}}U_{a^{t/2}}U_{b^{t/2}}.
\]
By Theorem~\ref{thm:JB-GGV}, this affine gyration is a unital Jordan
automorphism. Define
\[
 \gyr_t^{\mathrm p}[\bar a,\bar b]\bar c
 =\overline{K^{(t)}_{a,b}c}.
\]
If $a,b$ are replaced by $\alpha a,\beta b$, the homogeneity in
Lemma~\ref{lem:projective-well-defined} gives $d'=(\alpha\beta)^td$, so
\[
 K^{(t)}_{\alpha a,\beta b}
 =(\alpha\beta)^{-t}\alpha^t\beta^tK^{(t)}_{a,b}
 =K^{(t)}_{a,b}.
\]
This is an operator identity on $A$. Linearity also gives
$K^{(t)}_{a,b}(\gamma c)=\gamma K^{(t)}_{a,b}c$ for $\gamma>0$,
proving independence of all three representatives.
By Proposition~\ref{prop:gyrogroup-JB} and \eqref{eq:gyr-t},
both $K^{(t)}_{a,b}$ and its inverse preserve the affine operation
and positive scalar multiples, so the induced map is an automorphism
of the projective operation.

The surjective map $\pi_A$ intertwines the operations and gyrations:
\[
\begin{aligned}
 \pi_A(a\oplus_t b)&=\pi_A(a)\poplus_t\pi_A(b),\\
 \pi_A(K^{(t)}_{a,b}c)
 &=\gyr_t^{\mathrm p}[\bar a,\bar b]\pi_A(c).
\end{aligned}
\]
The identity $\bar e$, the inverse $\overline{a^{-1}}$, and the left
inverse law descend through $\pi_A$. In particular projective left
translations cancel. Projecting left gyroassociativity and cancelling
the left translation by $\bar a\poplus_t\bar b$ identifies
$\gyr_t^{\mathrm p}$ with the gyrator of $\poplus_t$.
The affine left loop identity
$K^{(t)}_{q_t(a,b),b}=K^{(t)}_{a,b}$
(Proposition~\ref{prop:gyrogroup-JB} and \eqref{eq:gyr-t})
therefore descends independently of representatives. Projecting
gyrocommutativity completes the gyrogroup axioms.

We verify the GGV axioms. With $K=K^{(t)}_{a,b}$, identity
\eqref{eq:Klog} and preservation of spectra give
\[
 \|\Phi_A(\gyr_t^{\mathrm p}[\bar a,\bar b]\bar c)\|_v
 =\operatorname{diam}\sigma(K(\log c))
 =\|\Phi_A(\bar c)\|_v,
\]
which is (GGV0).
Since $\pi_A(r\otimes a)=r\potimes\pi_A(a)$, the scalar identities
(GGV1)--(GGV3), (GGV5), and (GGV6) descend in the same way from
Theorem~\ref{thm:JB-GGV}. Functional calculus also gives
\begin{equation}\label{eq:projective-linear-rays}
 \Phi_A(r\potimes\bar a)=[\log(a^r)]
 =[r\log a]=r\Phi_A(\bar a)\qquad(r\in\mathbb R).
\end{equation}
This yields (GGV4) and (GGV7) by norm homogeneity.
The denominator in (GGV4) is nonzero when $\bar a\ne\bar e$ and $r\ne0$,
because $\Phi_A$ is injective and $\Phi_A(\bar e)=[0]$.
Since $A$ has dimension at least two, $[A]$ is a nonzero normed space.
Surjectivity of $\Phi_A$ therefore gives
\[
 \{\pm\|\Phi_A(\bar a)\|_v:\bar a\in\mathbb P\Omega\}=\mathbb R.
\]
With ordinary real operations, this is (GGVV) and gives real value-line
structure.

\end{proof}
\Needspace{8\baselineskip}
\subsection{Operator and matrix examples}

\Needspace{4\baselineskip}
\begin{example}[Exponential subsets in operator algebras]
\label{ex:operator-exponential-GGV}
This extends the assertion of \cite[Example~20]{HatoriExamples2017}
to the standard-norm property, with the nontriviality assumption
$E\ne\{0\}$ made explicit. Let $\mathcal A\subset B(\mathcal H)$ be
a unital $C^*$-subalgebra with the same identity $I$, and let $N$ be
a unitarily invariant uniform norm on all of $B(\mathcal H)$.
Here uniformity means
\[
 N(xy)\le\min\{\|x\|_{\mathrm{op}}N(y),\,
                    N(x)\|y\|_{\mathrm{op}}\}
 \qquad(x,y\in B(\mathcal H)),
\]
and unitary invariance means $N(uxv)=N(x)$ for all unitary $u,v$.
Let $E\ne\{0\}$ be a real linear subspace of $\mathcal A_{\sa}$ such
that
\[
 aba\in G:=\exp E\qquad(a,b\in G).
\]
No closedness assumption on $E$ is required. For every $t>0$, put
\[
 a\oplus_t b=(a^{t/2}b^ta^{t/2})^{1/t},\qquad
 r\otimes a=a^r,\qquad \phi_N(a)=\log a\in E.
\]
With $E$ equipped with the restriction of $N$, these operations make
$G$ a standard-norm GGV. Its gyrometric is
\[
 \rho_{N,t}(a,b)
 =\frac1t N\!\left(\log(a^{-t/2}b^ta^{-t/2})\right),
 \qquad a,b\in G,
\]
and its topology is the relative operator-norm topology on $G$.
\end{example}

\begin{proof}
With the operator norm and Jordan product $x\circ y=(xy+yx)/2$,
$\mathcal A_{\sa}$ is a JB-algebra and $U_xy=xyx$.
For $h\in E$, the element $\exp h$ is positive in $\mathcal A$ and
has inverse $\exp(-h)$, so $G\subset\Omega_{\mathcal A_{\sa}}$.
Since $E$ is real linear, $I\in G$ and $a^r=\exp(r\log a)\in G$
for $a\in G$ and $r\in\mathbb R$. The triple-product assumption
therefore makes $\oplus_t$ well defined on $G$. Each left translation
$\mathcal L_a^{(t)}(b)=a\oplus_t b$ restricts to a bijection of $G$, with
inverse $\mathcal L_{a^{-1}}^{(t)}$. Thus
\[
 \gyr_t[a,b]=(\mathcal L_{a\oplus_t b}^{(t)})^{-1}\mathcal L_a^{(t)}\mathcal L_b^{(t)}
\]
also restricts to a bijection of $G$. The gyrocommutative gyrogroup
axioms and (GGV1)--(GGV3), (GGV5), and (GGV6) are inherited from
Theorem~\ref{thm:JB-GGV} applied to the JB-algebra
$\mathcal A_{\sa}$. These assertions do not depend on the coordinate
norm, and $\phi_N=\log:G\to E$ is injective.

For the norm-dependent conditions, fix $a,b\in G$ and set
\[
 d=a^{t/2}b^ta^{t/2},\qquad
 v=d^{-1/2}a^{t/2}b^{t/2}.
\]
The operator $v$ is invertible and satisfies $vv^*=I$, so it is
unitary, also when $\mathcal H$ is infinite dimensional.
Formula~\eqref{eq:gyr-t} gives $\gyr_t[a,b]c=vcv^*$ for $c\in G$.
Functional calculus and unitary invariance yield
\[
 N(\phi_N(\gyr_t[a,b]c))
 =N(v(\log c)v^*)=N(\phi_N(c)),
\]
which proves (GGV0). The identity $\log(a^r)=r\log a$ gives
(GGV4) and (GGV7) by norm homogeneity; the denominator in (GGV4)
is nonzero for $a\ne I$ and $r\ne0$, since $\log$ is injective and
$\log I=0$. To verify (GGVV), choose
$0\ne h\in E$. For every $s\ge0$,
\[
 a_s=\exp\!\left(\frac{sh}{N(h)}\right)\in G,
 \qquad N(\log a_s)=s.
\]
Hence the signed norm-value set is $\mathbb R$, equipped with its
ordinary addition and scalar multiplication.

To check the completeness hypothesis in Hatori's inequality, fix a
rank-one orthogonal projection $P$. For unit vectors $\xi,\eta$, let
$P_\xi,P_\eta$ denote their range projections. The operator $P_\eta xP_\xi$
is a scalar multiple of a rank-one partial isometry of the form $uPv$
with $u,v$ unitary. Hence uniformity and unitary invariance give
\[
 |\langle x\xi,\eta\rangle|\,N(P)=N(P_\eta xP_\xi)\le N(x).
\]
Taking the supremum over $\xi,\eta$ and applying uniformity to $xI$ gives
\[
 N(P)\|x\|_{\mathrm{op}}\le N(x)\le N(I)\|x\|_{\mathrm{op}}.
\]
Thus $N$ is complete on $B(\mathcal H)$. Hatori's inequality
\cite[Theorem~8]{HatoriExamples2017} assumes a complete unitarily
invariant uniform norm on all of $B(\mathcal H)$, as here. Applied to
$a^t$ and $b^t$, it gives
\[
 \begin{aligned}
 N(\log(a\oplus_t b))
 &=\frac1t N(\log(a^{t/2}b^ta^{t/2}))\\
 &\le\frac1t\bigl(N(\log a^t)+N(\log b^t)\bigr)\\
 &=N(\log a)+N(\log b).
 \end{aligned}
\]
This proves (GGV8), and hence the GGV assertion with real value-line
structure.

In the JB-algebra $\mathcal A_{\sa}$, the inverse identity used to
derive \eqref{eq:rho-affine-log} gives
$\log(a\ominus_t b)=-t^{-1}\log(a^{-t/2}b^ta^{-t/2})$.
Taking $N$ gives the stated gyrometric.

For the topology, let $\rho_{\mathrm{op},t}$ be the gyrometric
of Theorem~\ref{thm:JB-GGV} on the positive invertible cone of
$\mathcal A_{\sa}$, restricted to $G$. The displayed formula for $\rho_{N,t}$ and \eqref{eq:rho-affine-log}
now imply
\[
 N(P)\rho_{\mathrm{op},t}(a,b)\le\rho_{N,t}(a,b)
 \le N(I)\rho_{\mathrm{op},t}(a,b)\qquad(a,b\in G).
\]
The two gyrometrics induce the same topology, which is the relative
operator-norm topology on $G$ by Theorem~\ref{thm:JB-GGV}.
For each fixed $r\in\mathbb R$, the map $a\mapsto a^r$ preserves
$G$ and is operator-norm continuous by
Lemma~\ref{lem:JB-functional-continuity}. Hence it is continuous
for $\rho_{N,t}$, proving the standard-norm assertion.
\end{proof}

\Needspace{4\baselineskip}
\begin{remark}\label{rem:operator-norm-scope}
In infinite dimensions the finite-$p$ Schatten norms are defined on
proper ideals. A uniform unitarily invariant norm on all of
$B(\mathcal H)$ that is not a multiple of the operator norm is
$N(X)=\|X\|_{\mathrm{op}}+\|\pi(X)\|$, where $\pi$ is the quotient by
the compact operators. The quotient norm is contractive and
submultiplicative, giving both uniformity inequalities. For a
rank-one orthogonal projection $P$ one has $N(P)=1$ and $N(I)=2$.
\end{remark}

\Needspace{4\baselineskip}
\begin{example}[Unitarily invariant norms on positive definite matrices]
\label{ex:affine-riemannian-GGV}
Let $\mathcal P_n$ be the cone of positive definite Hermitian matrices
and let $N$ be any unitarily invariant norm on $M_n(\mathbb C)$.
For every $t>0$, Example~\ref{ex:operator-exponential-GGV} with
$E=M_n(\mathbb C)_{\sa}$ gives a standard-norm GGV on $\mathcal P_n$
with gyrometric
\[
 \rho_{N,t}(A,B)=\frac1t N\!\left(\log(A^{-t/2}B^tA^{-t/2})\right).
\]
Indeed, $N$ is uniform because every contraction in $M_n(\mathbb C)$
is the average of two unitaries, by singular value decomposition.
This gives the standard-norm property for the GGVs in
\cite[Example~21]{HatoriExamples2017}.

At $t=1$, let $\lambda_1,\ldots,\lambda_n$ be the eigenvalues of
$A^{-1/2}BA^{-1/2}$. The Schatten $2$-norm and the operator norm give,
respectively,
\[
 \left(\sum_{j=1}^n|\log\lambda_j|^2\right)^{1/2}
 \qquad\text{and}\qquad
 \max_{1\le j\le n}|\log\lambda_j|.
\]
The first is the affine-invariant Riemannian distance for
$g_A(X,Y)=\operatorname{Tr}(A^{-1}XA^{-1}Y)$
\cite[Theorem~6.1.6]{BhatiaPDM}, \cite[Section~2]{BhatiaJain2014}.
The second is Thompson's distance by Corollary~\ref{cor:rho-Thompson}.
More generally, the Schatten $p$-distance at $t=1$ is the $\ell^p$-norm
of $(|\log\lambda_j|)_j$, and the Ky Fan $k$-distance is the sum of its
$k$ largest entries. As $t\downarrow0$, expansion of the exponential
and logarithm at the identity gives
$\rho_{N,t}(A,B)\to N(\log B-\log A)$ and
$A\oplus_t B\to\exp(\log A+\log B)$. These limits give the normed
space $(M_n(\mathbb C)_{\sa},N)$ in logarithmic coordinates.
\end{example}

\begin{example}[Exponential matrix subspaces]
\label{ex:four-exponential-matrix-GGVs}
The same construction applies to the determinant-one part of
$\mathcal P_n$, the real positive definite matrices, their
determinant-one part, and $\exp(i\mathfrak{so}(n))$, where
$\mathfrak{so}(n)$ consists of the real skew-symmetric matrices
\cite[Examples~22--25]{HatoriExamples2017}. Take $n\ge2$ except
for the full real cone, where $n\ge1$ suffices. In the first three
cases the logarithms form the corresponding real linear spaces of
Hermitian or real symmetric matrices, with trace zero when the
determinant is one. The defining conditions are preserved by $ABA$.
In the last case the positive matrices are characterized by
$\overline A=A^{-1}$, since $\overline{\log A}=\log\overline A$.
This condition also passes to $ABA$. Hence
Example~\ref{ex:operator-exponential-GGV} applies to all four spaces
for every $t>0$ and every unitarily invariant matrix norm.
\end{example}

\Needspace{8\baselineskip}

\section{Cone lifts and matrix parameters}
\label{app:cone-lifts}

Hilbert's metric on rays pulls back to a pseudometric on the
positive cone. We retain the full lifting and surjectivity statements,
since the pseudometric does not determine the behavior within a ray.
The final subsection proves the explicit matrix parameter assertions
of Corollary~\ref{cor:density-local}.
\Needspace{8\baselineskip}
\subsection{Cone-level consequences}
\label{subsec:Hilbert-cone}

For a unital JB-algebra $C$, define the pullback Hilbert pseudometric on its
positive cone by
\[
 d^{\mathrm{cone}}_{H,C}(a,a')=\dH(\bar a,\bar a'),
 \qquad a,a'\in C_+^\circ.
\]
It vanishes exactly on pairs belonging to the same positive ray,
because $\dH$ is a metric on the ray space; see
\eqref{eq:order-metrics} and Proposition~\ref{prop:prho-Hilbert}.
The following cone-level statement is a consequence of
Corollary~\ref{cor:Hilbert-global}.

\Needspace{4\baselineskip}
\begin{corollary}[Hilbert isometries on the positive cone]
\label{cor:Hilbert-cone-lift}
Let $A,B$ be unital JB-algebras of dimension at least two. A surjective map
$F:A_+^\circ\to B_+^\circ$ satisfying
\[
 d^{\mathrm{cone}}_{H,B}(F(a),F(a'))=d^{\mathrm{cone}}_{H,A}(a,a'),
 \qquad a,a'\in A_+^\circ,
\]
induces a well-defined bijective Hilbert isometry
\[
 \bar F:\mathbb P A_+^\circ\longrightarrow\mathbb P B_+^\circ,
 \qquad \bar F(\bar a)=\overline{F(a)}.
\]
Consequently there are a Jordan isomorphism $J:A\to B$, a sign
$\varepsilon\in\{-1,1\}$, an element $b\in B_+^\circ$, and a positive
scalar-valued function $\tau:A_+^\circ\to(0,\infty)$ such that
\[
 F(a)=\tau(a)U_bJ(a^\varepsilon)
 \qquad(a\in A_+^\circ).
\]
The ray $\bar b$ is unique and $J$ is unique for each admissible sign,
as in Corollary~\ref{cor:Hilbert-global}. Once the sign and a representative
$b$ are fixed, $F$ determines $\tau$ pointwise.
Conversely, every map of this form preserves the pullback Hilbert
pseudometric. Such a map is onto $B_+^\circ$ exactly when, for one
representative $a$ of each source ray, the map
\[
 (0,\infty)\longrightarrow(0,\infty),\qquad
 r\longmapsto r^\varepsilon\tau(ra)
\]
is onto. This condition is independent of the chosen representatives.
\end{corollary}

\begin{proof}
Preservation of the pullback Hilbert pseudometric gives
\[
 \bar a=\bar a'
 \quad\Longleftrightarrow\quad
 \overline{F(a)}=\overline{F(a')}.
\]
Thus $\bar F$ is well defined and injective; surjectivity of $F$
makes $\bar F$ surjective. By the pullback identity, it is a Hilbert
isometry. Corollary~\ref{cor:Hilbert-global} gives
$\overline{F(a)}=\overline{U_bJ(a^\varepsilon)}$, hence the stated
formula, with a unique positive $\tau(a)$ for fixed $F,J,\varepsilon,b$.
Conversely, Corollary~\ref{cor:Hilbert-global} makes
$\bar a\mapsto\overline{U_bJ(a^\varepsilon)}$ a bijective Hilbert
isometry. Multiplication by any positive $\tau(a)$ leaves its rays
unchanged, so the resulting $F$ preserves the pullback pseudometric.

Put $G(a)=U_bJ(a^\varepsilon)$. Then
$F(ra)=r^\varepsilon\tau(ra)G(a)$, and the induced projective map is
bijective. Each target ray therefore has exactly one source ray, so
the stated condition is necessary and sufficient to cover every point
of $B_+^\circ$. For $r_0>0$, setting $q_a(r)=r^\varepsilon\tau(ra)$ gives
\[
 q_{r_0a}(r)=r_0^{-\varepsilon}q_a(rr_0),
\]
which proves independence of the source representative.
\end{proof}

\Needspace{4\baselineskip}
\begin{remark}[Radial factors]\label{rem:Hilbert-radial-factors}
The function $\tau$ depends on the chosen parameters: replacing $b$
by $\beta b$, $\beta>0$, replaces $\tau$ by $\beta^{-2}\tau$ for the
same $F$. Pseudometric preservation alone imposes no norm continuity
on $\tau$, nor does it ensure surjectivity. For example, for any state
$\omega$ on $A$ (see the proof of Theorem~\ref{thm:projective-JB-GGV}),
the map $a\mapsto a/\omega(a)$ induces the identity on rays and has image
$\Sigma_\omega$. With $J=\mathrm{id}$, $\varepsilon=1$, and $b=e_A$,
the radial function of Corollary~\ref{cor:Hilbert-cone-lift} is
$r\tau(ra)=1/\omega(a)$, constant on $(0,\infty)$.
\end{remark}

\Needspace{4\baselineskip}
\begin{remark}[Positive operators]\label{rem:operator-Hilbert-classification}
Let $A=B=B(H)_{\sa}$, where $H$ is a complex Hilbert space with
$\dim H\ge2$, possibly infinite dimensional. By
Remark~\ref{rem:affine-cstar-classification}, $J$ extends complex-linearly
to a Jordan $*$-automorphism of $B(H)$. Since $B(H)$ is a factor,
the extension is a $*$-automorphism or a $*$-antiautomorphism by
Kadison's factor theorem \cite[Corollary~11, p.~335]{Kadison1951}. In the latter case,
composition with transpose in a fixed orthonormal basis gives a
$*$-automorphism. Every $*$-automorphism $\alpha$ of $B(H)$ is inner.
Indeed, for a unit vector $\xi$, let $P_\xi$ project onto $\mathbb C\xi$.
The minimal projection $\alpha(P_\xi)$ has rank one. Choose a unit vector
$\eta$ in its range. Writing
$E_{u,\xi}v=\langle v,\xi\rangle u$ with the inner product linear in its
first variable, $Vu=\alpha(E_{u,\xi})\eta$ defines a linear isometry
since $E_{u,\xi}^*E_{u,\xi}=\|u\|^2 P_\xi$. Its range is all of $H$:
it contains $\eta=V\xi$ and is invariant under every $\alpha(x)$
because $Vxu=\alpha(x)Vu$. Thus $\alpha(x)=VxV^*$.
On self-adjoint elements this gives $Jx=VxV^*$ or $Vx^{\mathrm T}V^*$,
equivalently $Jx=SxS^*$ with $S$ unitary or antiunitary. By
Lemma~\ref{lem:Ka-b}, $J(a^\varepsilon)=(Ja)^\varepsilon$.
Since $U_bx=bxb$, putting $T=bS$ gives
$F(a)=\tau(a)Ta^\varepsilon T^*$, where $T$ is bounded, invertible,
and linear or conjugate-linear. Conversely, polar decomposition gives
$T=(TT^*)^{1/2}S=bS$ with $b>0$ invertible and $S$ unitary or
antiunitary, yielding the Jordan form above; see also \cite[Section~1]{LRW}.
For $\dim H\ge3$, Moln\'ar gives this formula for bijections preserving
the Hilbert projective metric \cite[Theorem~2]{MolnarThompsonHilbert}. Corollary~\ref{cor:Hilbert-cone-lift} gives it for surjections
as well, without requiring norm continuity of $F$ or $\tau$.

In dimension two, let
\[
 \sigma_y=\begin{pmatrix}0&-i\\i&0\end{pmatrix},\qquad
 R=-i\sigma_y=\begin{pmatrix}0&-1\\1&0\end{pmatrix}.
\]
Then $\operatorname{adj}(x)=Rx^{\mathrm T}R^*$ and
$a^{-1}=(\det a)^{-1}Ra^{\mathrm T}R^*$ for $a>0$.
Thus inversion can be absorbed into the Jordan automorphism and the
positive factor $\tau$, in agreement with
\cite[Theorem~1.2 and Remark~1.3]{MolnarNagy}; that theorem already
assumes only surjectivity.
\end{remark}

Suppose that $\dim C\ge2$ and let $\omega$ be a state on $C$. Theorem~\ref{thm:projective-JB-GGV} gives the
homeomorphism
\[
 C_+^\circ\longrightarrow\mathbb P C_+^\circ\times(0,\infty),
 \qquad a\longmapsto(\bar a,\omega(a)),\qquad
 (\bar x,r)\longmapsto r\,s_\omega(\bar x)\ \text{(inverse)}.
\]
Indeed, $\pi_C(a)=\Phi_C^{-1}([\log a])$ is continuous, as is the
section $s_\omega$ by the same theorem. In these coordinates $\pi_C$ is the first projection, hence continuous
and open. A union of full rays $\pi_C^{-1}(U)$ corresponds to
$U\times(0,\infty)$, so it is norm open exactly when $U$ is Hilbert open,
and connected exactly when $U$ is connected.

\Needspace{4\baselineskip}
\begin{corollary}[Cone-level local-to-global consequence]
\label{cor:Hilbert-cone-local-global}
Let $A,B$ be unital JB-algebras of dimension at least two, and let
$U\subset\mathbb P A_+^\circ$ and
$V\subset\mathbb P B_+^\circ$ be nonempty open subsets, with $U$ connected.
Put $\widehat U=\pi_A^{-1}(U)$ and $\widehat V=\pi_B^{-1}(V)$,
the unions of the rays in $U$ and $V$, respectively.
Suppose $F:\widehat U\to\widehat V$ is a surjection satisfying
\[
 d^{\mathrm{cone}}_{H,B}(F(a),F(a'))
 =d^{\mathrm{cone}}_{H,A}(a,a')
 \qquad(a,a'\in\widehat U).
\]
Then $F$ induces a bijective projective Hilbert isometry
$\bar F:U\to V$. This map extends uniquely to a global bijective
Hilbert isometry
\[
 \widetilde{\bar F}:\mathbb P A_+^\circ\longrightarrow\mathbb P B_+^\circ,
\]
and there exist a Jordan isomorphism $J:A\to B$, a sign
$\varepsilon\in\{-1,1\}$, and $b\in B_+^\circ$ such that
\[
 \widetilde{\bar F}(\bar a)=\overline{U_bJ(a^\varepsilon)}
 \qquad(a\in A_+^\circ).
\]
On $\widehat U$, the original map consequently has the form
\[
 F(a)=\tau(a)U_bJ(a^\varepsilon),\qquad \tau(a)>0.
\]
For fixed $F,J,\varepsilon,b$, this determines $\tau$ uniquely on
$\widehat U$. Changing positive radial factors preserves the induced
projective map and the pseudometric, but surjectivity requires the
condition below.
Conversely, choose $J,\varepsilon,b$ so that the global projective
isometry $h(\bar a)=\overline{U_bJ(a^\varepsilon)}$ carries $U$ onto
$V$, and choose a positive function $\tau$ on $\widehat U$.
Then $F(a)=\tau(a)U_bJ(a^\varepsilon)$ maps $\widehat U$ into
$\widehat V$, preserves the pullback pseudometric, and is onto
$\widehat V$ exactly when, for one representative $a$ of each ray in
$U$, the map $r\mapsto r^\varepsilon\tau(ra)$ is onto $(0,\infty)$.
\end{corollary}

\begin{proof}
The ray-descent argument in Corollary~\ref{cor:Hilbert-cone-lift}
makes $\bar F(\bar a)=\overline{F(a)}$ a well-defined injective Hilbert
isometry. Since $F(\widehat U)=\widehat V$ and
$\pi_B(\widehat V)=V$, it is onto $V$.
Corollary~\ref{cor:Hilbert-local} gives its unique global extension
and formula. Equality of the rays determines the positive multiple
$\tau(a)$ for fixed $J,\varepsilon,b$.
Conversely, $\overline{F(a)}=h(\bar a)\in V$ implies
$F(a)\in\pi_B^{-1}(V)=\widehat V$, and the pullback identity gives
pseudometric preservation. The raywise calculation in
Corollary~\ref{cor:Hilbert-cone-lift} applies because the two saturated
sets contain full rays and $h|_U:U\to V$ is bijective.
\end{proof}

\Needspace{8\baselineskip}
\subsection{Parameters for density-matrix isometries}
\label{app:density-parameters}

\begin{proof}[Parameter verification for Corollary~\ref{cor:density-local}]
The trace-one section $s_{\operatorname{Tr}}(\bar a)=a/\operatorname{Tr}a$
is an isometric homeomorphism for
the transported gyrometric. Thus $s_{\operatorname{Tr}}^{-1}Ts_{\operatorname{Tr}}$
is a bijective isometry between the corresponding connected open
subsets of the projective cone.
Apply Corollary~\ref{cor:projective-local-classification} and transfer
its unique extension back through the section.
Every Jordan automorphism of $M_n(\mathbb C)_{\sa}$ has the form
$x\mapsto W\chi(x)W^*$; see
Remark~\ref{rem:operator-Hilbert-classification}.
Inserting it into \eqref{eq:projective-local-form} and taking trace-one
representatives gives the formula. Replacing $c$ by $\lambda c$,
$\lambda>0$, multiplies $X(a)$ by $\lambda$, so we may normalize
$\operatorname{Tr}c=1$ without changing $\widetilde T$.
Then $X(I/n)=n^{-\varepsilon}c$, whence $\widetilde T(I/n)=c$.
The converse follows from Theorem~\ref{thm:projective-global-classification}
and the same section identification.

For $n=2$, use the matrix $R=-i\sigma_y$ in
Remark~\ref{rem:operator-Hilbert-classification}. Functional calculus
applied to $a^{-1}=(\det a)^{-1}Ra^{\mathrm T}R^*$ gives
\[
 a^{-t}=(\det a)^{-t}R(a^{\mathrm T})^tR^*\qquad(t>0).
\]
If $\varepsilon=-1$, put $\chi'(a)=\chi(a)^{\mathrm T}$, which interchanges
the two choices of $\chi$. Since $\det\chi(a)=\det a$, applying this
identity to $\chi(a)$ yields
\[
 X(a)=(\det a)^{-1}
 \bigl(c^{t/2}(WR)\chi'(a)^t(WR)^*c^{t/2}\bigr)^{1/t}.
\]
The positive factor $(\det a)^{-1}$ cancels under trace normalization.
Since $WR$ is unitary, replacing $(W,\chi,-1)$ by $(WR,\chi',1)$ gives
the same map for every $t>0$.
For $n\ge3$, the algebra has rank $n$, so it is neither $\mathbb R^2$
nor a spin factor; Theorem~\ref{thm:projective-global-classification}
therefore makes the sign unique.
For a fixed sign, Corollary~\ref{cor:projective-local-classification}
makes $J$ unique. The two choices of $\chi$ cannot represent the same
$J$: their complex linear extensions are respectively an automorphism
and an antiautomorphism of the noncommutative algebra $M_n(\mathbb C)$.
With $\chi$ fixed, equality of two representations makes $W'^*W$
commute with every Hermitian matrix, hence with every matrix, so
$W'^*W$ is a scalar unitary.
\end{proof}
\clearpage
\section*{Statements and Declarations}

\noindent\textbf{Competing interests.}
The author has no relevant financial or non-financial interests to disclose.

\medskip
\noindent\textbf{Use of generative AI.}

ChatGPT (OpenAI) was used for language editing, organization of the
presentation, literature discussion, and discussion of mathematical
arguments. Claude (Anthropic) was also used for discussion of mathematical
arguments. The author is responsible for verifying the mathematical
statements and proofs and takes full responsibility for the manuscript.

\medskip
\noindent\textbf{Data availability.}
No datasets were generated or analysed during this study.


\begin{thebibliography}{99}
\setlength{\itemsep}{3pt plus 1pt minus 1pt}

\bibitem{Abe2016}
T.~Abe,
\emph{Normed gyrolinear spaces: a generalization of normed spaces based on
gyrocommutative gyrogroups},
Math. Interdiscip. Res. \textbf{1} (2016), 143--172.
\url{https://doi.org/10.22052/mir.2016.13912}.

\bibitem{AH2015}
T.~Abe and O.~Hatori,
\emph{Generalized gyrovector spaces and a Mazur--Ulam theorem},
Publ. Math. Debrecen \textbf{87} (2015), 393--413.
\url{https://doi.org/10.5486/PMD.2015.7234}.

\bibitem{AHrevisited}
T.~Abe and O.~Hatori,
\emph{Generalized gyrovector spaces revisited},
Nihonkai Math. J., to appear.
\url{https://arxiv.org/abs/2403.15015}.

\bibitem{Atiponrat2017}
W.~Atiponrat,
\emph{Topological gyrogroups: generalization of topological groups},
Topol. Appl. \textbf{224} (2017), 73--82.
\url{https://doi.org/10.1016/j.topol.2017.04.004}.

\bibitem{Ballmann2006}
W.~Ballmann,
\emph{Lectures on K\"ahler Manifolds},
ESI Lectures in Mathematics and Physics,
European Mathematical Society, Z\"urich, 2006.
\url{https://doi.org/10.4171/025}.

\bibitem{BaoLin2019}
M.~Bao and F.~Lin,
\emph{Feathered gyrogroups and gyrogroups with countable pseudocharacter},
Filomat \textbf{33} (2019), no.~16, 5113--5124.
\url{https://doi.org/10.2298/FIL1916113B}.

\bibitem{BhatiaPDM}
R.~Bhatia,
\emph{Positive Definite Matrices},
Princeton Series in Applied Mathematics, Princeton University Press,
Princeton, NJ, 2007.

\bibitem{BhatiaJain2014}
R.~Bhatia and T.~Jain,
\emph{Approximation problems in the Riemannian metric on positive definite matrices},
Ann. Funct. Anal. \textbf{5} (2014), no.~2, 118--126.
\url{https://www.maths.tcd.ie/EMIS/journals/AFA/v5n2/AFA-tex_v5_n2_a9.pdf}.

\bibitem{Bosche2012}
A.~Bosch\'e,
\emph{Symmetric cones, the Hilbert and Thompson metrics},
arXiv:1207.3214v1 (2012).
\url{https://arxiv.org/abs/1207.3214v1}.

\bibitem{FarautKoranyi}
J.~Faraut and A.~Kor\'anyi,
\emph{Analysis on Symmetric Cones},
Oxford Mathematical Monographs, Clarendon Press, Oxford, 1994.

\bibitem{FriedmanUngar1994}
Y.~Friedman and A.~A. Ungar,
\emph{Gyrosemidirect product structure of bounded symmetric domains},
Results Math. \textbf{26} (1994), 28--38.
\url{https://doi.org/10.1007/BF03322286}.

\bibitem{HOS}
H.~Hanche-Olsen and E.~St\o rmer,
\emph{Jordan Operator Algebras},
Monographs and Studies in Mathematics, vol.~21,
Pitman, Boston, 1984.

\bibitem{HatoriExamples2017}
O.~Hatori,
\emph{Examples and applications of generalized gyrovector spaces},
Results Math. \textbf{71} (2017), nos.~1--2, 295--317.
\url{https://doi.org/10.1007/s00025-016-0576-1}.

\bibitem{HatoriExtension}
O.~Hatori,
\emph{Extension of isometries in generalized gyrovector spaces of the positive cones},
Contemp. Math. \textbf{687} (2017), 145--156.
\url{https://doi.org/10.1090/conm/687/13728}.

\bibitem{HatoriMolnar}
O.~Hatori and L.~Moln\'ar,
\emph{Isometries of the unitary groups and Thompson isometries of the
spaces of invertible positive elements in $C^*$-algebras},
J. Math. Anal. Appl. \textbf{409} (2014), no.~1, 158--167.
\url{https://doi.org/10.1016/j.jmaa.2013.06.065}.

\bibitem{IKR1995}
J.~M. Isidro, W. Kaup and A. Rodr\'iguez,
\emph{On real forms of JB$^*$-triples},
Manuscripta Math. \textbf{86} (1995), 311--335.
\url{https://doi.org/10.1007/BF02567997}.

\bibitem{IRP}
J.~M. Isidro and A.~Rodr\'iguez-Palacios,
\emph{Isometries of JB-algebras},
Manuscripta Math. \textbf{86} (1995), no.~3, 337--348.
\url{https://doi.org/10.1007/BF02567998}.

\bibitem{Kadison1951}
R.~V. Kadison,
\emph{Isometries of operator algebras},
Ann. of Math. (2) \textbf{54} (1951), no.~2, 325--338.

\bibitem{Kaup1983}
W.~Kaup, \emph{A Riemann mapping theorem for bounded symmetric domains in
complex Banach spaces}, Math. Z. \textbf{183} (1983), 503--529.
\url{https://doi.org/10.1007/BF01173928}.

\bibitem{KimDistributivity2015}
S.~Kim,
\emph{Distributivity on the gyrovector spaces},
Kyungpook Math. J. \textbf{55} (2015), no.~1, 13--20.
\url{https://doi.org/10.5666/KMJ.2015.55.1.13}.

\bibitem{KimLawson2011}
S. Kim and J. Lawson, \emph{Smooth Bruck loops, symmetric spaces,
and nonassociative vector spaces}, Demonstratio Math. \textbf{44}
(2011), 755--779.

\bibitem{KimLawson2015}
S.~Kim and J.~Lawson,
\emph{Local midpoints on smooth manifolds},
Differential Geom. Appl. \textbf{39} (2015), 129--146.
\url{https://doi.org/10.1016/j.difgeo.2015.01.009}.

\bibitem{Klotz2011}
M.~Klotz,
\emph{Banach symmetric spaces},
arXiv:0911.2089v3 (2011).
\url{https://arxiv.org/abs/0911.2089v3}.

\bibitem{LemmensWalsh2026}
B. Lemmens and C. Walsh,
\emph{Carath\'eodory distance-preserving maps between bounded symmetric domains},
Math. Ann. \textbf{394} (2026), article 85.
\url{https://doi.org/10.1007/s00208-026-03408-6}.

\bibitem{LRW}
B.~Lemmens, M.~Roelands and M.~Wortel,
\emph{Hilbert and Thompson isometries on cones in JB-algebras},
Math. Z. \textbf{292} (2019), 1511--1547.
\url{https://doi.org/10.1007/s00209-018-2144-8}.

\bibitem{MackeyMellon2004}
M.~Mackey and P.~Mellon,
\emph{A Schwarz lemma and composition operators},
Integral Equations Operator Theory \textbf{48} (2004), 511--524.
\url{https://doi.org/10.1007/s00020-003-1240-1}.

\bibitem{Mankiewicz}
P.~Mankiewicz,
\emph{On extension of isometries in normed linear spaces},
Bull. Acad. Polon. Sci. S\'er. Sci. Math. Astronom. Phys.
\textbf{20} (1972), 367--371.

\bibitem{MatveevTroyanov}
V.~S. Matveev and M.~Troyanov,
\emph{The Myers--Steenrod theorem for Finsler manifolds of low regularity},
Proc. Amer. Math. Soc. \textbf{145} (2017), no.~6, 2699--2712.
\url{https://doi.org/10.1090/proc/13407}.

\bibitem{McCrimmon}
K.~McCrimmon,
\emph{A Taste of Jordan Algebras},
Universitext, Springer-Verlag, New York, 2004.
\url{https://doi.org/10.1007/b97489}.

\bibitem{Mellon2015}
P. Mellon, \emph{Dynamics of biholomorphic self-maps on bounded
symmetric domains}, Math. Scand. \textbf{117} (2015), 203--216.

\bibitem{MolnarThompsonHilbert}
L.~Moln\'ar,
\emph{Thompson isometries of the space of invertible positive operators},
Proc. Amer. Math. Soc. \textbf{137} (2009), no.~11, 3849--3859.
\url{https://doi.org/10.1090/S0002-9939-09-09963-8}.

\bibitem{MolnarNagy}
L.~Moln\'ar and G.~Nagy,
\emph{Thompson isometries on positive operators: the $2$-dimensional case},
Electron. J. Linear Algebra \textbf{20} (2010), 79--89.
\url{https://doi.org/10.13001/1081-3810.1358}.

\bibitem{Monod2020}
N.~Monod,
\emph{Notes on functions of hyperbolic type},
Bull. Belg. Math. Soc. Simon Stevin \textbf{27} (2020), no.~2, 167--202.
\url{https://arxiv.org/abs/1807.04157}.

\bibitem{Nica2012}
B.~Nica,
\emph{The Mazur--Ulam theorem},
Expo. Math. \textbf{30} (2012), no.~4, 397--398.
\url{https://doi.org/10.1016/j.exmath.2012.08.010}.

\bibitem{RW}
M.~Roelands and M.~Wortel,
\emph{Hilbert isometries and maximal deviation preserving maps on JB-algebras},
Adv. Math. \textbf{352} (2019), 836--861.
\url{https://doi.org/10.1016/j.aim.2019.06.027}.

\bibitem{Ungar2008}
A.~A. Ungar,
\emph{Analytic Hyperbolic Geometry and Albert Einstein's Special Theory of Relativity},
World Scientific, Singapore, 2008.
\url{https://doi.org/10.1142/6625}.

\bibitem{vandeWetering}
J.~van de Wetering,
\emph{Commutativity in Jordan operator algebras},
J. Pure Appl. Algebra \textbf{224} (2020), no.~11, 106407.
\url{https://doi.org/10.1016/j.jpaa.2020.106407}.

\bibitem{Walsh2018}
C.~Walsh,
\emph{Gauge-reversing maps on cones, and Hilbert and Thompson isometries},
Geom. Topol. \textbf{22} (2018), no.~1, 55--104.
\url{https://doi.org/10.2140/gt.2018.22.55}.

\bibitem{WangWang2021}
S.~Wang and Z.~Wang,
\emph{Operator means in JB-algebras},
Rep. Math. Phys. \textbf{88} (2021), no.~3, 383--398.
\url{https://doi.org/10.1016/S0034-4877(21)00087-2}.
Preprint version: \url{https://arxiv.org/abs/2012.13127v2}.

\bibitem{WrightYoungson1978}
J.~D.~Maitland Wright and M.~A.~Youngson,
\emph{On isometries of Jordan algebras},
J. London Math. Soc. (2) \textbf{17} (1978), no.~2, 339--344.
\url{https://doi.org/10.1112/jlms/s2-17.2.339}.

\end{thebibliography}
\end{document}